\documentclass[a4paper,oneside]{amsart}
\usepackage{amssymb}
\usepackage[T1]{fontenc}

\pdfoutput=1

\title{Realising all countable groups as quasi-isometry groups}
\author{Paula Heim, Joseph MacManus, and Lawk Mineh}
\date{First draft: 9 January 2026. This version: 4 February 2026.}

\usepackage[numbers]{natbib}
\usepackage{amsfonts}
\usepackage{amsmath,mathtools}
\usepackage{amsthm}
\usepackage{amssymb}
\usepackage{caption}
\usepackage{thmtools}
\usepackage{stmaryrd}
\usepackage{mathrsfs}  
\usepackage{todonotes}
\usepackage{graphicx}
\usepackage{tikz}
\usepackage{quiver}
\usepackage{tgpagella}
\usepackage[colorlinks]{hyperref} 
\hypersetup{
    linkcolor=teal,
    citecolor=blue,
}
\usepackage{fancyref}
\usepackage{todonotes}
\usepackage{faktor}

\DeclareMathOperator{\QI}{QI}
\DeclareMathOperator{\dHaus}{Haus}
\DeclareMathOperator{\dist}{d}
\DeclareMathOperator{\Aut}{Aut}
\DeclareMathOperator{\Isom}{Isom}

\DeclareMathOperator{\id}{id}

\DeclareMathOperator{\hull}{hull}
\DeclareMathOperator{\lk}{lk}
\DeclareMathOperator{\In}{In}

\newcommand{\R}{\mathbb{R}}
\newcommand{\Z}{\mathbb{Z}}
\newcommand{\N}{\mathbb{N}}

\newcommand{\X}{\mathbf{X}}
\newcommand{\Y}{\mathbf{Y}}
\newcommand{\HH}{\mathbb{H}\mathbf{H}}

\newcommand{\into}{\hookrightarrow}

\DeclarePairedDelimiter\abs{\lvert}{\rvert}
\DeclarePairedDelimiter\norm{\lVert}{\rVert}

\makeatletter
\let\oldabs\abs
\def\abs{\@ifstar{\oldabs}{\oldabs*}}
\let\oldnorm\norm
\def\norm{\@ifstar{\oldnorm}{\oldnorm*}}
\makeatother

\newcommand{\overbar}[1]{\mkern 3mu\overline{\mkern-3mu#1\mkern-1mu}\mkern 1mu}

\makeatletter
\newcommand{\myitem}[1]{%
\item[#1]\protected@edef\@currentlabel{#1}%
}
\makeatother

\newtheorem{theorem}{Theorem}[section]
\newtheorem*{theorem*}{Theorem}
\newtheorem*{question*}{Question}
\newtheorem*{claim*}{Claim}

\newtheorem{proposition}[theorem]{Proposition}
\newtheorem{lemma}[theorem]{Lemma}
\newtheorem{corollary}[theorem]{Corollary}

\theoremstyle{definition}
\newtheorem{definition}[theorem]{Definition}
\newtheorem{example}[theorem]{Example}
\newtheorem{remark}[theorem]{Remark}
\newtheorem{convention}[theorem]{Convention}

\begin{document}

\begin{abstract}
    Given any countable group $G$, we construct uncountably many quasi-isometry classes of proper geodesic metric spaces with quasi-isometry group isomorphic to \(G\).
    Moreover, if the group \(G\) is a hyperbolic group, the spaces we construct are hyperbolic metric spaces.
    
    We make use of a rigidity phenomenon for quasi-isometries exhibited by many symmetric spaces, called strong quasi-isometric rigidity. 
    Our method involves the construction of new examples of strongly quasi-isometrically rigid spaces, arising as graphs of strongly quasi-isometrically rigid rank-one symmetric spaces.
\end{abstract}

\maketitle

\section*{Introduction}

Quasi-isometries are among the principal morphisms in the study of large-scale structure of metric spaces, and their use is fundamental in many areas of geometry and geometric group theory.
Just as one considers the group of isometries of a metric space as a natural group of automorphisms, one can consider the \emph{quasi-isometry group} associated to a metric space.
This is the group of equivalence classes of self-quasi-isometries of a metric space, taken up to finite distance in the supremum metric.
The primary goal of this article is to address the question of which groups may be realised as quasi-isometry groups.

The quasi-isometry groups of many familiar metric spaces are quite wild, and their general theory is poorly understood.
For instance, we have an understanding of \(\QI(\R)\): it is infinite-dimensional \cite{gromov2006rigidity}, left-orderable and virtually a direct product \cite{ye2023group}, centreless \cite{chakraborty2019center}, and contains -- for example -- continuum-rank free groups and Thompson's group \(F\) \cite{sankaran2006homeomorphisms}.
Beyond this, little is known; see \cite{zhao2024group} for a survey on \(\QI(\R)\) and \cite{mitra2019embedding, bhowmik2024combinatorial} for some results on \(\QI(\R^n)\) for \(n > 1\).
The quasi-isometry groups of various solvable groups have also been computed \cite{Farb_Mosher,FEW,Wortman}.
Of course, quasi-isometries of spaces induce isomorphisms on quasi-isometry groups.

Most unbounded geodesic spaces seem to have tremendously large and complicated quasi-isometry groups.
On the other hand, there are somewhat pathological metric spaces with trivial quasi-isometry groups, such as the set of factorials \(\{n! : n \in \N\}\) equipped with the induced metric as a subset of \(\R\), though this space is not even coarsely connected.
To the best of our knowledge there were previously no known examples of any metric spaces with non-trivial finite quasi-isometry groups.
We show that the maximum diversity is in fact achieved among countable groups.

\begin{restatable}{alphtheorem}{main}\label{thm:main-theorem}
Let $G$ be a countable group. Then there exists uncountably many quasi-isometry classes of proper geodesic metric spaces $X$ with 
$
G \cong \QI(X).
$
Moreover, if $G$ is a hyperbolic group, then we may take every such $X$ to be hyperbolic.
\end{restatable}

The above result fits into a long history of so-called \emph{representation problems}: given a class of objects \(\mathcal C\), can one determine which groups \(G\) can be realised as \(\operatorname{Aut}(X)\) for an object \(X \in \mathcal C\)?
An early instance of this question was posed by K\"onig \cite{konig2001theorie}, who asked whether any group can be realised as the automorphism group of a graph.
K\"onig's question was settled for finite groups by Frucht in 1939 \cite{frucht1939herstellung}, with an elegant and simple construction, now appearing in many standard texts on graph theory.
Frucht's result was later extended to all groups independently by de Groot \cite{de1959groups} and Sabidussi \cite{sabidussi1960graphs}, and will play a large role in the proof of Theorem~\ref{thm:main-theorem}.

We give a brief and non-exhaustive account of similar representation problems from geometry. 
Every finite group can be realised as the isometry group of a closed hyperbolic manifold in every dimension: see Greenberg for dimension 2 \cite{greenberg1974maximal}, Kojima for dimension 3 \cite{kojima1988isometry}, and Belolipetsky--Lubotzky for higher dimensions \cite{belolipetsky2005finite}.
In the dimension 2 case, this was extended to all countable groups independently by Winkelmann \cite{winkelmann2001realizing} and Allcock \cite{allcock2006hyperbolic}. 
Remarkably, there also exists a single surface (of infinite type), such that varying the Riemannian metric can produce any countable group as the corresponding isometry group \cite{aougab2021isometry}.
Frigerio and Martelli have also shown that any countable group arises as the mapping class group of a hyperbolic \(3\)-manifold \cite{martelli_frig}.

Our strategy will be to exploit a rigidity phenomenon for quasi-isometries in order to reduce to a setting where we can apply a variant of Frucht's theorem.
We say that a metric space \(X\) is \emph{strongly quasi-isometrically rigid} if the natural map \(\Isom(X) \to \QI(X)\) is a surjection.
That is to say, every quasi-isometry of \(X\) is close to an isometry \(X\).
Strongly quasi-isometrically rigid spaces form one of the few classes of metric spaces for which the quasi-isometry group admits a clear description, provided one understands the isometry group.
Note that essentially all spaces that are known to be strongly quasi-isometrically rigid satisfy a slightly stronger, uniform property: see Definition~\ref{def:strong_qi_rigid}.

Strong quasi-isometric rigidity seems to be an intrinsic property of most symmetric spaces.
Notably, Pansu showed that the octonionic plane \(\mathbb{O}\mathbf{H}^2\) and the quaternionic hyperbolic spaces \(\HH^n\) are strongly quasi-isometrically rigid, as long as \(n \geq 2\) \cite{pansu1989metriques}.
Kleiner and Leeb further showed that irreducible higher-rank symmetric spaces of noncompact type and thick irreducible Euclidean buildings with cocompact affine Weyl group are also strongly quasi-isometrically rigid \cite{kleiner1997rigidity}, and Bourdon--Pajot and Xie extended this to some hyperbolic buildings \cite{bourdon2000rigidity,xie2006quasi}.
When \(n \geq 3\), the Cayley graphs of maximal non-arithmetic non-uniform lattices in \(n\)-dimensional rank one symmetric spaces and universal covers of hyperbolic \(n\)-manifolds with non-empty totally geodesic boundary also enjoy this property \cite{schwartz}.
We also mention a beautiful, but somewhat exceptional, example of a strongly quasi-isometrically rigid space due to Kleiner and Kapovich, that is the Cayley graph of a carefully constructed hyperbolic group \cite{KK}.

As an essential step in the proof of the main theorem above, we construct new examples of strongly quasi-isometrically rigid spaces.
These arise from the following theorem, which states that well-behaved graphs of uniformly strongly quasi-isometrically rigid spaces are themselves quasi-isometrically rigid.
We can view this as a sort of combination theorem for strong quasi-isometric rigidity.

\begin{restatable}{alphtheorem}{strong-qi-rigidity}\label{thm:strong-qi-rigidity-theorem}
    Let \(\X\) be a graph of spaces. Suppose that \(\X\) is uniformly hyperbolic, link bottlenecked, link rigid, locally congruent, and has unbounded edge spaces. If the vertex spaces of \(\X\) are uniformly strongly quasi-isometrically rigid, then so is the realisation of \(\X\).
\end{restatable}

We direct the reader to Section~\ref{sec:gos} for a precise definition of graph of spaces in this context, and to Sections~\ref{sec:uniformly-hyperbolic-gos} and~\ref{sec:bottlenecks} for definitions of the terms appearing in the theorem statement.
We give a brief overview of this list of adjectives here.
A graph of spaces is said to be \emph{uniformly hyperbolic} if the vertex spaces are all hyperbolic with the same hyperbolicity constant, and the edge spaces in some sense diverge uniformly quickly from one another in each vertex space.
\emph{Link bottlenecked} essentially means that the edge spaces do not coarsely separate the vertex spaces, and \emph{link rigid} means that any isometry of a vertex space that coarsely fixes its incident edge spaces is trivial.
Finally, \emph{local congruence} means that, for any two vertices in the same orbit of an automorphism of the underlying graph, the corresponding vertex spaces and their configuration of incident edge spaces are identical.
The properties of being uniformly hyperbolic and link bottlenecked together imply that quasi-isometries of the graph of spaces coarsely respect the graph structure, so that we may use strong quasi-isometric rigidity of the vertex spaces to promote quasi-isometries to piecewise isometries of the whole space.
Link rigidity and local congruence then ensure that the isometry group is both rich enough and well-behaved enough that we can coherently glue these piecewise isometries together into a global isometry.

This paper is structured as follows. 
Section~\ref{sec:prelims} contains a collection of preliminary material. 
In Section~\ref{sec:frucht}, we discuss a strengthening of the aforementioned theorem of Frucht. 
In Section~\ref{sec:HH2}, we construct a sets of points in the boundary of a visibility manifold with a certain rigidity property. 
The next three sections are devoted to graphs of metric spaces: Section~\ref{sec:gos} contains definitions and basic properties, Section~\ref{sec:uniformly-hyperbolic-gos} introduces uniformly hyperbolic graphs of spaces and explores their properties, and in Section~\ref{sec:bottlenecks} we obtain control over the quasi-isometry groups of graphs of spaces to obtain Theorem~\ref{thm:strong-qi-rigidity-theorem}. 
Section~\ref{sec:construction} includes an explicit construction of a graph of spaces to which Theorem~\ref{thm:strong-qi-rigidity-theorem} applies, and the proof of Theorem~\ref{thm:main-theorem}.
Finally, Section~\ref{sec:hyperbolicity} is dedicated to understanding geodesics in uniformly hyperbolic graphs of spaces, which is essential for the hyperbolicity claim of Theorem~\ref{thm:main-theorem}.

\subsection*{Acknowledgements}

The second author was supported by the Additional Funding Programme for Mathematical Sciences, delivered by EPSRC (EP/V521917/1) and the Heilbronn Institute for Mathematical Research. The third author was supported by the European Research Council (SATURN, 101076148) and the Deutsche Forschungsgemeinschaft  (EXC-2047/1, 390685813). 

The authors would like to thank the Isaac Newton Institute for Mathematical Sciences, Cambridge, for support and hospitality during the programme \textit{Non-positive curvature and applications}, where much of the work on this paper was undertaken. This programme was supported by EPSRC grant EP/Z000580/1.

\section{Preliminaries}\label{sec:prelims}

We begin by standardising our notation and terminology.

\subsection{Metric notions}

Let $(X,\dist_X)$ be a metric space. 
Write $\Isom(X)$ for the group of isometries $f \colon X \to X$. 
Given $r \geq 0$, $x \in X$, we write $B_X(x;r)$ for the closed ball of radius $r$ centred at $x$. 
Given $A \subset X$, we define $B_X(A;r)$ as the union of all balls \(B_X(x;r)\) with \(x \in A\). 
If $A, B \subset X$, we write $\dist_X(A,B)$ to mean the infimal distance between $A$ and $B$, i.e. 
\[
    \dist_X(A, B) = \inf_{a \in A, \  b \in B} \dist_X(a,b). 
\]
If $A = \{x\}$ is a singleton, we abuse notation and write $ \dist_X(x,B) := \dist_X(\{x\},B)$. 
We denote by $\dHaus_X(A,B)$ the \emph{Hausdorff distance} in \(X\) between $A$ and $B$. That is,
\[
    \dHaus_X(A,B) = \inf \{r > 0 : \text{$A \subset B_X(B;r)$ and $B \subset B_X(A;r)$}\}. 
\]
For any set \(Y\), denote the supremum metric on functions $f, g \colon Y \to X$ as
\[
    \dist_\infty(f, g) = \sup_{y \in Y} \dist_X(f(y), g(y)), 
\]
noting that this metric can take infinite values in general.

\begin{definition}[Quasi-isometry]
    Let $\lambda \geq 1$ and \(c \geq 0\). A map $\varphi \colon X \to Y$ between metric spaces is called a \emph{\((\lambda,c)\)-quasi-isometric embedding} if
    $$
        \frac 1 \lambda \dist_X(x,y) - c \leq \dist_Y(\varphi(x), \varphi(y)) \leq \lambda \dist_X(x,y) + c
    $$
    for all $x, y \in X$. 
    If there is \(K \geq 0\) such that for all $y \in Y$ there exists $x \in X$ such that $\dist_Y(y, \varphi(x)) \leq K$, we say that $\varphi$ is a \emph{\(K\)-coarsely surjective}.
    
    A \emph{\((\lambda,c)\)-quasi-isometry} is a \((\lambda,c)\)-quasi-isometry embedding that is \(c\)-coarsely surjective.
    A map \(\varphi \colon X \to Y\) is simply called a \emph{quasi-isometry} if it is a \((\lambda,c)\)-quasi-isometry for some \(\lambda \geq 1\) and \(c \geq 0\).
    If $\varphi \colon X \to Y$, $\psi : Y \to X$ are quasi-isometries, we call $\psi$ a \emph{quasi-inverse} to $\varphi$ if there is some \(Q\geq 0\) with $\dist_\infty(\psi \circ \varphi, \id_X) \leq Q$ and $\dist_\infty(\varphi \circ \psi, \id_X) \leq Q$.
\end{definition}

It is a standard fact that for every $(\lambda,c)$-quasi-isometry $\varphi \colon X \to Y$, there exists another $(\lambda,c')$-quasi-isometry $\psi \colon Y \to X$ with \(\dist_\infty(\psi \circ \varphi, \id_X) \leq Q\), where $Q$ and $c'$ depend only on $\lambda$ and $c$. We call any such function a \emph{quasi-inverse} to \(\varphi\).

\begin{definition}[Quasi-geodesics]
    Let \(I \subseteq \R\) be a closed interval, bounded or unbounded, and \(X\) an arbitrary metric space.
    A \emph{\((\lambda,c)\)-quasi-geodesic} in \(X\) is a \((\lambda,c)\)-quasi-isometric embedding \(\varphi \colon I \to X\).
    Similarly to quasi-isometries, we sometimes omit mention of the particular constants.
    We may often abuse notation and identify a quasi-geodesic with its image in \(X\).
    Further, we will call such \(\varphi\) a \emph{\((\lambda,c)\)-quasi-geodesic segment, ray} or \emph{line} when \(I\) is a bounded interval, a half-bounded interval, and the entire real line respectively. 
    
    As a special case, an \emph{geodesic} in \(X\) will for us be an isometric embedding \(I \to X\) (i.e. a \((1,0)\)-quasi-geodesic).
    Given two points \(x, y \in X\), we will often denote by \([x,y]\) a choice of geodesic in \(X\) with endpoints \(x\) and \(y\).
\end{definition}

We define the equivalence relation on quasi-isometries \(\varphi, \psi\) by setting \(\varphi \sim \psi\) if and only if $\dist_\infty(\varphi, \psi)$ is finite. 
Given metric space $X$ and \(Y\), let
$$
\QI(X,Y) = \faktor{\{\text{quasi-isometries $\varphi \colon X \to Y$}\}}{\sim}. 
$$
It is straightforward to check that composition of maps gives rise to a well-defined group operation on this set.
We write \(\QI(X)\) as shorthand for \(\QI(X,X)\), and we call this the \emph{quasi-isometry group of $X$}. 
Of course, every isometry of \(X\) is a quasi-isometry, so there is a map from the isometry group of \(X\) to its quasi-isometry group.
In general, this map need neither be injective nor surjective. 
The property of surjectivity is of particular interest, as it means that the quasi-isometries of \(X\) behave in a very rigid way.

\begin{definition}[Strongly QI-rigid space]
\label{def:strong_qi_rigid}
    We say that a metric space $X$ is \emph{strongly QI-rigid} if the natural map \(\Isom(X) \to \QI(X)\) is a surjection.

    We say that \(X\) is \emph{uniformly strongly QI-rigid} if for all \(\lambda \geq 1, c \geq 0\), there is \(\mu \geq 0\) such that for any \((\lambda,c)\)-quasi-isometry \(\varphi \colon X \to X\), there is an isometry \(f \in \Isom(X)\) with \(\dist_{\infty}(\varphi,f) \leq \mu\).
\end{definition}

We recall the notion of coarse connectedness for a metric space.

\begin{definition}[Coarse connectedness]
    Let \((X,\dist)\) be a metric space, \(\kappa > 0\) a constant, and \(x, y \in X\).
    A \emph{\(\kappa\)-path} between \(x\) and \(y\) is a sequence of points \((x_0, \dots, x_n)\) in \(X\) for which \(x_0 = x, x_n = y,\) and \(\dist(x_{i-1},x_i) \leq \kappa\) for every \(i = 1, \dots, n\).
    We say that \(X\) is \emph{\(\kappa\)-coarsely connected} if there is a \(\kappa\)-path between any two points of \(X\).
\end{definition}

The following lemma is a well-known observation about coarse separation of subsets (cf. \cite[Lemma 2.3]{bensaid2024coarseseparationlargescalegeometry}).
Important for us here is the uniformity of the constants involved: we include a proof to account for these details.

\begin{lemma}
\label{lem:effective_coarse_separation}
    For every \(\lambda \geq 1\) and \(c \geq 0\) there is a constant  \(\kappa_0=\kappa_0(\lambda, c) \geq 1\) such that  for any \(\kappa \geq \kappa_0\) 
    there are $\rho=\rho(\lambda, c, \kappa)$ with $\underset{\kappa\to\infty}{\lim}\rho(\lambda, c, \kappa)=\infty$ and \(L = L(\kappa) \geq \rho\)
    such that the following is true.
    
    Let \(\varphi \colon X \to X'\) be a \((\lambda,c)\)-quasi-isometry of metric spaces.
    Suppose that \(Y \subseteq X\) is \(\lambda\)-coarsely connected, \(a, b \in Y\), and \(Z \subseteq X\) are such that \(a\) and \(b\) lie in different \(\kappa\)-coarsely connected components of \(Y - B_{X}(Z;\kappa)\) and 
    \[\dist_{X}(a,Z) > \lambda L + \lambda c  \quad \mathrm{ and } \quad \dist_{X}(b,Z) > \lambda L + \lambda c.\]
    Then \(\varphi(a)\) and \(\varphi(b)\) lie in different \(\rho\)-coarsely connected components of the set \(\varphi(Y) - B_{X'}(\varphi(Z);L)\).
\end{lemma}

\begin{proof}
    Let \(\psi \colon X' \to X\) be a quasi-inverse of \(\varphi\), so that there is \(Q \geq 0\) depending only on the quasi-isometry constants of \(\varphi\) with \(\dist_\infty(\psi \circ \varphi,\id_X) \leq Q\).
    We may suppose that \(\psi\) is also a \((\lambda,c)\)-quasi-isometry. 
    Write \(\rho_0 = \lambda^2 + c\), and note that \(\varphi(Y)\) is \(\rho_0\)-coarsely connected.
    We further write \(\rho'_0 = \lambda \rho_0 + c\).
    Define \(\kappa_0 = \rho'_0 + Q\), let $\kappa\geq \kappa_0$ and set $\rho = \frac{1}{\lambda} (\kappa-c-Q)$, $\rho^\prime = \lambda\rho+c$ and \(L = \lambda \kappa + c + Q\).
    It is clear that $\rho$ diverges as $\kappa$ tends to infinity.
    
    Let \(p = (x_0, \dots, x_n)\) be a \(\rho\)-path in \(\varphi(Y)\) whose endpoints are \(\varphi(a)\) and \(\varphi(b)\), which exists by $\rho_0$-connectedness of~$Y$.
    Then \(\psi(p) = (\psi(x_0), \dots, \psi(x_n))\) is a \(\rho'\)-path whose endpoints are \(\psi(\varphi(a))\) and \(\psi(\varphi(b))\).
    Since \(\psi\) is quasi-inverse to \(\varphi\), we obtain a \(\kappa\)-path \((a,\psi(x_0), \dots, \psi(x_n),b)\) from \(a\) and \(b\) in \(Y\).
    By assumption, \(a\) and \(b\) lie in different \(\kappa\)-coarse components of \(Y - B_X(Z;\kappa)\), so this path contains a point \(z \in Y\) with \(\dist_X(z,Z) \leq \kappa\).
    The assumptions imply that \(a\) and \(b\) are at a distance of greater than \(\kappa\) from \(Z\), so we must have \(z = \psi(x_i)\) for some \(i = 0, \dots, n\).
    But then by the fact that \(\varphi\) is a \((\lambda,c)\)-quasi-isometry and the choice of \(L\),
    \[
        \dist_{X'}(x_i,\varphi(Z)) \leq \lambda \kappa + c + Q = L.
    \]
    This means that any \(\rho\)-path in \(Y\) from \(\varphi(a)\) to \(\varphi(b)\) passes \(L\)-close to \(\varphi(Z)\).
    Since \(a\) and \(b\) are a distance greater than \(\lambda L+ \lambda c\) from \(\varphi(Z)\), the points \(\varphi(a)\) and \(\varphi(b)\) are a distance greater than \(L\) from \(\varphi(Z)\).
    Therefore \(\varphi(a)\) and \(\varphi(b)\) lie in distinct \(\rho\)-coarsely connected components of \(Y - B_{X'}(\varphi(Z);L)\) as required.
\end{proof}

\subsection{Hyperbolic metric spaces}

We here recall the definition of a hyperbolic metric space, along with some useful facts.

\begin{definition}[Gromov product]
    Let \((X,\dist)\) be a metric space, and \(x, y, z \in X\) be points.
    The \emph{Gromov product} of \(x\) and \(y\) with respect to \(z\) is
    \[
        \langle x, y \rangle_z = \frac{1}{2}\Big(\dist(x,z) + \dist(y,z) - \dist(x,y)\Big).
    \]
    If \(A, B \subseteq X\) are subspaces, we write
    \[
        \langle A, B \rangle_z = \sup\{\langle a, b \rangle_z \, | \, a \in A, b \in B \} \in [0,\infty] 
    \]
\end{definition}

\begin{definition}[Thin triangles]
    Let \(\Delta\) be a geodesic triangle with vertices \(x, y\), and \(z\) in a metric space \(X\), and let \(\delta \geq 0\).
    Call \(T_\Delta\) the tripod with leg lengths \(\langle x, y \rangle_z, \langle x,z \rangle_y,\) and \(\langle y,z \rangle_z\).
    There is a unique map \(\varphi \colon \Delta \to T_\Delta\) such that \(x, y\), and \(z\) map to the extremal vertices of \(T_\Delta\) and \(\varphi\) restricts to an isometry on each side of \(\Delta\).
    We say \(\Delta\) is \emph{\(\delta\)-thin} if \(\operatorname{diam}\varphi^{-1}(\{t\}) \leq \delta\) for all \(t \in T_\Delta\).
\end{definition}

\begin{definition}[Hyperbolic metric space]
    Let \(X\) be a geodesic metric space.
    If there is \(\delta \geq 0\) such that every geodesic triangle in \(X\) is \(\delta\)-thin, we say that \(X\) is a \emph{\(\delta\)-hyperbolic metric space}. 
\end{definition}

\begin{remark}
\label{rem:four_point}
    It is well-known that \(\delta\)-hyperbolicity of a space \(X\) implies the following four point condition: for all \(x, y, z, w \in X\)
    \[
        \min\{ \langle x, y \rangle_w, \langle y,z \rangle_w \} \leq \langle x,z \rangle_w + 2\delta.
    \]
\end{remark}

\begin{remark}
\label{rem:hyp_constant_cat-1}
    It is well known that the real hyperbolic plane \(\R \mathbf{H}^2\) is \(\delta\)-hyperbolic with \(\delta = \log(3)\).
    Any \(\operatorname{CAT}(-1)\) space is therefore also \(\log(3)\)-hyperbolic.
\end{remark}

The following lemma is essentially a restating of the thin triangles property in a form that is convenient for our later use.

\begin{lemma}
\label{lem:basically_thin_triangles}
    Let \((X,\dist)\) be a \(\delta\)-hyperbolic metric space, \(x, y, z \in X\).
    Any geodesic \([y,z]\) contains a point \(t\) with \(\dist_X(t,x) \leq \langle y,z \rangle_x + \delta\).
\end{lemma}

\begin{proof}
    Let \(\varphi \colon \Delta \to T_\Delta\) be the canonical map from the triangle \(\Delta\) with vertices \(x, y, z\) to the associated tripod \(T_\Delta\).
    Call \(\ast \in T_\Delta\) the middle point of the tripod.
    Take \(s \in [x,y]\) and \(t \in [y,z]\) to be the unique points of \([x,y]\) and \([y,z]\) in the preimage \(\varphi^{-1}(\{\ast\})\) respectively.
    By definition, \(\dist(x,s) = \langle y,z \rangle_x\) and, since \(\Delta\) is \(\delta\)-thin, \(\dist(s,t) \leq \delta\).
    The lemma follows now by applying the triangle inequality.
\end{proof}

We will use the following statement about closest point projections.

\begin{lemma}
\label{lem:triangle_with_projection}
    Let \((X,\dist)\) be a \(\delta\)-hyperbolic metric space, \(Y \subseteq X\) a subspace, \(\pi \colon X \to Y\) a closest point projection.
    For any \(x \in X\) and \(y \in Y\) we have
    \[
        \langle x,y \rangle_{\pi(x)} \leq \delta
    \]
    and
    \[
        \dHaus_X([x,y], [x,\pi(x)] \cup [\pi(x),y]) \leq 2\delta.
    \]
\end{lemma}

\begin{proof}
    By hyperbolicity of $X$ there exist $a\in [x,\pi(x)]$ and $b \in [\pi(x),y]$ with $\dist(a, b)\leq \delta$ and \(\dist(a,\pi(x)) = \dist(b,\pi(x)) = \langle x,y\rangle_{\pi(x)}\). 
    By the fact that~$\pi$ is a closest point projection and \(a\) lies on a geodesic between \(x\) and \(\pi(x)\), we have \[\langle x ,y \rangle_{\pi(x)} = \dist(b,\pi(x)) = \dist(a, \pi(x))\leq \dist(a,b) \leq \delta.\]
    By the thinness of triangles, we also have that $\dHaus_X([x, y], [x, a] \cup [b, y])\leq \delta$, from which it follows that $\dHaus_{X}([x, y], [x, \pi(x)]\cup [\pi(x), y])\leq 2\delta$.
\end{proof}

A hyperbolic metric space admits a natural bordification \(\partial X\) called the \emph{Gromov boundary}.
Among proper spaces, this coincides with the \emph{geodesic boundary}, defined as equivalence classes of geodesic rays considered up to bounded Hausdorff distance; see \cite[Chapter 11]{drutukapovich} for details.
The Gromov product on a proper hyperbolic metric space extends naturally to its Gromov boundary.
Given a (bi-)infinite geodesic \(\gamma \colon I \to X\) in a proper hyperbolic space \(X\), its \emph{endpoints} are the points each of its subrays represents in \(\partial X\).

A core feature of hyperbolic metric spaces is the stability of quasi-geodesics.
This is encapsulated in the following, commonly known as the Morse Lemma.

\begin{lemma}[{\cite[Theorem 11.72 and Lemma~11.105]{drutukapovich}}]
\label{lem:morse_lemma}
    Let \((X,\dist)\) be a proper \(\delta\)-hyperbolic metric space, \(\lambda \geq 1\) and \(c \geq 0\).
    There is \(M = M(\lambda,c,\delta) \geq 0\) such that for any \((\lambda,c)\)-quasi-geodesic \(\varphi \colon I \to X\), any geodesic \(\gamma \colon J \to X\) with the same endpoints as \(\varphi\) satisfies
    \[
        \dHaus_X(\varphi(I),\gamma(J)) \leq M.
    \]
    Here, endpoints can be points in~$X$ or \(\partial X\).
\end{lemma}

For $Y\subseteq X$, where $X$ is a $\delta$-hyperbolic metric space, we denote by $\Lambda Y$ the \emph{limit set} of~$Y$. We will need the following lemma about the limit sets of coarsely separating sets in hyperbolic spaces. 

\begin{lemma}\label{lem:limit_sets_of_coarse_components_are_separated}
    Let $X$ be a $\delta$-hyperbolic metric space, let $Z \subseteq X$ be a subset and let $U$ and $V$ be different coarse components of $X-Z$ for some $\rho\geq 0$. Then $\Lambda U \cap \Lambda V \subseteq \Lambda Z$. 
\end{lemma}

\begin{proof}
    Let \(\rho \geq 0\) be a constant such that \(U\) and \(V\) are in different \(\rho\)-coarse components of \(X - Z\).
    Let $z \in \partial U \cap \partial V$, so there are sequences $(u_n) \subseteq U$ and $(v_n) \subseteq V$ with \(u_n, v_n \to z\) as \(n \to \infty\).
    Since $U$ and $V$ are different $\rho$-coarse components of \(X - Z\), for every $n\in\N$ there is $z_n\in [u_n, v_n] \cap B(Z;\rho)$.
    But then the sequence of $(z_n)$ also converges to $z$, so that \(z \in \partial Z\).
\end{proof}


\section{Frucht's theorem}\label{sec:frucht}

In this section we will discuss Frucht's theorem. 
We will need a particular variant for countable groups which produces uncountably many regular graphs of arbitrary degree, and the main goal of this section is to prove this. 
This can essentially be pieced together from proofs already present in the literature, but we include our own for the sake of self-containment.
We first introduce some terminology. 

\begin{definition}[Vertex-free and edge-free graphs]
    A graph $\Gamma$ is said to be \emph{vertex-free} if the stabiliser of every vertex in $\Aut(\Gamma)$ is trivial. Similarly, $\Gamma$ is said to be \emph{edge-free} if the setwise stabiliser of every pair $\{u,v\}$ of adjacent vertices is trivial.
\end{definition}

We will mostly be interested in vertex-free graphs, but it is possible for us to ensure our constructed graphs are also edge-free. 
A vertex-free graph $\Gamma$ is edge-free if and only if $\Aut(\Gamma)$ acts \emph{without edge inversions}. 
That is, no isometry of $\Gamma$ transposes the two endpoints of any edge.

\begin{definition}[Labelled graphs]
    Let $\Gamma$ be a graph.
    Given $e \in E(\Gamma)$ with endpoints $x, y \in V(\Gamma)$, an \emph{orientation} of $e$ is a choice of endpoint. 
    A \emph{partial orientation} of $\Gamma$ is a choice of orientation for a subset of the edges. Edges without a choice of orientation are called \emph{undirected}, and those with an orientation as \emph{directed}. 
    
    Let $S$ be a set. An \emph{$S$-labelling} of $\Gamma$ is a partial orientation of $\Gamma$ together with an assignment of an element of $S$ to every edge. 
    We call \(S\) the labelling set. 
    A graph with an $S$-labelling is called a \emph{labelled graph}.

    We define the \emph{label-preserving automorphism group}, denoted $\Aut_{\ell}(\Gamma)$, as the subgroup of $\Aut(\Gamma)$ which preserves the labelling of $\Gamma$. 
\end{definition}

Cayley graphs are the prototypical examples of labelled graphs.

\begin{definition}[Cayley graph]\label{def:cay graph}
    Let $G$ be a group, and $S \subset G - \{1\}$ a generating set. 
    Assume further that if \(s, s^{-1} \in S\), then \(s\) has order 2. 
    The \emph{Cayley graph} of $G$ with respect to $S$, denoted $\Gamma_{G,S}$, is the simplicial, labelled graph with:
    \begin{enumerate}
        \item $V(\Gamma_{G,S}) = G$;

        \item for every $g \in G$, $s \in S$ not of order 2, there is an edge between $g$ and $gs$, labelled by $s$, directed from $g$ to $gs$;

        \item for every $g \in G$, $s \in S$ of order 2, there is an undirected edge between $g$ and $gs$ labelled by $s$, undirected. 
    \end{enumerate}
\end{definition}

\begin{remark}
\label{rem:auts_of_labelled_cayley_graph}    
    The group of labelled automorphisms of a Cayley graph is exactly the original group.
    This is straightforward to see by the action of the group on its Cayley graph by (left) translation.
\end{remark}

The following is the central construction of the section.
It associates to each 3-regular labelled graph another 3-regular graph, whose automorphism group is the same as the group of label-preserving automorphisms of the original graph.

\begin{proposition}
\label{prop:labelled_aut->aut}
    Let $\Gamma_0$ be a 3-regular vertex-free labelled graph. 
    Then there exist uncountably many connected, simplicial, 3-regular, vertex-free, edge-free graphs $\Gamma$ with \[\Aut(\Gamma) \cong \Aut_\ell(\Gamma_0).\]
    Moreover, if \(\Gamma_0\) is hyperbolic, then \(\Gamma\) is also hyperbolic.
\end{proposition}

\begin{proof}
    Consider the set $\{0,1\}^{\N}$ of one-way infinite binary sequences. 
    for each sequence $\sigma \in \{0,1\}^{\N}$, we construct a \emph{tag} graph, denoted $T_\sigma$, as depicted in Figure~\ref{fig:tag-construction}. 
    The construction is a sort of ladder graph with two types of rungs, whose decoration is determined by the binary sequence \(\sigma\). 
    The rungs corresponding to \(1\)s will be single edges, while those corresponding to \(0\)s will be a copy of the \emph{Frucht graph}, which is a 3-regular graph with trivial automorphism group \cite[Theorem 2.3]{frucht1949graphs}. 

    All tags are 3-regular except for a single leaf. 
    As such, any automorphism of a tag graph \(T_\sigma\) must preserve the single leaf and so also the sequence of rungs in the tag.
    Since the Frucht graph has no non-trivial automorphisms, this implies that \(T_\sigma\) has trivial automorphism group unless \(\sigma\) is the constant sequence of \(1\)s. 
    Finally, note that by the same reasoning, $T_\sigma \cong T_{\sigma'}$ implies that $\sigma = \sigma'$. 

    \begin{figure}[ht]
        \centering \input{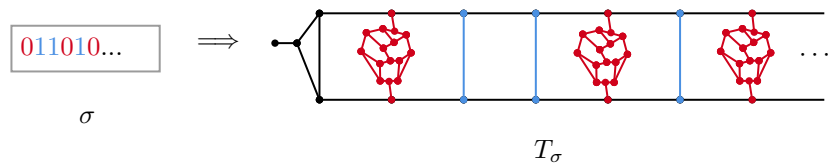}
        \caption{Constructing the tag $T_\sigma$.}
        \label{fig:tag-construction}
    \end{figure}

    Consider the set $L$ of labels appearing in $\Gamma_0$, which is countable. 
    Fix some injection $h \colon L \into \{0,1\}^{\N}$ such that no $\ell \in L$ maps to the constant sequence of \(1\)s. 

    We now blow up $\Gamma_0$ to another graph $\Gamma$ as follows. 
    Given an edge $e \in E(\Gamma_0)$ with directed label $\ell \in L$, we replace this edge with the gadget depicted in Figure~\ref{fig:directed-case}. 
    Given an edge labelled by an undirected label replace this edge with the gadget depicted in Figure~\ref{fig:undirected-case}. 
    
    Now any automorphism \(\varphi\) of \(\Gamma\) must send tags of the form \(T_{h(a)}\) to tags of the same form.
    As such, \(\varphi\) preserves the gadgets of \(\Gamma\).
    Moreover, the only automorphism each gadget allows is an involutive automorphism, exactly when the label is undirected.
    Therefore \(\varphi\) gives rise to a labelled graph automorphism of \(\Gamma_0\).
    Conversely, each label preserving automorphism of \(\Gamma_0\) trivially defines an automorphism of \(\Gamma\).

    As noted, each gadget has no non-trivial automorphisms beside the obvious involution of gadgets corresponding to undirected labels, and these automorphisms fix no vertex or edges.
    Therefore any vertex of \(\Gamma\) fixed by an automorphism must correspond to a vertex of \(\Gamma_0\) fixed by a label-preserving automorphism.
    Since \(\Gamma_0\) was vertex free, though, there are no such vertices, so \(\Gamma\) is vertex-free.
    Further, \(\Gamma\) is edge-free, since the only label-preserving automorphisms of \(\Gamma_0\) invert an edge with an undirected label.
    
    Recall now that a \emph{block} is a maximal connected subgraph with no disconnecting vertices.
    Every graph admits a canonical decomposition into a tree of blocks (see, for example, \cite[\S3.1]{diestel2025graph}).
    It is straightforward consequence that a graph is \(\delta\)-hyperbolic if and only if each of its blocks are \(\delta\)-hyperbolic.
    Each block of the gadgets is at most \(5\)-hyperbolic, so \(\Gamma\) is hyperbolic whenever \(\Gamma_0\) is.

    \begin{figure}[htbp]
        \begin{minipage}[t]{.5\textwidth}
            \centering \input{figures/edge-gadget-directed}
            \captionof{figure}{}
            \label{fig:directed-case}
        \end{minipage}
        \begin{minipage}[t]{.5\textwidth}
            \centering \input{figures/edge-gadget-undirected}
            \captionof{figure}{}
            \label{fig:undirected-case}
        \end{minipage}
    \end{figure}

    Finally, note that if we chose a different labelling $h' :  L \into \{0,1\}^{\N}$ with image disjoint to that of $h$, then we would necessarily get a different graph. Since there are uncountably many such injections with pairwise disjoint images, we are done.
\end{proof}

An additional blow-up construction allows us to go from 3-regular graphs to \(d\)-regular graphs for any \(d > 3\).
We are able to do this blow up in such a way to preserve the large-scale geometric structure of the graph.

\begin{lemma}
\label{lem:3-reg->d-reg}
    Let $\Gamma$ be a 3-regular connected, simplicial, vertex-free, edge-free graph, $d > 3$. 
    Then there exists a $d$-regular, connected, simplicial, edge-free graph $\Gamma^{(d)}$ such that $\Aut(\Gamma) \cong \Aut(\Gamma^{(d)})$, and $\Gamma$ is quasi-isometric to $\Gamma^{(d)}$.

    Moreover, if $\Delta$ is another 3-regular connected, simplicial, edge-free graph such that $\Gamma^{(d)}\cong\Delta^{(d)}$, then $\Gamma \cong \Delta$.  
\end{lemma}

\begin{proof}
    Let $\Gamma$ be a 3-regular graph, and $d > 3$. We will construct a new $d$-regular augmented graph $\Gamma^{(d)}$ as follows. 

    Let $A_1, A_2, \ldots, A_{d-3}$ be pairwise non-isomorphic $(d-1)$-regular graphs of order a power of ${d-1}$, and with trivial automorphism group.  Such graphs are known to be generic among regular graphs by probabilistic arguments \cite[Corollary 1.2]{kim2002asymmetry}.
    Any \(d\)-clique of a \((d-1)\)-regular graph is a complete connected component, which has non-trivial automorphism group.
    Hence no \(A_i\) contains a \(d\)-clique.
    Consider now a copy of the complete graph $K_d$ on $d$ vertices, and choose three distinguished vertices $u_1, u_2, u_3$.
    We label the other vertices \(v_1, \dots, v_{d-3}\).
    
    For each $i = 1, \ldots, d-3$, let $n_i > 0$ be such that $A_i$ has order $(d-1)^{n_i}$.
    We define the graph $A_i'$  as the full rooted $(d-1)$-ary tree of depth $n_i$, adding edges to the ${(d-1)}^{n_i}$ leaves so that the subgraph of $A_i'$ induced by the leaves is precisely $A_i$. 
    The root vertex $w_i \in V(A_i')$ is the only vertex of \(A_i'\) with degree $d-1$, and every other vertex has degree $d$.
    Note that since \(A_1,\dots,A_n\) were chosen pairwise non-isomorphic, \(A_1',\dots,A_n'\) are pairwise non-isomorphic. 
    We construct a new graph \(P\) as the union of \(K_d\) and \(A_1', \dots, A_{d-3}'\), with an extra edge between \(v_i\) and the root \(w_i\) of \(A_i'\) for each \(i = 1, \dots, d-3\).
    Now every vertex of \(P\) has degree \(d\), except for \(u_1, u_2,\) and \(u_3\), which have degree $d-1$.

    We claim that $\Aut(P)$ acts 3-transitively on the vertices \(u_1,u_2,\) and \(u_3\), and that every other vertex of $P$ is fixed by every automorphism. 
    We first note that $\Aut(P)$ setwise fixes \(\{u_1,u_2,u_3\}\), as these are the only vertices of degree other than $d$. 
    Moreover, since \(\Aut(K_d) = \operatorname{Sym}(V(K_d))\) acts \(d\)-transitively on its vertex set, $\Aut(P)$ acts 3-transitively on \(\{u_1,u_2,u_3\}\), by permuting these vertices and their incident edges, and acting trivially elsewhere. 
    
    We must show that every other vertex in $P$ is fixed by every automorphism. 
    Let $\varphi \in \Aut(P)$.
    Since none of the \(A_i\) contain \(d\)-cliques, the base copy of \(K_d\) is the only $d$-clique in $P$.
    This copy of \(K_d\) must therefore be fixed by $\varphi$. 
    Since \(\varphi\) permutes \(u_1,u_2,\) and \(u_3\}\), it must also permute the remaining vertices \(v_1, \dots, v_{d-3}\) of \(K_d\). 
    Therefore, \(\varphi\) must permute the copies of \(A_i'\) in \(P\).
    However, the graphs \(A_1', \dots, A_n'\) are pairwise non-isomorphic, so this induced permutation must be trivial. 
    Thus \(\varphi\) restricts to a root-fixing automorphism on \(A_i'\) for each \(i = 1,\dots,n\).
    Any non=trivial root-fixing automorphism of \(A_i'\) restricts to a non-trivial automorphism of the copy of \(A_i\) in \(A_i'\).
    Since we chose \(A_i\) to be a graph with no non-trivial automorphisms, \(\varphi\) must restrict to a trivial automorphism on \(A_i'\).
    Therefore \(\varphi\) fixes every vertex of \(P\) except \(u_1, u_2,\) and \(u_3\).

    For each vertex \(v \in V(\Gamma)\), choose a bijection \(\varphi_v\) between \(\{u_1, u_2, u_3\}\) and the 3 edges incident to \(v\).
    We now form $\Gamma^{(d)}$ by taking a copy of a copy $P_v$ of $P$ for every vertex \(v \in V(\Gamma)\), and joining the copy of \(u_i\) in \(P_v\) to the copy of \(u_j\) in \(P_w\) whenever \(v\) and \(w\) are adjacent and \(\varphi_v(u_i) = \varphi_w(u_j)\). 
    The resulting graph is $d$-regular, and the map collapsing each copy of \(P\) to a single vertex is a quasi-isometry to $\Gamma$, as \(P\) is finite. 
    
    To conclude the proof, consider the \emph{graph of $d$-cliques} of \(\Gamma^{(d)}\), denoted $C_d(\Gamma^{(d)})$.
    The vertices of \(C_d(\Gamma^{(d)})\) are the \(d\)-cliques of \(\Gamma^{(d)}\), and two vertices are joined by an edge if their corresponding \(d\)-cliques contain adjacent vertices.
    Since automorphisms of graphs permute its \(d\)-cliques, there is always a homomorphism 
    \begin{equation}
    \label{eq:aut_clique_graph}
        F \colon \Aut(\Gamma^{(d)}) \to \Aut(C_d(\Gamma^{(d)})).
    \end{equation}
    
    If \(\varphi \in \ker F\), then \(\varphi\) fixes every clique of \(\Gamma^{(d)}\) and the adjacency relations between them.
    This means that it restricts to an automorphism on each copy of \(P\) in \(\Gamma^{(d)}\).
    Moreover, \(\varphi\) must fix each copy of the vertices \(u_1,u_2,\) and \(u_3\), as these determine the adjacency of cliques.
    We saw that every automorphism of \(P\) also fixes every vertex besides \(u_1, u_2,\) and \(u_3\), so \(\varphi\) must be trivial.
    Hence \(F\) is injective.

    Let us show $F$ is surjective. 
    Since $\Aut(P)$ acts 3-transitively on the \(\{u_1,u_2,u_3\}\), any automorphism of $C_d(\Gamma^{(d)})$ lifts to an automorphism of $\Gamma^{(d)}$. 
    It follows that $F$ is an isomorphism.

    Now observe that any \(d\)-clique of \(\Gamma^{(d)}\) is contained in a copy of \(P\).
    Each copy of \(P\) in turn has a unique \(d\)-clique, which contains copies of the vertices \(u_1, u_2, \) and \(u_3\).
    By construction, any two such cliques are adjacent in \(C_d(\Gamma^{(d)})\) if and only if the copies of \(P\) they lie in correspond to adjacent vertices of \(\Gamma\).
    Therefore, we have an isomorphism \(\Gamma \cong C_d(\Gamma^{(d)})\), and thus an isomorphism $\Aut(\Gamma^{(d)}) \cong \Aut(\Gamma)$ by (\ref{eq:aut_clique_graph}).
    This also shows that $\Gamma^{(d)} \cong \Delta^{(d)}$ implies $\Gamma \cong C_d(\Gamma^{(d)}) \cong C_d(\Delta^{(d)}) \cong \Delta$.  
    
    To conclude, we show that \(\Gamma^{(d)}\) is edge-free.
    Suppose that \(\varphi\) is an automorphism of \(\Gamma^{(d)}\) fixing an edge,
    If \(\varphi\) fixes an edge in a copy of \(P\), then it must fix that copy of \(P\) and, therefore, its unique \(d\)-clique.
    But then \(F(\varphi)\) fixes a vertex of \(C_d(\Gamma^{(d)}) \cong \Gamma\), which is vertex-free.
    Similarly, if \(\varphi\) fixes an edge adjoining two copies of \(P\), then \(F(\varphi)\) fixes an edge of \(C_d(\Gamma^{(d)}) \cong \Gamma\), which is edge-free.
    Hence \(\Gamma^{(d)}\) is edge-free.
\end{proof}

We now prove our variant of Frucht's theorem, the main result of this section.

\begin{theorem}\label{thm:frucht}
    Let $G$ be a countable group, and $d \geq 3$. Then there exist uncountably many $d$-regular, simplicial, connected, edge-free graphs $\Gamma$ such that $\Aut(\Gamma) \cong G$, up to isomorphism. 
    Moreover, if $G$ is a hyperbolic group, then we may additionally take every such $\Gamma$ to be hyperbolic.
\end{theorem}

\begin{proof}
    We first deal with some exceptional small cases.
    Let \(\sigma \in \{0,1\}^\N\) be an infinite binary sequence that is not the constant sequence of \(1\)s.
    If \(G\) is trivial, we take \(\Gamma_0\) to be a copy of the tag \(T_\sigma\) from the proof of Proposition~\ref{prop:labelled_aut->aut}, with a copy of the Frucht graph attached to the leaf.
    If \(G\) is the group of order two (respectively, three), take \(\Gamma_0\) to be two (respectively, three) copies of \(T_\sigma\) with the leaf vertices joined to a vertex-free and edge-free \(3\)-regular graph whose automorphism group has order \(2\) (respectively, \(3\)) as in \cite[Theorem 2.4]{frucht1949graphs} (respectively, \cite[Theorem 3.1]{frucht1949graphs}).

    In each case, \(\Gamma_0\) is vertex-free, \(3\)-regular, satisfies \(\Aut(\Gamma_0) \cong G\), and is quasi-isometric to a tree, which is hyperbolic.
    Moreover, there are uncountably many pairwise non-isomorphic tags, so there are uncountably many such \(\Gamma_0\).
    Applying Lemma~\ref{lem:3-reg->d-reg} then yields the theorem in these cases.
    
    Suppose now that \(\abs{G} \geq 4\).
    We will construct a labelled graph $\Gamma_0$ such that: 
    \begin{enumerate}
        \item \(\Gamma_0\) is 3-regular and vertex-free;
        \item \(\Aut_\ell(\Gamma_0)\) is isomorphic to \(G\);
        \item if \(G\) is finitely generated, then \(\Gamma_0\) is quasi-isometric to \(G\).
    \end{enumerate}

    Let $S \subset G - \{1\}$ be a generating set containing at least \(3\) elements, with the property that if $s, s^{-1} \in S$ then \(s\) has order 2. 
    Such an $S$ exists, since one may take any generating set and delete one of every inverse pair that appears in the set.
    Let $I \subset S$ denote the subset containing all order 2 elements, and let $T = S - I$. 
    
    Let $\Gamma_{G,S}$ be the labelled Cayley graph of \(G\) with respect to \(S\), as described in Definition~\ref{def:cay graph}. 
    As in Remark~\ref{rem:auts_of_labelled_cayley_graph}, \(\Aut_\ell(\Gamma_{G,S})\) is isomorphic to \(G\). 
    We now modify $\Gamma_{G,S}$ as follows to make it 3-regular. To do so, we must consider two cases separately. 
    
    Firstly, let us consider the case where $S$ is countably infinite.
    Since $G$ is countable, fix a choice of bijection $f \colon \Z \to G$. 
    We then blow up each vertex in $\Gamma_{G,S}$ by replacing it with a copy of the real line, with incoming edge attached at each integer point, where the order is determined by the bijection \(f\). 
    We also label these new edges with directed labels, where the edge from $n$ to $n+1$ is equipped with a directed label of the integer $n$; see  Figure~\ref{fig:blow-up}. 
    We call the resulting labelled graph $\Gamma_0$.
    
    Now any label-preserving automorphism of \(\Gamma_0\) must permute the copies of \(P\).
    Moreover, the choice of labelling on \(P\) ensures that the group of label-preserving automorphisms of $\Gamma_0$ is isomorphic to $G$. 

    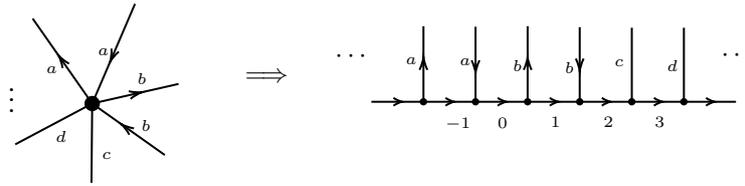
\begin{figure}[h]
        \centering
        \tikzset{every picture/.style={line width=0.75pt}} 

\begin{tikzpicture}[x=0.75pt,y=0.75pt,yscale=-1,xscale=1]

\draw    (173.4,120.89) -- (143.69,78.15) ;
\draw [shift={(156.49,96.56)}, rotate = 55.2] [color={rgb, 255:red, 0; green, 0; blue, 0 }  ][line width=0.75]    (6.56,-1.97) .. controls (4.17,-0.84) and (1.99,-0.18) .. (0,0) .. controls (1.99,0.18) and (4.17,0.84) .. (6.56,1.97)   ;
\draw    (173.4,120.89) -- (194.77,70.75) ;
\draw [shift={(182.28,100.05)}, rotate = 293.09] [color={rgb, 255:red, 0; green, 0; blue, 0 }  ][line width=0.75]    (6.56,-1.97) .. controls (4.17,-0.84) and (1.99,-0.18) .. (0,0) .. controls (1.99,0.18) and (4.17,0.84) .. (6.56,1.97)   ;
\draw    (173.4,120.89) -- (216.5,111.02) ;
\draw [shift={(198.46,115.15)}, rotate = 167.11] [color={rgb, 255:red, 0; green, 0; blue, 0 }  ][line width=0.75]    (6.56,-1.97) .. controls (4.17,-0.84) and (1.99,-0.18) .. (0,0) .. controls (1.99,0.18) and (4.17,0.84) .. (6.56,1.97)   ;
\draw    (173.4,120.89) -- (209.62,146.78) ;
\draw [shift={(187.76,131.16)}, rotate = 35.56] [color={rgb, 255:red, 0; green, 0; blue, 0 }  ][line width=0.75]    (6.56,-1.97) .. controls (4.17,-0.84) and (1.99,-0.18) .. (0,0) .. controls (1.99,0.18) and (4.17,0.84) .. (6.56,1.97)   ;
\draw    (173.4,120.89) -- (173.03,160.75) ;
\draw    (173.4,120.89) -- (135,141.43) ;
\draw    (173.4,120.89) ;
\draw [shift={(173.4,120.89)}, rotate = 0] [color={rgb, 255:red, 0; green, 0; blue, 0 }  ][fill={rgb, 255:red, 0; green, 0; blue, 0 }  ][line width=0.75]      (0, 0) circle [x radius= 3.35, y radius= 3.35]   ;
\draw    (312.75,120) -- (338.75,120) ;
\draw [shift={(338.75,120)}, rotate = 0] [color={rgb, 255:red, 0; green, 0; blue, 0 }  ][fill={rgb, 255:red, 0; green, 0; blue, 0 }  ][line width=0.75]      (0, 0) circle [x radius= 1.34, y radius= 1.34]   ;
\draw [shift={(328.15,120)}, rotate = 180] [color={rgb, 255:red, 0; green, 0; blue, 0 }  ][line width=0.75]    (4.37,-1.96) .. controls (2.78,-0.92) and (1.32,-0.27) .. (0,0) .. controls (1.32,0.27) and (2.78,0.92) .. (4.37,1.96)   ;
\draw    (338.75,120) -- (364.75,120) ;
\draw [shift={(364.75,120)}, rotate = 0] [color={rgb, 255:red, 0; green, 0; blue, 0 }  ][fill={rgb, 255:red, 0; green, 0; blue, 0 }  ][line width=0.75]      (0, 0) circle [x radius= 1.34, y radius= 1.34]   ;
\draw [shift={(354.15,120)}, rotate = 180] [color={rgb, 255:red, 0; green, 0; blue, 0 }  ][line width=0.75]    (4.37,-1.96) .. controls (2.78,-0.92) and (1.32,-0.27) .. (0,0) .. controls (1.32,0.27) and (2.78,0.92) .. (4.37,1.96)   ;
\draw    (364.75,120) -- (390.75,120) ;
\draw [shift={(390.75,120)}, rotate = 0] [color={rgb, 255:red, 0; green, 0; blue, 0 }  ][fill={rgb, 255:red, 0; green, 0; blue, 0 }  ][line width=0.75]      (0, 0) circle [x radius= 1.34, y radius= 1.34]   ;
\draw [shift={(380.15,120)}, rotate = 180] [color={rgb, 255:red, 0; green, 0; blue, 0 }  ][line width=0.75]    (4.37,-1.96) .. controls (2.78,-0.92) and (1.32,-0.27) .. (0,0) .. controls (1.32,0.27) and (2.78,0.92) .. (4.37,1.96)   ;
\draw    (390.75,120) -- (416.75,120) ;
\draw [shift={(416.75,120)}, rotate = 0] [color={rgb, 255:red, 0; green, 0; blue, 0 }  ][fill={rgb, 255:red, 0; green, 0; blue, 0 }  ][line width=0.75]      (0, 0) circle [x radius= 1.34, y radius= 1.34]   ;
\draw [shift={(406.15,120)}, rotate = 180] [color={rgb, 255:red, 0; green, 0; blue, 0 }  ][line width=0.75]    (4.37,-1.96) .. controls (2.78,-0.92) and (1.32,-0.27) .. (0,0) .. controls (1.32,0.27) and (2.78,0.92) .. (4.37,1.96)   ;
\draw    (416.75,120) -- (442.75,120) ;
\draw [shift={(442.75,120)}, rotate = 0] [color={rgb, 255:red, 0; green, 0; blue, 0 }  ][fill={rgb, 255:red, 0; green, 0; blue, 0 }  ][line width=0.75]      (0, 0) circle [x radius= 1.34, y radius= 1.34]   ;
\draw [shift={(432.15,120)}, rotate = 180] [color={rgb, 255:red, 0; green, 0; blue, 0 }  ][line width=0.75]    (4.37,-1.96) .. controls (2.78,-0.92) and (1.32,-0.27) .. (0,0) .. controls (1.32,0.27) and (2.78,0.92) .. (4.37,1.96)   ;
\draw    (442.75,120) -- (468.75,120) ;
\draw [shift={(468.75,120)}, rotate = 0] [color={rgb, 255:red, 0; green, 0; blue, 0 }  ][fill={rgb, 255:red, 0; green, 0; blue, 0 }  ][line width=0.75]      (0, 0) circle [x radius= 1.34, y radius= 1.34]   ;
\draw [shift={(458.15,120)}, rotate = 180] [color={rgb, 255:red, 0; green, 0; blue, 0 }  ][line width=0.75]    (4.37,-1.96) .. controls (2.78,-0.92) and (1.32,-0.27) .. (0,0) .. controls (1.32,0.27) and (2.78,0.92) .. (4.37,1.96)   ;
\draw    (468.75,120) -- (494.75,120) ;
\draw [shift={(484.15,120)}, rotate = 180] [color={rgb, 255:red, 0; green, 0; blue, 0 }  ][line width=0.75]    (4.37,-1.96) .. controls (2.78,-0.92) and (1.32,-0.27) .. (0,0) .. controls (1.32,0.27) and (2.78,0.92) .. (4.37,1.96)   ;
\draw    (338.75,120) -- (338.75,82) ;
\draw [shift={(338.75,97.4)}, rotate = 90] [color={rgb, 255:red, 0; green, 0; blue, 0 }  ][line width=0.75]    (6.56,-1.97) .. controls (4.17,-0.84) and (1.99,-0.18) .. (0,0) .. controls (1.99,0.18) and (4.17,0.84) .. (6.56,1.97)   ;
\draw    (364.75,120) -- (364.75,82.25) ;
\draw [shift={(364.75,105.73)}, rotate = 270] [color={rgb, 255:red, 0; green, 0; blue, 0 }  ][line width=0.75]    (6.56,-1.97) .. controls (4.17,-0.84) and (1.99,-0.18) .. (0,0) .. controls (1.99,0.18) and (4.17,0.84) .. (6.56,1.97)   ;
\draw    (416.75,82.5) -- (416.75,120) ;
\draw [shift={(416.75,104.85)}, rotate = 270] [color={rgb, 255:red, 0; green, 0; blue, 0 }  ][line width=0.75]    (6.56,-1.97) .. controls (4.17,-0.84) and (1.99,-0.18) .. (0,0) .. controls (1.99,0.18) and (4.17,0.84) .. (6.56,1.97)   ;
\draw    (390.75,120) -- (390.75,82) ;
\draw [shift={(390.75,97.4)}, rotate = 90] [color={rgb, 255:red, 0; green, 0; blue, 0 }  ][line width=0.75]    (6.56,-1.97) .. controls (4.17,-0.84) and (1.99,-0.18) .. (0,0) .. controls (1.99,0.18) and (4.17,0.84) .. (6.56,1.97)   ;
\draw    (442.75,82.75) -- (442.75,120) ;
\draw    (468.75,82.75) -- (468.75,120) ;

\draw (130,103.4) node [anchor=north west][inner sep=0.75pt]    {$\vdots $};
\draw (149,100.65) node [anchor=north west][inner sep=0.75pt]  [font=\tiny]  {$a$};
\draw (174.75,91.4) node [anchor=north west][inner sep=0.75pt]  [font=\tiny]  {$a$};
\draw (194.5,103.9) node [anchor=north west][inner sep=0.75pt]  [font=\tiny]  {$b$};
\draw (196.5,127.65) node [anchor=north west][inner sep=0.75pt]  [font=\tiny]  {$b$};
\draw (176.75,143.65) node [anchor=north west][inner sep=0.75pt]  [font=\tiny]  {$c$};
\draw (153.75,133.15) node [anchor=north west][inner sep=0.75pt]  [font=\tiny]  {$d$};
\draw (293.25,92.65) node [anchor=north west][inner sep=0.75pt]    {$\cdots $};
\draw (486,92.15) node [anchor=north west][inner sep=0.75pt]    {$\cdots $};
\draw (328.5,96.15) node [anchor=north west][inner sep=0.75pt]  [font=\tiny]  {$a$};
\draw (355.5,96.4) node [anchor=north west][inner sep=0.75pt]  [font=\tiny]  {$a$};
\draw (382,97.65) node [anchor=north west][inner sep=0.75pt]  [font=\tiny]  {$b$};
\draw (408,98.4) node [anchor=north west][inner sep=0.75pt]  [font=\tiny]  {$b$};
\draw (432.75,97.15) node [anchor=north west][inner sep=0.75pt]  [font=\tiny]  {$c$};
\draw (459,97.4) node [anchor=north west][inner sep=0.75pt]  [font=\tiny]  {$d$};
\draw (348.25,126.4) node [anchor=north west][inner sep=0.75pt]  [font=\tiny]  {$-1$};
\draw (374.25,126.4) node [anchor=north west][inner sep=0.75pt]  [font=\tiny]  {$0$};
\draw (400.75,126.15) node [anchor=north west][inner sep=0.75pt]  [font=\tiny]  {$1$};
\draw (427.25,125.9) node [anchor=north west][inner sep=0.75pt]  [font=\tiny]  {$2$};
\draw (452.75,125.9) node [anchor=north west][inner sep=0.75pt]  [font=\tiny]  {$3$};
\draw (248.25,103.9) node [anchor=north west][inner sep=0.75pt]    {$\Longrightarrow $};

\end{tikzpicture}
        \caption{Blowing up $\Gamma_{G,S}$ to a 3-regular graph $\Gamma_0$, in the case where $S$ is countably infinite.}
        \label{fig:blow-up}
    \end{figure}

    Now suppose instead that $S$ is finite. 
    Let $k$ denote the degree of $\Gamma_{G,S}$.
    If \(k = 3\), then we are done, so suppose \(k \geq 4\).
    We now form $\Gamma_0$ by equivariantly blowing up each vertex $v$ into a path $P$ of length $k-3$, adjoining exactly one of the edges incident upon $v$ to each internal vertex in this copy of $P$, and two to each of the edges. 
    We equivariantly label the edges of each copy of $P$ with the numbers $1$ through $k-3$.
    Again call the resulting labelled graph $\Gamma_0$ for this case. 
    As before, the group of label-preserving automorphisms of $\Gamma_0$ is isomorphic to $G$. 
    Moreover, in this case, collapsing each copy of \(P\) yields a quasi-isometry from $\Gamma_0$ to $\Gamma_{G,S}$.

    To conclude the proof, we apply Proposition~\ref{prop:labelled_aut->aut} to obtain uncountably many \(3\)-regular vertex- and edge-free graphs whose automorphism groups are isomorphic to \(\Aut_\ell(\Gamma_0) \cong G\).
    The proposition tells us that the graphs are hyperbolic whenever \(\Gamma_0\) is, and therefore, as \(\Gamma_0\) is quasi-isometric to \(G\), whenever \(G\) is hyperbolic.
    Finally, applying Lemma~\ref{lem:3-reg->d-reg} to each member of this uncountable family of graphs gives us the required family of graphs.
\end{proof}

\section{Boundary sets with trivial stabiliser}\label{sec:HH2}

In this section, we construct a finite set in the visual boundary \(\partial M\) of certain non-positively curved manifolds \(M\) with trivial pointwise stabiliser in \(\Isom(M)\). 
These sets will determine the configuration of edge spaces in our graph of groups construction in Section~\ref{sec:construction}.
A general setting this construction makes sense in is the class of visibility manifolds.

\begin{definition}
    A complete, simply connected and non-positively curved Riemannian manifold~$M$ of dimension at least two is called a \emph{visibility manifold} if for every two distinct points~$x, y\in\partial M$ there exists a bi-infinite geodesic~$\gamma$ with $\gamma(\infty)=x$ and $\gamma(-\infty)=y$.
\end{definition}

Note that every simply connected symmetric space of non-compact type is a visibility manifold~\cite{eberlein1996geometry}.

Roughly speaking, the idea in constructing our set of boundary points is to pick the vertices of an ideal polytope that is large enough to not be contained in any proper geodesically complete and totally geodesic submanifold.
As such, there will be no non-trivial isometries stabilising its vertices. 
For this purpose, let us introduce the following terminology.

\begin{definition}
    Let~$M$ be a visibility manifold, and let $Z\subseteq \partial M$ be a set of cardinality at least two. We define \emph{hull} of \(Z\), denoted~$\hull(Z)$, to be the minimal geodesically complete and totally geodesic submanifold of~$M$ whose limit set contains~$Z$.
\end{definition}

Note that $\hull (Z)$ does not generally coincide with the convex hull of~$Z$, but is rather the minimal totally geodesic submanifold containing its convex hull.
We see that an isometry stabilising a set~$\Omega$ with the property that its hull spans the entire visibility manifold must be the identity.

\begin{lemma}\label{lem:Omega-has-trivial-point-stabiliser-visibility}
    Let~$M$ be a visibility manifold.
    If \(\Omega \subseteq \partial M\) is a subset with \(\hull(\Omega) = M\), then \(\Omega\) has trivial pointwise stabiliser in~$\Isom(M)$.
\end{lemma}

\begin{proof}	
    First we observe that the condition~$\hull(\Omega)=M$ implies that~$\Omega$ has at least three points, for otherwise~$\hull{\Omega}$ is a geodesic with endpoints the two points of~$\Omega$.
    
    Suppose $f$ is an isometry of~$M$ whose induced map \(\partial f\) on $\partial M$ fixes~$\Omega$ pointwise.
    By well-known facts about isometries of visibility manifolds (see e.g.~\cite{bgs}), \(\partial f\) fixes 0, 1, 2, or infinitely many points of \(\partial M\). 
    As \(\partial f\) fixes \(\Omega\) and \(\abs\Omega \geq 3\), it must be that \(\partial f\) fixes infinitely many points of~$\partial M$.
    Thus $f$ is an elliptic isometry. 
    
    By~\cite[Lemma~6.3]{bgs}, the fixed point set of~$f$ in~$M$ is a complete totally geodesic submanifold containing all geodesics between pairs of points in~$\Omega$.
    By the lemma hypothesis, this must be~$M$. It follows that~$f$ is the identity map.
\end{proof}

There are finite subsets in the boundary of a visibility manifold satisfying the condition of the above lemma, whose cardinality depends only on the dimension of the space.

\begin{lemma}\label{lemma:choosing-omega-visibility}
    Let~$M$ be a visibility manifold or $M=\R$. Then there is a set of $\Omega\subset \partial M$ with~$|\Omega|\leq \dim M+1$ such that $\hull(\Omega)=M$.
\end{lemma}

\begin{proof}
    We proceed by induction on the dimension.
    By definition, a visibility manifold has dimension at least~$2$, so the base case is that $M=\R$. 
    In this case, $\partial M$ only has two points, and the geodesic spanning both of them is the entire line. 

    Now let \(n > 1\), and suppose the statement holds for visibility manifolds of dimension at most~$n-1$.
    Let~$N \subset M$ be a geodesically complete and totally geodesic proper submanifold of maximal dimension. 
    Such a submanifold always exists since every bi-infinite geodesic is a complete totally geodesic submanifold.
    By geodesic completeness we have $\dim N<\dim M$. 
    
    We claim that~$N$ is also a visibility manifold, so that the induction hypothesis applies.
    Let ~$x, y\in\partial N$ and choose $o\in N$ and geodesic rays $\gamma_x, \gamma_y$ based at \(o\) with $\gamma_x(\infty)=x, \gamma_y(\infty)=y$. 
    Consider the sequence of geodesic segments $\gamma_t=[\gamma_x(t),\gamma_y(t)]$, and note that $\gamma_t\subset N$ for every~$t$ as~$N$ is totally geodesic. 
    By completeness, the geodesic~$\gamma$ that $(\gamma_t)_t$ subconverges to is contained in~$N$.
    Necessarily, we have that $\gamma(\infty)=x$ and $\gamma(-\infty)=y$. 
    It follows that~$N$ is a visibility manifold.
    
    Now let~$\Omega^\prime \subset \partial N$ be a set of boundary points with~$\hull(\Omega^\prime)=N$ with \(\abs{\Omega'} \leq \dim N + 1\). 
    Take any point $\omega\in\partial M- \partial N$, which exists since \(\dim N < \dim M\), and set $\Omega=\Omega^\prime\cup \{\omega\}$. 
    It follows that \(\abs{\Omega} \leq \dim N+2 \leq \dim M + 1\).
    By definition, $\hull(\Omega)$ is a totally geodesic submanifold and contains \(N \subseteq \hull(\Omega')\).
    Since \(\omega \notin \partial N\) and \(\hull(\Omega)\) is geodesically complete, \(\dim \hull(\Omega) > \dim N\).
    As \(N\) was a proper submanifold with these properties of maximal dimension, we must have that \(\hull(\Omega) = M\).
\end{proof}

\begin{proposition}\label{prop:finite-set-with-trivial-pstab-visibility}
    Let~$M$ be a visibility manifold.
    Then there exists a finite set $\Omega\subset \partial M$ with $|\Omega| = n+1$ that has trivial pointwise stabiliser in~$\Isom(M)$.
\end{proposition}

\begin{proof}
    Lemmas~\ref{lemma:choosing-omega-visibility} and~\ref{lem:Omega-has-trivial-point-stabiliser-visibility} imply the existence of a set of boundary points~$\Omega$ with trivial pointwise stabiliser in~$\Isom(M)$. If~$\Omega$ has cardinality less than $\dim M +1$, we may add rays in our favourite directions, which does not change the properties of~$\Omega$.
\end{proof}

\begin{remark}
    Given a particular visibility manifold, it may be possible to choose such a set~$\Omega$ with strictly less than $\dim M +1$ elements.
    This can be done by analysing its totally geodesic submanifolds in more detail. 
    However, for our purposes, this non-sharp statement suffices.
\end{remark}


\section{Graphs of metric spaces}\label{sec:gos}

In this section, we will define our notion of a graph of metric spaces, which is an adaptation of a construction already appearing in Scott--Wall \cite[p. 155]{scott1979topological}.

Let $\Gamma$ be a graph. In this paper, every graph will be undirected and simplicial. That is, \(\Gamma\) will have no loops and no double-edges. 
Denote by $V(\Gamma)$ its vertex set, and $E(\Gamma)$ its set of edges. 
Given any graph \(\Gamma\), we will always fix an orientation on its edge set, so that each edge \(e \in E(\Gamma)\) has a fixed initial vertex \(\iota(e)\) and terminal vertex \(\tau(e)\).
As an abuse of notation, we will write \(\bar e\) to denote the edge \(e \in E(\Gamma)\) with reverse orientation, so that \(\iota(\bar e) = \tau(e)\) and \(\tau(\bar e) = \iota(e)\). 
For a vertex $v\in V(\Gamma)$ we denote by~$\In(v)$ the set of edges incident to the vertex~$v$, that is, either \(\iota(e) = v\) or \(\tau(e) = v\). 

\begin{definition}[Graph of spaces]
    A \emph{graph of metric spaces} $\mathbf X = (\Gamma,X_-,Y_-, \alpha_-)$ consists of the following data:
\begin{enumerate}
    \item a connected, simplicial graph $\Gamma$; 

    \item for every $v \in V(\Gamma)$, a geodesic metric space $X_v$;

    \item for every $e \in E(\Gamma)$, a length metric space $Y_e$; 

    \item for every $e \in E(\Gamma)$, a pair of isometric embeddings $\alpha_e \colon Y_e \to X_{\tau(e)}$ and \(\alpha_{\bar e} \colon Y_e \to X_{\iota(e)}\). 
\end{enumerate}
\end{definition}
For brevity, we will often refer to a graph of metric spaces as simply a `graph of spaces' throughout; the presence of additional metric structure is always implicit.
Following Scott--Wall \cite{scott1979topological}, we construct a realisation of a graph of metric spaces.

\begin{definition}[Realisation of a graph of metric spaces]
    Let \(\X = (\Gamma,X_-,Y_-,\alpha_-)\) be a graph of metric spaces.    
    Let $A$ denote the disjoint union 
    \[
        A = \bigsqcup_{v \in V(\Gamma)} X_v \sqcup \bigsqcup_{e \in E(\Gamma)} Y_e \times [0,1].
    \]
    Let $\sim$ be the equivalence relation on $A$ generated by relations of the form
    \[
        (x,0) \sim \alpha_e(x), \ (x,1) \sim \alpha_{\bar e}(x),
    \]
    where $e \in E(\Gamma)$, $x \in Y_e$. 
    The \emph{realisation} of \(\X\) is then the quotient space $\abs\X =A / \sim$. 

    We will refer to the image of each $X_v$ in $|\X|$ as a \emph{vertex space}, and images of the cylinders $Y_e \times [0,1]$ as \emph{edge cylinders}.
    Edge cylinders will always be equipped with the \(\ell_2\)-metric corresponding to the product.
    We will often write $Z_e = Y_e \times [0,1]$ to denote an edge cylinder, and we call the subspace \(Y_e \times (0,1)\) the \emph{interior of \(Z_e\)}.
    A subset of \(\X\) is called a \emph{piece of \(\X\)} if it is equal to either \(X_v\) for some \(v \in V(\Gamma)\) or \(Z_e\) for some \(e \in E(\Gamma)\).
\end{definition}

Note that, as \(\Gamma\) is simplicial, the projection map \(A \to \abs\X\) described above restricts to a topological embedding on pieces of \(\X\).
Hence we may view these as genuine subspaces of \(\abs\X\).

We now describe how to metrise the realisation of a graph of metric spaces. 
This construction is roughly based on \cite[\S~I.7]{bridson2013metric}. 

\begin{definition}[String]
\label{def:string}
    Let \(\X = (\Gamma,X_-,Y_-,\alpha_-)\) be a graph of metric spaces, and let \(x, y \in \abs\X\).
    A \emph{string} in $\mathbf X$ between $x$ and $y$ is a finite sequence $S = (x_0, \ldots, x_{n})$, $x_0 = x$, $x_n = y$, together with a choice of piece $P_i\subset |\mathbf X|$ containing both $x_{i-1}$and $x_{i}$ for each $i = 1, \ldots, n$. 
    The \emph{length} of $S$, denoted $\ell(S)$, is defined as 
    \[
        \ell(S) = \sum_{i=i}^{m} \dist_{P_i}(x_{i-1}, x_{i}).
    \]
\end{definition}
 
\begin{definition}[Intrinsic metric]
\label{def:gos_pseudo_metric}
    Let \(\X\) be a graph of metric spaces.
    We call the function \(\dist_\X \colon \abs\X \times \abs\X \to \R\) defined by
    \begin{equation*}
        \dist_{\X}(x,y) = \inf\{\ell(S) : \text{$S$ is a string with endpoints $x$ and $y$}\}.
    \end{equation*}
    the \emph{intrinsic metric on \(\abs\X\)}.
\end{definition}

\begin{remark}
\label{rem:subgraph_of_spaces}
    Suppose that \(\X = (\Gamma,X_-,Y_-,\alpha_-)\) is a graph of spaces and \(\Gamma' \subseteq \Gamma\) is a subgraph.
    Writing \(\X' = (\Gamma',X_-,Y_-,\alpha_-)\), there is a canonical map of realisations \(\abs{\X'} \to \abs\X\), which is a topological embedding.
    Moreover, any string in \(\X'\) maps to a string of the same length in \(\X\) under this embedding, so that \(\dist_\X \leq \dist_{\X'}\) on the image of \(\abs{\X'}\) in \(\abs\X\).
\end{remark}

The function \(\dist_\X\) defined above is easily seen to be a pseudo-metric, at least: lengths of strings are always non-negative and the triangle inequality follows by concatenating strings.
We will see shortly that it is in fact a metric.

\begin{definition}[Reduced string]
    Let \(\X\) be a graph of metric spaces.
    A string \(S\) in \(\X\) is called \emph{reduced} if the pieces of \(S\) alternate between vertex spaces and edge cylinders, and any consecutive triple of pieces are pairwise distinct.
\end{definition}

Note that it is not necessary for consecutive terms in a reduced string to be distinct.
It is straightforward to see that the distance between two points can be realised as the infimum of lengths of reduced strings.

\begin{lemma}
\label{lem:reduced_string}
    Let \(\X\) be a graph of metric spaces.
    For any \(x, y \in \abs{\X}\), we have
    \[
        \dist_{\X}(x,y) = \inf \{ \ell(S) : S \textnormal{ is a \textbf{reduced} string between } x \textnormal{ and }y\}.
    \]
\end{lemma}

\begin{proof}
    For each string \(S\), there is a reduced string \(S'\) with the same endpoints, obtained by deleting entries of \(S\) and then, if needed, adding redundant entries.
    Moreover, each edge space is isometrically embedded in its adjacent vertex spaces, so \(\ell(S') \leq \ell(S)\) by applying the triangle inequality in each piece.
\end{proof}

On small scales, the intrinsic metric on the realisation of a graph of spaces resembles the metric on individual pieces.

\begin{lemma}
\label{lem:local_gos_metric_equals_piece_metric}
    Let \(\X\) be a graph of metric spaces and let \(P\) be a piece of \(\X\).
    If \(x, y \in P\) are such that \(\dist_\X(x,y) < 1\), then \(\dist_\X(x,y) = \dist_P(x,y)\).
\end{lemma}

\begin{proof}
    Let \(\varepsilon > 0\) and let \(S\) be a reduced string from \(x\) to \(y\) with \(\ell(S) < \dist_\X(x,y) + \varepsilon\).
    Since \(\dist_\X(x,y) < 1\), we may assume without loss of generality that \(\ell(S) < 1\).
    Recall that if \(Z = Y \times [0,1]\) is an edge cylinder, \(\dist_Z((a,0),(b,1)) \geq 1\) for any \(a,b \in Y\).
    Hence the string \(S\) does not contain two points on opposite sides of an edge cylinder.
    In particular, each of the points of \(S\) belong to a single vertex space \(X\) and its adjacent edge cylinders.
    
    Since the pieces of \(S\) alternate between edge cylinders and vertex spaces, every other piece of \(S\) must be \(X\).
    As \(S\) is reduced and the first and last pieces of \(S\) are both \(P\), this implies that \(S\) has exactly one piece.
    This means that \(S = (x, y)\), and so \(\ell(S) = \dist_P(x,y)\).
    Hence \(\dist_P(x,y) < \dist_\X(x,y) = \epsilon\) as required.
\end{proof}

\begin{lemma}
\label{lem:pseudometric_is_a_metric}
    Let \(\X\) be a graph of metric spaces.
    The intrinsic metric \(\dist_\X\) is a metric.
\end{lemma}

\begin{proof}
    To verify that \(\dist_\X\) is a metric, we need only check that it is positive definite.
    Let \(x\) and \(y\) be points of \(\abs{\X}\) and suppose that \(\dist_{\X}(x,y) = 0\).
    By Lemma~\ref{lem:reduced_string}, \(\dist_{\X}(x,y)\) is the infimum of lengths of reduced strings between \(x\) and \(y\).

    Suppose \(x\) is contained in the interior of an edge cylinder \(Z \cong Y \times (0,1)\), so that \(x = (a,r)\) for some \(a \in Y\) and \(r \in (0,1)\).
    Let \(\varepsilon = \max\{r,1-r\}\) and let \(S = (x_0, \dots, x_n)\) be a reduced string between \(x_0 = x\) and \(x_n = y\) with \(\ell(S) < \varepsilon\). 
    By the choice of \(\varepsilon\), it must be that \(x_i \in Z\) for all \(i = 1, \dots, n\).
    Lemma~\ref{lem:local_gos_metric_equals_piece_metric} then implies \(\dist_Z(x,y) = \dist_\X(x,y) = 0\).
    As \(\dist_Z\) is a metric, we must have \(x = y\),
    Similarly, if \(y\) is in the interior of an edge cylinder, we also have \(x = y\).

    Consider now the remaining case, that \(x\) and \(y\) are both contained in vertex spaces.
    Now let \(S = (x_0, \dots, x_n)\) be a reduced string between \(x_0 = x\) and \(x_n = y\) with \(\ell(S) < 1\). 
    It must be that \(x\) and \(y\) belong to some vertex space \(X\), since the edge cylinders have width \(1\).
    Again Lemma~\ref{lem:local_gos_metric_equals_piece_metric} implies \(\dist_X(x,y) = \dist_\X(x,y) = 0\), so \(x = y\) as required.
\end{proof}

We describe a general construction that will be useful in describing paths in a graph of spaces.

\begin{definition}[Path and string maps]
\label{def:path_str}
    Let \(\X = (\Gamma,X_-,Y_-.\alpha_-)\) be a graph of spaces. 
    Denote by \(\mathcal{S}_\X\) the set of strings in \(X\) and \(\mathcal{P}_\X\) the set of continuous paths in \(\abs\X\) whose image meets the interior of finitely many pieces of \(\X\).
    We define the map \(\operatorname{Str}_\X \colon \mathcal{P}_\X \to \mathcal{S}_\X\) as follows.
    Given a continuous path \(p \colon I \to \abs\X\) in \(\mathcal{P}_\X\), where \(I\) is a closed interval, let \(t_0, \dots, t_n\) be the partition of \(I\) that is minimal with respect to the property that \(p|_{(t_{i-1},t_i)}\) is contained in the interior of exactly one piece \(P_i\) of \(\X\) for each \(i = 1, \dots, n\).
    Then we define the string
    \[
        \operatorname{Str}_\X(p) = (p(t_0), \dots, p(t_n)).
    \]
    with the given piece decomposition \(P_1, \dots, P_n\).
    
    For each piece \(P\) of \(\X\) and each pair of points \(x, y \in P\), fix a choice of geodesic \([x,y]\).
    We define the map \(\operatorname{Path}_\X \colon \mathcal{S_\X \to \mathcal{P}_\X}\) as follows.
    Suppose that \(S = (x_0,\dots,x_n)\) is a string in \(\X\) and, for each \(i = 1, \dots, n\), let \(P_i\) denote the choice of piece of \(\X\) containing \(x_{i-1}\) and \(x_i\).
    Then define \(\operatorname{Path}_\X(S) \colon [0,\ell(S)] \to \abs\X\) as the path that is the concatenation of the \(P_i\)-geodesics \([x_{i-1},x_i]\) over \(i = 1, \dots, n\).
\end{definition}

\begin{remark}
    If there is \(\delta \geq 0\) such that all vertex spaces of \(\X\) are \(\delta\)-hyperbolic, then any two geodesics in a piece with the same endpoints will be a Hausdorff distance of \(\delta+1\) from one another.
    In such a situation, \(\operatorname{Path}_\X(S)\) is independent of the choice of geodesics made in the definition of \(\operatorname{Path}_\X\) given a string \(S\), up to a Hausdorff distance of \(\delta+1\).
    Adding \(\delta+1\) to constants in the relevant statements, therefore, we may make any choice of geodesics that is convenient to us in each given situation.
    We will use this fact without reference in the remainder of the paper.
\end{remark}

Under some basic assumptions on the underlying graph and vertex and edge spaces, this intrinsic metric makes the realisation into a geodesic space in an effective manner.
We will invoke these assumptions fairly often in the remainder of the paper, so we collect them here.

\begin{convention}\label{convention:assumptions_for_hyperbolicity_proof}
    Given a graph of spaces \(\X = (\Gamma,X_-,Y_-,\alpha_-)\), we will assume that the graph $\Gamma$ is locally finite, and \(X_v\) and \(Y_e\) are complete and locally compact for each \(v \in V(\Gamma)\) and \(e \in E(\Gamma)\).
\end{convention}

\begin{proposition}
\label{prop:graph-of-spaces-is-proper-geodesic}
    Let \(\X = (\Gamma,X_-,Y_-,\alpha_-)\) be a graph of metric spaces.
    Under Convention~\ref{convention:assumptions_for_hyperbolicity_proof}, \(\abs\X\) is a proper geodesic space.
    
    Moreover, for any geodesic \(p\) in \(\abs\X\), there is a reduced string \(S_p\) whose points lie on \(p\), and \(p\) is a concatenation of geodesics between the terms of \(S_p\), each of which lie in the interior of a single piece of \(\X\).
\end{proposition}

\begin{proof}
    Recall that the Hopf--Rinow theorem tells us that being complete and locally compact implies being proper and geodesic, among the class of length metric spaces \cite[Proposition~I.3.7]{bridson2013metric}.
    Hence we need only verify the former properties.
    Let \(x \in \abs\X\) be a point.
    As \(\Gamma\) is locally finite, \(B = B_\X(x;1)\) intersects at most finitely many pieces of \(\X\).
    Thus the closure \(\overbar{B}\) is the union of finitely many closed and bounded subsets in pieces of \(\X\).
    Again by the Hopf--Rinow theorem each piece of \(\X\) is proper, so each of these finitely many subsets are compact.
    Hence \(\overbar{B}\) is compact, and so \(x\) has a compact neighbourhood.

    To see that \(\abs\X\) is complete, let \((a_n)_{n \in \N}\) be a Cauchy sequence in \(\abs\X\).
    For each \(n \in \N\), let \(P_n\) be a piece of \(\X\) containing \(a_n\).
    There is \(N \in \N\) such that for \(n, m \geq N\), we have \(\dist_\X(a_n,a_m) < 1\).
    This implies that for any \(n \geq N\), the terms \(a_n\) and \(a_N\) lie in adjacent pieces.
    As \(\Gamma\) is locally finite, there are only finitely many pieces adjacent to \(P_N\), so we may pass to a subsequence \((a_{n_i})_{i\in\N}\) whose terms eventually lie in the same piece \(P\).
    By Lemma~\ref{lem:local_gos_metric_equals_piece_metric}, this subsequence is also Cauchy in \(P\) and, hence converges in \(P\), as \(P\) is complete.
    It follows that \(\abs\X\) is complete, since any Cauchy sequence with a convergent subsequence converges.

    For the latter claim, suppose that \(p \colon I \to \abs\X\) is a geodesic, where \(I = [a,b]\) is a closed interval.
    Since \(p\) is continuous, its image \(p(I)\) is a compact subset of \(\abs\X\).
    As \(\abs\X\) is a complete metric space, \(p(I)\) is totally bounded.
    In particular, there is a finite cover of \(p(I)\) by subsets whose diameter is at most \(1/2\), say.
    Now since \(\Gamma\) is locally finite, each of these subsets meets at most finitely many pieces of \(\X\).
    Therefore \(p(I)\) meets the interior of at most finitely many pieces of \(\X\).
    Thus \(S_p = \operatorname{Str}_\X(p)\) is defined, and we may take it to be reduced as \(p\) is geodesic and the edge spaces of \(\X\) are isometrically embedded in their adjacent vertex spaces.
    Moreover, each subpath of \(p\) is also geodesic, so that its subpaths between the finitely many terms of \(S\) are geodesics lying in pieces of \(\X\), as required.
\end{proof}


\section{Uniformly hyperbolic graphs of spaces}\label{sec:uniformly-hyperbolic-gos}

\subsection{Definition and basic properties}
We are particularly interested in graphs of hyperbolic metric spaces that are assembled in a uniformly regular way. We will call these \emph{uniformly hyperbolic} graphs of spaces. 
In this section we will see that in such graphs of spaces, geodesics are well-behaved.
As a consequence, these spaces satisfy auspicious metric properties such as quasiconvexity of the vertex spaces.

\begin{definition}(Uniformly hyperbolic graph of spaces)
    Let \(\X = (\Gamma,X_-,Y_-,\alpha_-)\) be a graph of metric spaces, \(\delta \geq 0\), and \(C \geq 0\).
    We say that \(\X\) is \emph{\((\delta,C)\)-uniformly hyperbolic} if:
    \begin{enumerate}
        \item \(X_v\) is \(\delta\)-hyperbolic for every \(v \in V(\Gamma)\);
        \item for any edge \(e \in E(\Gamma)\), there is \(y_e \in Y_e\) such that for any distinct edge \(e' \in E(\Gamma)\) we have
        \begin{enumerate}
            \item \(\langle \alpha_e(Y_e), \alpha_{e'}(Y_{e'}) \rangle_{\alpha_e(y_e)} \leq C\) if \(\tau(e') = \tau(e)\);
            \item \(\langle \alpha_{\bar{e}}(Y_e), \alpha_{e'}(Y_{e'}) \rangle_{\alpha_{\bar{e}}(y_e)} \leq C\) if \(\tau(e') = \iota(e)\).
        \end{enumerate}   
        Moreover, for any \(e, e' \in E(\Gamma)\) with \(\tau(e) = \tau(e')\), we have \(\alpha_e(y_e) = \alpha_{e'}(y_{e'})\).
    \end{enumerate}
    The last condition implies that for any given vertex space~$X_v$ there is a unique point $y_v=\alpha_{e}(y_e)$, which does not depend on the choice of incident edge~$e$, satisfying condition~(2). We will call this point the \emph{basepoint} of~$X_v$.
\end{definition}

The second condition should be seen as a sort of acylindricity condition. It essentially states that the images of distinct edge spaces diverge from one another uniformly quickly in the vertex spaces.

\begin{remark}
\label{rem:edge_pieces_quasiconvex}
    Let \(\X = (\Gamma,X_-,Y_-,\alpha_-)\) be a \((\delta,C)\)-uniformly hyperbolic graph of spaces.
    Since each edge space isometrically embeds into its adjacent vertex spaces, each edge space is also \(\delta\)-hyperbolic, so long as it is a geodesic space.
    
    Moreover, by hyperbolicity, any two geodesics joining the same two points of \(X_v\) are \(\delta\)-Hausdorff close in \(X_v\).
    Since each \(\alpha_e\) is an isometric embedding, this implies that \(\alpha_e(Z)\) is \(\delta\)-quasiconvex in \(X_{\tau(e)}\) for each connected component \(Z \subseteq Y_e\).
\end{remark}

The realisation of a uniformly hyperbolic graph of spaces contains a naturally embedded copy of its underlying graph \(\Gamma\), as the union of subsets \(\{y_e\} \times [0,1]\) ranging over \(e \in E(\Gamma)\), where the points \(y_e\) are given by uniform hyperbolicity.
It is immediate from Lemma~\ref{lem:local_gos_metric_equals_piece_metric} that this embedding is a local isometry.
We see that this copy of \(\Gamma\) is in fact globally isometrically embedded.

\begin{lemma}\label{lemma:graph_is_isometrically_embedded}
    Let \(\X = (\Gamma,X_-,Y_-,\alpha_-)\) be a $(\delta, C)$-uniformly hyperbolic graph of spaces, and 
    consider the subspace
    \begin{equation*}
        \abs\Y =\underset{e \in E(\Gamma)}{\bigcup} \{y_e\} \times [0,1] \subset \abs\X, 
    \end{equation*}
    where \(y_e\) are the points given by uniform hyperbolicity.
    Then~\(|\mathbf{Y}|\) equipped with the subspace metric \(\dist_\mathbf{Y}\) induced from~$|\X|$ is a geodesic and convex subspace of~$|\X|$ and isometric to the geometric realisation~\(|\Gamma|\) of~$\Gamma$.
\end{lemma}

\begin{proof}
    Define a map
    \begin{align*}
        \Phi\colon V(\Gamma)&\to \{(y_e,1):\, e\in E(\Gamma)\},\\
        v&\mapsto (y_e,1),  \text{ where }e\text{ is such that }v=\tau(e).
    \end{align*}
    Note that~$\Phi(v)$ exists for every vertex $v\in V(\Gamma)$ because~$\Gamma$ is connected.
    The definition of uniform hyperbolicity implies that $\Phi(v)$ is well defined, for whenever we have two edges~$e$ and~$e^\prime$ with $\tau(e)=\tau(e^\prime)$ then $\alpha_e(y_e) = \alpha_{e^\prime}(y_{e^\prime})$ in \(X_{\tau(e)}\). 
    In other words, $(y_e, 1)=(y_{e^\prime}, 1)$ in~$|\X|$.

    This can be extended to a map $\Psi\colon |\Gamma|\to |\mathbf{Y}|$.
    Let \(e \in E(\Gamma)\) be an edge with \(\iota(e) = u\) and \(\tau(e) = v\).
    We map a point~$t \in [0,1]$ in the interior of the realisation of edge \(e\) to $(y_e, t) \in Z_e$. 
    The map~$\Psi$ is a continuous bijection. 
    We will now show that~$\Psi$ is an isometry.

    Let $\varepsilon>0$, let $x, y\in |\mathbf{Y}|$ and let $p\subset |\mathbf{X}|$ be a path from~$x$ to~$y$ such that $\ell(p) \leq d_{\X}(x, y)+\varepsilon$. 
    As auxiliary terminology, we will say that $p$ \emph{traverses an edge cylinder}~$Z_e$ whenever $p\cap Z_e$ is connected and $p\cap X_{i(e)}\neq \emptyset, p\cap X_{\tau(e)}\neq \emptyset$. 
    Of course, \[\ell(p)\geq \# \{e\in E(\Gamma) : p \text{ traverses }Z_e\},\] so that when $x, y\in \Psi(V(\Gamma))$ we can already conclude that \[d_{\mathbf{Y}}(x,y)=d_{\mathbf{X}}(x,y) \geq d_{\Gamma}(\Psi^{-1}(x),\Psi^{-1}(y)) - \varepsilon.\]
    In the general case one needs to add the distance to the first and last vertex, which by construction of~$\Psi$ is at least as large as the corresponding distance in the graph. The same inequality of distances follows.
    As \(\varepsilon > 0\) was arbitrary and \(\Psi\) is a bijection, we have \(\dist_\Y(\Psi(u),\Psi(v)) \geq \dist_\Gamma(u,v)\) for all \(u,v \in |\Gamma|\).

    On the other hand, for every pair of points $\Psi(u), \Psi(v)$ in~$|\mathbf{Y}|$, let~$p$ be the geodesic in~$|\Gamma|$ from $u$ to $v$. 
    Then $\Psi(p)$ is a path from~$\Psi(u)$ to~$\Psi(v)$ whose length is equal to $\ell_\Gamma(p)$, so that $d_{\mathbf{Y}}(\Psi(u),\Psi(v)) \leq d_\Gamma (u, v)$. 
    Thus \(\Psi\) is an isometry between \(\abs\Gamma\) and \(\Y\).
    As \(\abs\Gamma\) is geodesic, so $|\mathbf{Y}|$ is geodesic. 

    It remains to show that $|\mathbf{Y}|$ is convex. We define a projection 
    \begin{align*}
        \Pi\colon |\X|&\to |\mathbf{Y}|\\  x&\mapsto \begin{cases}
            \alpha_e(y_e) & \text{ if }x\in X_{\tau(e)},\\
            (y_e, t) & \text{ if }x=(z,t)\in Z_e.
        \end{cases}  
    \end{align*}
    Now~$\Pi$ is continuous and length non-increasing on pieces, so that $\ell(p)\geq \ell(\Pi(p))$ for any rectifiable path \(p\) in \(\X\).
    Further, whenever~$p$ intersects a piece~$P$ in a segment~$q$ such that $q\cap |\mathbf{Y}|\neq\emptyset$ and $q\cap (|\X|- |\mathbf{Y}|)\neq\emptyset$, then $\ell_P(q)>\ell_P(\Pi(q))$, so the above length inequality is strict. 
    We already know that any two points $x, y\in |\mathbf{Y}|$ can be joined by a geodesic contained entirely in~$|\mathbf{Y}|$, so every other geodesic also has to be contained entirely in~$|\mathbf{Y}|$.
\end{proof}

\subsection{Geodesics in uniformly hyperbolic graphs of spaces}\label{subsection:geodesics_in_uniformly_hyperbolic_graphs_of_spaces}
This subsection is devoted to understanding geodesics in uniformly hyperbolic graphs of spaces.
We will show that geodesics in a uniformly hyperbolic graph of spaces trace out quasi-geodesics in the underlying graph.
We begin by analysing strings whose lengths are close to realising the distances between points.
It turns out that such strings cannot travel too deep into any vertex piece, besides the start and end terms.

\begin{lemma}
\label{lem:almost_geodesics_stay_close_to_basepoint}
    Let \(\X = (\Gamma,X_-,Y_-,\alpha_-)\) be a \((\delta,C)\)-uniformly hyperbolic graph of spaces, and let \(x, y \in \abs{\X}\).
    There is \(D = D(\delta,C) \geq 0\) such that for any \(\varepsilon > 0\), the following holds.
    
    Let \(S = (x_0, \dots, x_m)\) be a reduced string in \(\X\) with \(\ell(S) \leq \dist_\X(x_0,x_m) + \varepsilon\).
    For each \(i = 1, \dots, m\), write \(P_i\) for the piece of \(\X\) containing \(x_{i-1}\) and \(x_i\) and suppose that \(P_1\) and \(P_m\) are vertex spaces of \(\X\).
    If \(P_i = X_v\) is a vertex space then the following is true:
    \begin{itemize}
        \item if \(2 \leq i \leq m-3\), then \(x_i \in B_{X_v}(y_v;D+\varepsilon)\); and
        \item if \(4 \leq i \leq m-1\), then \(x_{i-1} \in B_{X_v}(y_v;D+\varepsilon)\).
    \end{itemize}
    where \(y_{v}\) is the basepoint of \(X_v\) given by the definition of uniform hyperbolicity.
\end{lemma}

\begin{proof}
    Define \(D = 4C + 4\delta + 1\).
    We will show the first inclusion in the statement, for the latter will follow from an exactly symmetrical argument.
    
    Let \(2 \leq i \leq m-3\) and suppose \(P_i\) corresponds to a vertex space \(X_v\) as in the statement.
    Since \(S\) is a reduced string, \(P_{i-1}\), \(P_{i}\), and \(P_{i+1}\) are distinct pieces of \(\X\).
    Thus, \(P_{i-1}\) and \(P_{i+1}\) correspond to distinct edge cylinders \(Z_e\) and \(Z_{e'}\) incident on \(X_v\).
    As \(i \leq m-3\) and \(S\) is reduced, \(P_{i+2}\) is a vertex space \(X_{v'}\) distinct from \(X_v\).
    We will first consider the former case.
    
    Let \(y_v \in X_v\) be the basepoint of \(X_v\) given by uniform hyperbolicity.
    Lemma~\ref{lem:basically_thin_triangles} gives us a point \(s\) on the \(X_v\)-geodesic \([x_i,x_{i+1}]\) such that
    \begin{equation}
    \label{eq:dist_from_basepoint_in_Xv}
        \dist_{X}(s,y_v) \leq C + \delta.
    \end{equation}
    Moreover, observe that \(i+2 \leq m-1\), and so \(P_{i+3}\) must be an edge cylinder adjacent to \(P_{i+2}\).
    Applying the same argument as above to \(P_{i+2}\), there is a point \(t\) on the \(X_{v'}\)-geodesic \([x_{i+2},x_{i+3}]\) such that
    \begin{equation}
    \label{eq:dist_from_basepoint_in_Xv'}
        \dist_{X}(t,y_{v'}) \leq C + \delta,
    \end{equation}
    similarly writing \(y_{v'}\) for the basepoint of \(X_v\).
    Define the string \[S' = (x_0, \dots x_i,s,y_v,y_{v'},t,x_{i+3}, \dots, x_m).\] 
    Now \(S'\) by definition has length 
    \begin{equation}
    \label{eq:len_of_shortened_string}
    \begin{split}
        \ell(S') = \ell(S) &- \dist_{P_i}(s,x_{i+1}) - \dist_{P_{i+1}}(x_{i+1},x_{i+2}) - \dist_{P_{i+2}}(x_{i+2},t) \\
            &+ \dist_{P_i}(s,y_v) + \dist_{P_{i+1}}(y_v,y_{v'}) + \dist_{P_{i+2}}(y_{v'},t).
    \end{split}
    \end{equation}
    
    The triangle inequality combined with (\ref{eq:dist_from_basepoint_in_Xv}) implies that
    \begin{equation*}
        \dist_{P_i}(s,y_v) - \dist_{P_i}(s,x_{i+1}) \leq 2C + 2\delta - \dist_{P_i}(x_{i+1},y_v)
    \end{equation*}
    and similarly (\ref{eq:dist_from_basepoint_in_Xv'}) gives that
    \[
        \dist_{P_{i+2}}(t,y_{v'}) - \dist_{P_{i+2}}(t,x_{i+2}) \leq 2C + 2\delta - \dist_{P_{i+2}}(x_{i+2},y_{v'}).
    \]
    As \(\dist_{P_{i+1}}(y_v,y_{v'}) = 1\) by definition, combining the above two equations with (\ref{eq:len_of_shortened_string}) and the lemma hypothesis yields
    \[\ell(S') \leq \dist_\X(x_0,x_m) - \dist_{P_i}(x_{i+1},y_v) - \dist_{P_{i+2}}(x_{i+2},y_{v'}) + \varepsilon + 4C + 4\delta + 1.\]
    Finally, we have \(\ell(S') \geq \dist_{\X}(x_0,x_m)\) by definition of the metric, so we obtain the inequality \(\dist_{P_i}(x_{i+1},y_v) \leq D + \varepsilon\).
\end{proof}

Reduced strings close to realising the distance between two points track a path through the underlying graph that is close to being a geodesic.

\begin{proposition}
\label{prop:almost_geodesics_give_graph_quasigeodesics}
    Let \(\X = (\Gamma,X_-,Y_-,\alpha_-)\) be a \((\delta,C)\)-uniformly hyperbolic graph of spaces. 
    There is a constant \(K = K(\delta,C) \geq 0\) such that the following is true.

    Let \(u, v \in V(\Gamma)\) and \(x \in X_u, y \in X_v\), and let \(\varepsilon > 0\).
    Suppose that \(S\) is a reduced string between \(x\) and \(y\) with \(\ell(S) \leq \dist_\X(x,y) + \varepsilon\).
    Write \(v_1, \dots, v_n \in V(\Gamma)\) for the sequence of vertices for which a term in \(S\) is contained in \(X_{v_i}\), in the order they appear.
    Then \[n \leq \dist_\Gamma(u,v) + K + 5\varepsilon.\]
    In particular, the path in \(\Gamma\) with vertices \(v_1, \dots, v_n\) is a \((1,K+5\varepsilon)\)-quasi-geodesic between \(u\) and \(v\) in \(\Gamma\).
\end{proposition}

\begin{proof}
    Let \(S = (x_0, \dots, x_m)\) be a reduced string between \(x = x_0\) and \(y = x_n\) as in the statement, with \(\dist_\X(x,y) \leq \ell(S) + \varepsilon\).
    For each \(i = 1, \dots, m\), denote by \(P_i\) the piece of \(\X\) containing \(x_{i-1}\) and \(x_i\).
    As \(S\) is reduced, the pieces \(P_i\) alternate between vertex spaces and edge cylinders.
    Since \(P_1\) and \(P_m\) are both vertex spaces, this means that \(m\) must be odd.
    We thus have that \(m = 2n-1\) and \(P_{2i-1}\) is the vertex space \(X_{v_{i}}\) for each $i=1, \dots, n$.
    Without loss of generality we may take \(n \geq 3\), for otherwise setting \(K = 2\) gives us the statement trivially.
    
    Let \(y_{v_1}, \dots, y_{v_n}\) be the basepoints of vertex spaces \(X_{v_1}, \dots, X_{v_n}\), and \(D\) the constant given by Lemma~\ref{lem:almost_geodesics_stay_close_to_basepoint} applied to \(S\).
    We will use these points to obtain our estimate for \(n\).
    Take \(s \in P_3\) to be a point on a \(P_3\)-geodesic \([x_2,x_3]\) with 
    \begin{equation}
    \label{eq:dist_s_estimate}
        \dist_{P_3}(s,y_{v_{2}}) \leq C + \delta,    
    \end{equation}
    which exists by Lemma~\ref{lem:basically_thin_triangles}.
    Likewise, there is a point \(t \in P_{m-2}\) on \(P_{m-2}\)-geodesic \([x_{m-3},x_{m-2}]\) with
    \begin{equation}
    \label{eq:dist_t_estimate}
        \dist_{P_{m-2}}(t,y_{v_{n-1}}) \leq C + \delta.    
    \end{equation}
    By the definition of uniform hyperbolicity, we have \(\dist_{P_{2k}}(y_{v_k},y_{v_{k+1}}) = 1\) for each \(k= 1, \dots, n-1\). 

    Recall that every edge space of \(\X\) is isometrically embedded in its adjacent vertex spaces, so that
    \[
        \dist_{P_{2k}}(a,b) = \dist_{P_{2k-1}}(a,b) = \dist_{P_{2k+1}}(a,b) \quad \textnormal{ for all } k = 1, \dots, n-1
    \]
    whenever \(a, b \in P_{2k} = Y \times [0,1]\) are elements with the same second component, where \(Y\) is the edge space corresponding to the cylinder \(P_{2k}\).
    We will use this fact freely below without further reference.

    Observe that \(\dist_{P_1}(x_1,y_{v_1}) \leq \dist_{P_2}(x_1,x_2) + \dist_{P_2}(x_2,y_{v_2}) + \dist_{P_2}(y_{v_1},y_{v_2})\) by the triangle inequality.
    Moreover, \(\dist_{P_2}(x_2,y_{v_2}) \leq \dist_{P_3}(x_2,s) + \dist_{P_3}(s,y_{v_2})\).
    Similar estimates hold for \(\dist_{P_m}(x_{m-1},y_{v_n})\) and \(\dist_{P_{m-2}}(x_{m-2},y_{v_{n-1}})\).
    Combining these with (\ref{eq:dist_s_estimate}) and (\ref{eq:dist_t_estimate}), we obtain
    \begin{equation}
    \label{eq:x1y1_estimate}
    \begin{split}
        \dist_{P_1}(x_1,&y_{v_1}) \leq \dist_{P_2}(x_1,x_2) + \dist_{P_3}(x_2,s) + C + \delta + 1 \quad \textnormal{and}\\
        \dist_{P_m}(x_{m-1},y_{v_n}&) \leq \dist_{P_{m-1}}(x_{m-1},x_{m-2}) + \dist_{P_{m-2}}(x_{m-2},t) + C + \delta + 1.
    \end{split}
    \end{equation}

    By the definition of uniform hyperbolicity, we have that \(\dist_{\X}(y_{v_1},y_{v_n}) \leq \dist_\Gamma(u,v)\).
    Now by the triangle inequality, 
    \[
        \begin{split}
            \dist_{\X}(x_0,x_m) \leq &\dist_{\X}(x_0,x_1) + \dist_{\X}(x_0,y_{v_1}) \\
                &+ \dist_\X(y_{v_1},y_{v_n}) + \dist_{\X}(y_{v_n},x_{m-1}) + \dist_{\X}(x_{m-1},x_m).    
        \end{split}
    \]
    Together with the pair of inequalities in (\ref{eq:x1y1_estimate}), this gives us that
    \begin{equation*}
    \begin{split}
        \dist_{\X}(x_0,x_m) \leq &\dist_{P_1}(x_0,x_1) + \dist_{P_2}(x_1,x_2) + \dist_{P_3}(x_2,s) + \dist_\Gamma(u,v) \\
            &+ \dist_{P_{m-1}}(x_{m-1},x_{m-2}) + \dist_{P_{m-2}}(x_{m-2},t) + 2C + 2\delta + 1.
    \end{split}
    \end{equation*}
    By (\ref{eq:dist_s_estimate}) and Lemma~\ref{lem:almost_geodesics_stay_close_to_basepoint}, we can also see that \(\dist_{P_3}(x_2,s) \leq \dist_{P_3}(x_2,x_3) + C + \delta + D + \varepsilon\).
    Again, a symmetrical estimate holds for \(\dist_{P_{m-2}}(x_{m-2},t)\).
    Therefore, we can deduce from the previous inequality that
    \begin{equation}
    \label{eq:dist_of_endpoints_in_base_vertex}
        \dist_{\X}(x_0,x_m) \leq \ell(S) - \sum_{i=4}^{m-3}\dist_{P_i}(x_{i-1},x_i) + \dist_\Gamma(u,v) + 4C + 4\delta + 2D + 2\varepsilon + 2.
    \end{equation}
    
    On the other hand, \(\ell(S) \leq \dist_{\X}(x_0,x_m) + \varepsilon\) by assumption, so that (\ref{eq:dist_of_endpoints_in_base_vertex}) yields
    \begin{equation}
    \label{eq:sum_of_xi_lens}
        \sum_{i=4}^{m-3}\dist_{P_i}(x_{i-1},x_i) \leq \dist_\Gamma(u,v) + 4C + 4\delta + 2D + 3\varepsilon + 2.
    \end{equation}
    Repeated applications of the reverse triangle inequality to the left hand side give us a lower bound on the left hand side sum:
    \[
        \sum_{i=2}^{n-2}\dist_{P_{2i}}({y_{v_i},y_{v_{i+1}}}) - \dist_{P_3}(x_3,y_{v_2}) - \dist_{P_{m-2}}(x_{m-3},y_{v_{n-1}}) \leq \sum_{i=4}^{m-3}\dist_{P_i}(x_{i-1},x_i).
    \]
    Combined with (\ref{eq:sum_of_xi_lens}) and Lemma~\ref{lem:almost_geodesics_stay_close_to_basepoint}, we have
    \[
        \sum_{i=2}^{n-2}\dist_{P_{2i}}(y_{v_i},y_{v_{i+1}}) \leq \dist_\Gamma(u,v) + 4C + 4\delta + 4D + 5\varepsilon + 2.
    \]
    By definition, the sum on the left is simply equal to \(n-3\).
    Setting \(K = 4C + 4\delta + 4D + 5\) hence gives us the required inequality.
\end{proof}

\subsection{Local metric properties}
We will now explore the consequences of the machinery developed above.
Firstly, we see that the vertex spaces in a uniformly hyperbolic graph of spaces are quasi-isometrically embedded.

\begin{proposition}
\label{prop:vertex_spaces_qi_embed}
    Let \(\X = (\Gamma,X_-,Y_-,\alpha_-)\) be a \((\delta,C)\)-uniformly hyperbolic graph of spaces satisfying Convention~\ref{convention:assumptions_for_hyperbolicity_proof}.
    There is a constant \(k = k(\delta,C) \geq 0\) such that the inclusion \(X_v \hookrightarrow \abs\X\) is a \((1,k)\)-quasi-isometric embedding for all \(v \in V(\Gamma)\).
\end{proposition}

\begin{proof}
    Let \(v \in V(\Gamma)\) and \(x, x' \in X_v\).
    By definition \(\dist_\X(x,x') \leq \dist_{X_v}(x,x')\).
    Suppose then that \(\dist_\X(x,x') < \dist_{X_v}(x,x')\).
    Since \(\abs\X\) is a geodesic space by Proposition~\ref{prop:graph-of-spaces-is-proper-geodesic}, this means that there is a geodesic \(p\) in \(\abs\X\) between \(x\) and \(x'\) which is not contained in \(X_v\).
    It follows as in (\ref{eq:x1y1_estimate}) from the proof of Proposition~\ref{prop:almost_geodesics_give_graph_quasigeodesics} that
    \begin{equation*}
        \dist_{X_v}(x,y_v) + \dist_{X_v}(x',y_v) \leq \ell(p) + 2\delta + 2C + 2,
    \end{equation*}
    where \(y_v\) denotes the basepoint of \(X_v\).
    By the triangle inequality, this implies 
    \[
        \dist_{X_v}(x,x') - 2\delta - 2C - 2 \leq \ell(p) = \dist_{\X}(x,x')
    \]
    so that setting \(k = 2\delta + 2C + 2\) completes the proof.
\end{proof}

Another consequence is that the vertex spaces in a uniformly hyperbolic graph of spaces are uniformly quasiconvex.
This is, of course, an application of Proposition~\ref{prop:almost_geodesics_give_graph_quasigeodesics} to the case that the quasi-geodesic in the underlying graph is a loop.

\begin{proposition}
\label{prop:vertex_spaces_quasiconvex}
    Let \(\X = (\Gamma,X_-,Y_-,\alpha_-)\) be a \((\delta,C)\)-uniformly hyperbolic graph of spaces satisfying Convention~\ref{convention:assumptions_for_hyperbolicity_proof}. 
    Then there is a constant \(\sigma = \sigma(\delta,C) \geq 0\) such that each vertex space of \(\X\) is \(\sigma\)-quasiconvex in $\abs \X$.
\end{proposition}

\begin{proof}
    Let \(D = D(\delta,C) \geq 0\) be the constant of Lemma~\ref{lem:almost_geodesics_stay_close_to_basepoint}, and \(K = K(\delta,C) \geq 0\) the constant of Proposition~\ref{prop:almost_geodesics_give_graph_quasigeodesics}.
    We will show that each vertex space of \(\X\) is \(\sigma\)-quasiconvex with \(\sigma = K + D + 1\).
    
    By Proposition~\ref{prop:graph-of-spaces-is-proper-geodesic}, \(\abs\X\) is a geodesic space.
    Let $v\in V(\Gamma)$ and let \(p \colon I \to \abs\X\) be a geodesic with endpoints \(x, y \in X_v\).
    Proposition~\ref{prop:graph-of-spaces-is-proper-geodesic} tells us that these is a reduced string \(S_p = (x_0, \dots, x_m)\) whose terms lie on \(p\), and \(p\) is a concatenation of geodesics \([x_{i-1},x_i]\) in the interiors of pieces joining points of \(S_p\).
    Let \(v_1, \dots, v_n \in V(\Gamma)\) be the sequence of vertices for which the points of \(S_p\) are contained in \(X_{v_i}\).
    As \(S_p\) is reduced, we have \(m = 2n-1\).
    Of course, \(v_1 = v_n = v\), so by Proposition~\ref{prop:almost_geodesics_give_graph_quasigeodesics}, we have \(n \leq K\).

    Write \(P_1, \dots, P_m\) for the pieces of \(S_p\), so \(X_{v_i} = P_{2i-1}\) for \(i = 1, \dots, n\).
    Let \(y_3, \dots, y_{m-2}\) be the points obtained from Lemma~\ref{lem:almost_geodesics_stay_close_to_basepoint}.
    Write \(Z_{e_i} = Y_{e_i} \times [0,1]\) for the edge cylinder such that \(\iota(e_i) = v_i\) and \(\tau(e_i) = v_{i+1}\) for \(i = 1, \dots, n-1\).
    We write \(y_1 = \alpha_{\bar{e_1}}(y_{e_1})\), where \(y_{e_1} \in Y_{e_1}\) is provided by uniform hyperbolicity.

    Now let \(i = 1, \dots, m\) and let \(z\) be a point on the geodesic subsegment \([x_{i-1},x_i]\) of \(p\).
    If \(i = 1\) or \(i = m\) then \(z \in X_v\) trivially, so suppose otherwise.
    If \(i\) is odd, then \(z\) is contained in a \(\delta\)-neighbourhood of \([x_{i-1},y_i] \cup [x_i,y_i]\) by uniform hyperbolicity of \(\X\).
    Moreover, by Lemma~\ref{lem:almost_geodesics_stay_close_to_basepoint}, \(x_i\) is at distance of at most \(D\) from \(y_i\) if \(i \leq m-3\), and similarly \(x_{i-1}\) from \(y_i\) if \(i \geq 4\).
    Hence, in these cases, \(z\) is contained in a \(\sigma\)-neighbourhood of \(X_v\).
    As the edge cylinders are equipped with the \(\ell^2\)-metric and have width \(1\), this inclusion also holds when \(i\) is even and in the same range.
    Finally, the cases that \(i = 3\) and \(i = m-2\) are dealt with by noting that \([x_2,y_3]\) and \([x_{m-2},y_{m-2}]\) are contained in a \(1\)-neighbourhood of \(X_v\), since they are contained in edge pieces \(P_2\) and \(P_{m-1}\) which are adjacent to \(P_1 = P_m = X_v\).
    This proves that \(X_v\) is \(\sigma\)-quasiconvex, as required.
\end{proof}

\begin{remark}
\label{rem:vertex_space_plus_link_quasiconvex}
    With only superficial modifications to the proof, the above also yields that the subspace of \(\abs\X\) consisting of a vertex space \(X_v\) and its adjacent edge cylinder -- that is, \(B_\X(X_v;1)\) -- is also \(\sigma\)-quasiconvex.
    Increasing the value of \(\sigma\) by 1, the same proof shows that \(X_u \cup Z_e \cup X_v\) is \(\sigma\)-quasiconvex when \(X_u\) and \(X_v\) are adjacent vertex spaces joined by edge cylinder \(Z_e\).
\end{remark}

We are also able to analyse the intersections of neighbourhoods of pairs of pieces.
We consider the cases that both the pieces in question are components of an edge space, or are both vertex spaces. 

\begin{lemma}
\label{lem:intersection_of_edge_piece_nbhds}
    Let \(\X\) be a \((\delta,C)\)-uniformly hyperbolic graph of spaces.
    For any \(r \geq 0\), there is \(R = R(r,\delta,C) \geq 0\) such that for any distinct edge pieces \(Z_e\) and \(Z_{e'}\) of \(\X\), the set
    \[
        B_\X(Z_e;r) \cap B_\X(Z_{e'};r)
    \]
    has diameter at most \(R\).
\end{lemma}

\begin{proof}
    Let us first observe that the lemma reduces to the case that \(Z\) and \(Z'\) are both adjacent to the same vertex space \(X\).
    Indeed, any path joining a point in \(\abs{X}\) to a point in an edge cylinder must pass through a vertex space adjacent to this edge cylinder.
    It follows that \(B_\X(Z;r) \cap B_\X(Z';r) \subseteq B_\X(Z;r) \cap B_\X(Z'';r)\), where \(Z''\) shares an adjacent vertex space with \(Z\), since this observation applies to a path of length \(r\) joining a point in the former intersection to \(Z\).
    Hence without loss of generality suppose that \(Z\) and \(Z'\) are such.
    
    Let \(x\) and \(y\) be points in the given intersection.
    Then there are points \(x_Z, y_Z \in Z \cap X\) and \(x_{Z'}, y_{Z'} \in Z' \cap X\) at most \(r\) from \(x\) and \(y\) respectively.
    We can deduce
    \begin{equation}
    \label{eq:dist_xe_ye}
        \dist_{\X}(x,y) \leq \dist_\X(x_Z,y_Z) + 2r
    \end{equation}
    by the triangle inequality.

    Now uniform hyperbolicity gives us that there is \(z \in Z \cap Z'\) with \(\langle y_Z,y_{Z'}\rangle_{z} \leq C\).
    By assumption, we have \(\dist_{\X}(x_Z,x_{Z'}) \leq 2r\) and \(\dist_{\X}(y_Z,y_{Z'}) \leq 2r\).
    Applying Proposition~\ref{prop:vertex_spaces_qi_embed}, there is \(k = k(\delta,C) \geq 0\) such that the above implies \(\dist_{X}(x_Z,x_{Z'}) \leq 2r + k\) and \(\dist_{X}(y_Z,y_{Z'}) \leq 2\lambda r\).
    Combining these two inequalities implies
    \[
        \dist_{X}(x_Z,z) + \dist_{X}(x_{Z'},z) \leq 2C + 2r+k.
    \]
    An identical equation holds for \(y_Z\) and \(y_{Z'}\) in place of \(x_Z\) and \(x_{Z'}\) respectively.
    Hence, all four points \(x_Z, x_{Z'}, y_Z,\) and \(y_{Z'}\) lie in \(B_{X}(z;2C+2r+k)\).
    In particular,
    \begin{equation}
    \label{eq:dist_xe_ye_in_vertex}
        \dist_{X}(x_Z,y_Z) \leq 4C + 4\lambda r.
    \end{equation}

    Observing that \(\dist_{\X}(x_Z,y_Z) \leq \dist_{X}(x_Z,y_Z)\), equations (\ref{eq:dist_xe_ye}) and (\ref{eq:dist_xe_ye_in_vertex}) give us
    \[
        \dist_{\X}(x,y) \leq 4C + 4\lambda r + 2r,
    \]
    so that setting \(R = 4C + 4\lambda r + 2r\) completes the lemma.
\end{proof}

\begin{lemma}\label{lem:intersections-of-nbhds-of-vertex-spaces}
    Let \(\X = (\Gamma,X_-,Y_-,\alpha_-)\) be a \((\delta,C)\)-uniformly hyperbolic graph of spaces.
    For any \(r \geq 0\), there is \(R = R(r,\delta,C) \geq 0\) such that for any distinct vertex pieces \(X_u\) and \(X_v\) of \(\X\), then either 
    \begin{enumerate}
        \item \(\dist_\Gamma(u,v) = 1\) and any connected component of \(B_\X(X_u;r) \cap B_\X(X_v;r)\) has diameter at most \(R\) or is      contained in \(B_\X(Z_e;R)\), where \(e \in E(\Gamma)\) has endpoints \(u\) and \(v\) and \(Z_e \subseteq Y_e\) is an edge         cylinder; or
        \item \(\dist_\Gamma(u,v) > 1\) and \(B_\X(X_u;r) \cap B_\X(X_v;r)\) has diameter at most \(R\).
    \end{enumerate} 
\end{lemma}

\begin{proof}
    Let \(z \in B_\X(X_u;r) \cap B_\X(X_v;r)\).
    Observe that there are paths \(p_u\) and \(p_v\) of length at most \(r\) joining \(z\) to a point of \(X_u\) and a point of \(X_v\) respectively.
    For the first statement, if either \(p_u\) or \(p_v\) intersect \(Z_e\), then \(z \in B_{\X}(Z_e;r)\) and we are done.
    Suppose otherwise, that then either \(p_u\) or \(p_v\) (or both) must pass through two distinct edge cylinders \(Z\) and \(Z'\).
    But then \(z\) belongs to a path component of \(B_\X(Z;r) \cap B_\X(Z';r)\), which has diameter at most \(R = R(r,\delta,C)\) by Lemma~\ref{lem:intersection_of_edge_piece_nbhds}.

    For the latter case, as \(\dist_{\Gamma}(u,v) > 1\), the paths \(p_u\) and \(p_v\) must pass through at least two distinct edge cylinders \(Z\) and \(Z'\).
    It follows that \(z \in B_{\X}(Z;r) \cap B_\X(Z';r)\), whence again applying Lemma~\ref{lem:intersection_of_edge_piece_nbhds} bounds the diameter by \(R\).
\end{proof}


\section{Strong QI-rigidity for graph of spaces}\label{sec:bottlenecks}

In this section we describe a series of additional technical properties pertaining to graphs of spaces, which will provide strong control over their quasi-isometry groups. 
The purpose of these properties is to guarantee strong QI-rigidity for the realisation of a graph of spaces; the results are summarised in the statement of Theorem~\ref{thm:conclusion-of-rigidity-section}.

\subsection{Link bottlenecked pairs}

The first condition we describe allows us to control what happens to vertex spaces under a quasi-isometry of graphs of spaces \(\X\) and \(\Y\).
Roughly speaking, the condition states that the edge spaces of \(\Y\) do not coarsely separate vertex spaces of \(\X\) into arbitrarily deep components.
Before we define it formally, it will be helpful introduce the following notation. 

\begin{definition}[Link of a vertex space]
    Let $\X = (\Gamma, X_-, Y_-, \alpha_-)$ be a graph of spaces. Let $v \in V(\Gamma)$. Then the \emph{$\X$-link} of $v$, denoted $\mathrm {lk}_\X(v)$, is defined as the space
    $$
    \mathrm {lk}_\X(v) := \bigcup_{e \in \mathrm{In(v)}} \alpha_e(Y_e),
    $$
    equipped with the subspace metric induced by the metric $\dist_{X_v}$ on $X_v$.
\end{definition}

\begin{definition}[Link bottlenecked pairs]
    Let 
    $$\X = (\Gamma, X_-, Y_-, \alpha_-), \ \ \Y = (\Delta, X'_-, Y'_-, \beta_-)
    $$
    be graphs of spaces. We say that the ordered pair $(\X, \Y)$ is \emph{link bottlenecked} if:
    \begin{equation}\tag{LB}\label{eq:LB}
        \parbox{4.3in}{
        for all $\lambda \geq 1$ and $c \geq 0$, there are $K \geq 0$ and $\rho \geq 0$ 
        such that for all $v \in V(\Gamma)$ and all subsets $Z \subset X_v$, if there exists $u \in V(\Delta)$ such that $Z$ admits a $(\lambda, c)$-quasi-isometric embedding into $\lk_\Y(u)$, then there exists a unique $\rho$-coarse component $U \subset X_v - Z$ with $U \not\subset B_{X_v}(Z; K)$.
        }
    \end{equation}
    If the pair $(\X, \X)$ is link bottlenecked, then we simply that say \emph{$\X$ is link bottlenecked}.    
    If both $\X$ and $\Y$ are link bottlenecked, and both pairs $(\X,\Y)$ and $(\Y,\X)$ are link bottlenecked, we call $(\X,\Y)$ a \emph{totally link bottlenecked pair}.
\end{definition}

\begin{example}
    We describe some basic examples. 
    \begin{enumerate}
        \item Consider a graph of spaces $\X$ where every vertex space is isometric to the Euclidean plane $\R^2$, and every $\X$-link is uniformly bounded in diameter. It is easy to verify that $\X$ satisfies (\ref{eq:LB}). 
    
        \item Given any graph $\Gamma$, one can form a trivial graph of spaces $\X$ where each vertex and edge space is just a single point, and so trivially $\abs\X$ is isometric to $\Gamma$. This is generally \textbf{not} link bottlenecked.
    \end{enumerate}
    
\end{example}

\begin{remark}\label{rmk:totally-lbp}
    The condition (\ref{eq:LB}) depends only on the geometry of the vertex spaces and their links. 
    Thus, if $\X$, $\Y$ are graphs of spaces such that all vertex spaces of $\X$ and $\Y$ are pairwise isometric, and all links across both $\X$ and $\Y$ are pairwise isometric, then the pair $(\X, \Y)$ is totally link bottlenecked if and only if $\X$ and $\Y$ are link bottlenecked. 
\end{remark}

We record some basic and useful consequences of being link bottlenecked.

\begin{lemma}\label{lem:vertex-space-one-ended}
    Let $\X = (\Gamma, X_-, Y_-, \alpha_-)$ be a link bottlenecked graph of spaces, where $\Gamma$ is not a single point. Then each vertex space $X_v$ is one-ended. 
\end{lemma}

\begin{proof}
    If $X_v \subset \abs\X$ were bounded then it is trivially quasi-isometric to a subset of some link, which contradicts (\ref{eq:LB}). 
    Thus, no vertex space is bounded. 
    Let $\rho$ be the constant given by the definition~\eqref{eq:LB} for the coarseness of components.
    Suppose then that $X_v$ is unbounded and has more than one end. 
    Then there exists a bounded subset $Z \subset X_v$ such that $X_v - Z$ contains two $\rho$-coarse 
    components containing points arbitrarily far from \(Z\). 
    This also contradicts (\ref{eq:LB}). 
\end{proof}

\begin{lemma}\label{lem:LB-implies-vertex-space-bigger-than-link}
    Let $\X = (\Gamma, X_-, Y_-, \alpha_-)$ be a link bottlenecked graph of spaces. Let $v \in V(\Gamma)$. Then 
    $
    \dHaus_\X(X_v, \lk_\X(v)) = \infty
    $.
    
    Moreover, if $u, v \in V(\Gamma)$ are distinct vertices then
    $
    \dHaus_\X(X_v, X_u) = \infty
    $.
\end{lemma}

\begin{proof}
    If these were not the case then we would immediately contradict (\ref{eq:LB}).
\end{proof}

We now give a technical, more quantitative version of Lemma~\ref{lem:LB-implies-vertex-space-bigger-than-link}, which states that (\ref{eq:LB}) implies that given any point in a vertex space, there is always a uniformly nearby point which is far away from the link.

\begin{lemma}\label{lem:not-trapped-by-link}
    Let $\X = (\Gamma, X_-, Y_-, \alpha_-)$ be a link bottlenecked graph of spaces.
    Then for all $D > 0$, there exists $r = r(D, \X)> 0$ such that for all $v \in V(\Gamma)$ and all $x \in X_v$ there exists $y \in X_v$ such that
    \begin{enumerate}
        \item $\dist_{X_v}(x,y) \leq r$,

        \item $\dist_{X_v}(y, \lk_\X(v)) > D$. 
    \end{enumerate}
\end{lemma}

\begin{proof}
    Suppose this were not the case. Then this means that there exists $D > 0$ such that for all $r > 0$, there exists $v \in V(\Gamma)$ and $x \in X_v$ such that for all $y \in B_{X_v}(x;r)$, we have that $\dist_{X_v}(y,\lk_\X(v)) \leq D$.
    Let $\rho = \rho(1,2D)$ be the constant given by the definition~\eqref{eq:LB}.
    For \(r > \rho \), consider the set
    $$
    S_r = \{z \in X_v : \dist_{X_v}(x,z) \in [r-\rho, r] \}.
    $$
    By our assumption, any closest point projection of $S_r$ to $\lk_\X(v)$ is a $(1,2D)$-quasi-isometric embedding into $\lk_\X(v)$. 
    However, $x$ lies in a bounded $\rho$-coarse component of $X_v - S_r$. 
    Since $X_v$ is unbounded by Lemma~\ref{lem:vertex-space-one-ended}, this contradicts (\ref{eq:LB}) if $r$ is chosen to be sufficiently large with respect to \(K(1,2D)\).  
\end{proof}

The key consequence of property (\ref{eq:LB}) is that the images of vertex spaces under quasi-isometries of link bottlenecked graphs of spaces are uniformly close to vertex spaces in the codomain.

\begin{lemma}\label{lem:vertex-spaces-map-to-vertex-spaces}
    Let $\X = (\Gamma, X_-, Y_-, \alpha_-)$, $\Y = (\Delta, X'_-, Y'_-, \beta_-)$ be uniformly hyperbolic graphs of spaces. 
    Suppose further that the pair $(\X,\Y)$ is totally link bottlenecked. Then for all $\lambda \geq 1$, $c \geq 0$, there exists $C = C(\lambda, c,\X,\Y) \geq 0$ such that for all $(\lambda, c)$-quasi-isometries $\varphi \colon \abs\X \to \abs\Y$ and all $v \in V(\Gamma)$, there exists a unique $u \in V(\Delta)$ such that 
    $$
    \dHaus_\Y(X'_u, \varphi(X_v)) \leq C. 
    $$
    In particular, every quasi-isometry \(\varphi \colon \abs\X \to \abs\Y\) induces a map \(\Phi(\varphi) \colon V(\Gamma) \to V(\Delta)\) on the vertices of the underlying graphs of \(\X\) and \(\Y\).
\end{lemma}

\begin{proof}
    Let $\psi \colon \Y \to \X$ be a quasi-inverse to $\varphi$, so \(\dist_\infty(\id_\X, \psi\varphi), \dist_\infty(\id_\Y,\varphi\psi) \leq Q\) where \(Q\) depends only on \(\lambda\) and \(c\).
    After possibly increasing \(c\), we may assume \(\psi\) is also a \((\lambda,c)\)-quasi-isometry.
    Since $\X$ and $\Y$ are uniformly hyperbolic, Proposition~\ref{prop:vertex_spaces_quasiconvex} gives us a constant $k \geq 0$ such that for all $v \in V(\Gamma)$, $v' \in V(\Delta)$, the inclusions $X_v \into \abs\X$ and $X'_{v'} \into \abs\Y$ are $(1, k)$-quasi-isometric embeddings.
    
    We first show that $\varphi(X_v)$ lands arbitrarily deep inside vertex spaces of $\Y$. 

    \begin{claim*}
        For each $a > 0$, there is $u \in V(\Delta)$ such that 
        $$
        \varphi(X_v) \cap \left(X_u' - B_{X_u'}(\lk_\Y(u); a) \right) \neq \emptyset.
        $$
    \end{claim*}

    \begin{proof}[Proof of claim.]
        For any \(x \in X_v\), there is $x' \in \abs\Y$ in a vertex space, say $x' \in X'_u$, such that $\dist_\Y(x', \varphi(x)) \leq 1$ (projecting to one side of an edge cylinder if necessary). 
        Let $r = r(a + c + k, \Y)$ be as in Lemma~\ref{lem:not-trapped-by-link} so that, given \(x' \in X'_u\), there is $y' \in X_u'$ with
        \begin{equation*}
            \dist_{X_u'}(x',y') \leq r \quad \textnormal{ and } \quad \dist_{X_u'}(y', \lk_\Y(u)) > a + c + k. 
        \end{equation*}
        Moreover, any \(y' \in \abs\Y\) is a distance at most \(c\) from the image of \(\varphi\), as \(\varphi\) is \(c\)-coarsely surjective.
        Thus by the triangle inequality and the fact that the inclusion $X'_{u} \into \abs\Y$ is a $(1, k)$-quasi-isometric embedding, for any \(x \in X_v\) there is \(y \in \abs\Y\) with
        \begin{equation}
        \label{eq:props_of_y_given_x}
            \dist_{X_u'}(y, \lk_\Y(u)) > a \quad \textnormal{ and } \dist_{\Y}(y,\varphi(x)) \leq 1 + r + c.
        \end{equation}
        
        We show that for \(x \in X_v\) far enough from \(\lk_\X(v)\), the \(y\) obtained as above lies in \(\varphi(X_v)\).
        To this end, take  \(R = \lambda(1+r+2c) + k\).
        By Lemma~\ref{lem:LB-implies-vertex-space-bigger-than-link}, there is $x \in X_v$ such that 
        \(
            \dist_{X_v}(x, \lk_\X(v)) > R.
        \)
        
        Let $z \in \varphi^{-1}(y)$. 
        We verify that $\dist_\X(x,z) < \dist_\X(x, \lk_\X(v))$, which implies that $z \in X_v$. 
        Using that \(\varphi\) is a quasi-isometry and (\ref{eq:props_of_y_given_x}),
        \[
            \dist_\X(x, z) \leq \lambda \left( \dist_\Y(\varphi(x),y) + c\right) \leq \lambda (1 + r + 2c) = R - k.
        \]
        To conclude, again note that since \(X_v \hookrightarrow \abs\X\) is a \((1,k)\)-quasi-isometric embedding, the above equation implies
        \[
            \dist_\X(x,z) \leq R - k < \dist_{X_v}(x, \lk_\X(v)) - k \leq \dist_\X(x, \lk_\X(v))
        \]
        as required.
    \end{proof}

    Next, we use (\ref{eq:LB}) to show that $\varphi(X_v)$ cannot contain points too far away from $X_u'$. 
    This is one of the two needed inclusions to prove the lemma.
    
    \begin{claim*}
        There exists $A = A(\lambda, c, \X,\Y) \geq 0$ and \(u \in V(\Delta)\) such that
        \[\varphi(X_v) \subset B_{\Y}(X'_u;A).\] 
        Moreover, the vertex \(u\) is unique.
    \end{claim*}

    \begin{proof}[Proof of claim.]
        Suppose otherwise, so that for any \(A \geq 0\) and \(u \in V(\Delta)\), there is \(y \in X_v\) with \(\dist_\Y(\varphi(y),X'_u) > A\).
        By the first claim, there is $u \in V(\Delta)$ such that $\varphi(X_v) \cap \left(X_u' - B_{X_u'}(\lk_\Y(u); A) \right) \neq \emptyset$. 
        That is, there is $x \in X_v$ such that $\varphi(x) \in X_u'$ and $\dist_{\Y}(\varphi(x), \lk_\Y(u)) > A$. 
        By definition, we have that $\lk_\Y(u)$ separates $\varphi(x)$ from $\varphi(y)$. 
        In particular, for any $\kappa\geq 0$ we know that $\varphi(x), \varphi(y)$ are in different $\kappa$-coarse components of $\abs\Y -  B_\X(\lk_\Y(u),\kappa)$.  
        Now applying Lemma~\ref{lem:effective_coarse_separation},
        we obtain constants $\rho=\rho(\lambda, c, \kappa)$
        and $L > 0$, depending only on $\kappa, \lambda$ and $c$, such that the set 
        \[
            B = B_\X(\psi(\lk_\Y(u)); L)
        \]
        separates $\psi\varphi(x)$ from $\psi\varphi(y)$ in $\abs\X$, so long as \(A \geq \lambda L + \lambda c\). 
        Setting $Z = B \cap X_v$, then, we have that $x$ and $y$ lie in distinct $(\rho+2Q)$-coarse components 
         of $X_v - Z$, when equipped with the subspace metric induced from~$\abs\X$, as \(\dist_\infty(\psi\varphi,\id_\X) \leq Q\).

        Recall that all vertex spaces in $\X$ and $\Y$ are $(1, k)$-quasi-isometrically embedded in \(\abs\X\) and \(\abs\Y\).
        By Lemma~\ref{lem:effective_coarse_separation}, we know that there exists $\rho^\prime=\rho^{\prime}(1, \rho+2Q+k)$ such that $x, y$ lie in different $\rho^\prime$-coarse components of $X_v - Z$, when equipped with the metric of~$X_v$. Note that~$\rho^\prime$ diverges as~$\kappa$ tends to~$\infty$.
        
        Moreover, $Z$  as a subspace of \((X_v,\dist_{X_v})\), admits a 
        $(\lambda, c+\lambda k)$-quasi-isometric embedding into $\lk_\Y(u)$.
        Let $K = K(\lambda, c + \lambda k)$ be as in the definition of property (\ref{eq:LB}). 
        Now if \(A \geq \lambda(K+L) + c\), we have
        \begin{align*}
            \dist_{X_v}(x, Z) &\geq \dist_{\X}(x, B) \\
            &\geq \dist_{\X}(x, \psi(\lk_\Y(u))) - L \\
            &\geq \frac 1 \lambda (\dist_{\Y}(\varphi(x), \lk_\Y(u)) - c ) - L\\
            &> \frac 1 \lambda (A - c ) - L \geq K
        \end{align*}
        Identical reasoning also shows that $\dist_{X_v}(y, Z) > K$.

    Observe that we can do this for any~$\kappa\geq 0$, so that there is no~$\rho^\prime$ such that at most one of the $\rho^\prime$-coarse components of~$X_v-Z$ is at bounded distance from~$Z$. 
        This contradicts property (\ref{eq:LB}), and thus the first statement in the claim is shown.
        That the vertex is unique follows immediately from Lemma~\ref{lem:LB-implies-vertex-space-bigger-than-link}.
    \end{proof}

    Finally, we conclude the proof of the lemma by symmetry.
    Let \(A' = A(\lambda, c, \Y,\X)\) and \(v' \in V(\Gamma)\)
    be as in second claim applied to \(\psi \colon \abs\Y \to \abs\X\), so that
    \[
        \psi(X_u') \subset B_{\X}(X_{v'};A').
    \]
    We first argue that $v' = v$. 
    Indeed, it follows that \(\psi\varphi(X_v)\) is contained in a bounded neighbourhood of \(X_{v'}\).
    Since \(\psi\) is quasi-inverse to \(\varphi\), this implies \(X_v\) is contained in a bounded neighbourhood of \(X_{v'}\).
    Hence \(v = v'\) by Lemma~\ref{lem:LB-implies-vertex-space-bigger-than-link}.

    Now, as \(\varphi\) is a \((\lambda,c)\)-quasi-isometry, \(\varphi\psi(X'_u) \subseteq B_{\Y}(\varphi(X_v);\lambda A' + c)\).
    Since \(\dist \varphi\psi,\id_\Y) \leq Q\), we have \(\dHaus_\Y(\varphi\psi(X'_u),X'_u) \leq Q\).
    Combining these shows that \(\varphi|_{X_v}\) is \(A''\)-coarsely surjective onto \(X'_u\), where \(A'' = \lambda A' + c + Q\).
    Setting \(C = \max\{A,A',A''\}\) concludes the proof.
\end{proof}

\begin{remark}
    Uniform hyperbolicity of \(\X\) is essential in the above, as it provides a bound on the distortion of the vertex spaces in \(\abs\X\). 
    In principle, one could forgo this assumption, formulating property (\ref{eq:LB}) in terms of the global metric \(\dist_\X\) as opposed to the local metrics \(\dist_{X_v}\). 
    However, such a condition would be unwieldy -- potentially very difficult -- to verify in practice.
\end{remark}

As a consequence of the above, quasi-isometries of a link bottlenecked graph of spaces also coarsely preserve the edge cylinders.

\begin{lemma}\label{lem:qi-almost-fixes-edge-spaces-new}
    Let $\X = (\Gamma, X_-, Y_-, \alpha_-)$ be a uniformly hyperbolic, link bottlenecked graph of spaces.
    Let $\varphi \colon \abs\X \to \abs\X$ be a $(\lambda,c)$-quasi-isometry such that $\Phi(\varphi) = \id_\Gamma$. 
    There is $T = T(\lambda,c, \X) \geq 0$ such that for any $e \in E(\Gamma)$,
    \[
        \dHaus_{\mathbf X} (\varphi(Z_e), Z_e) \leq T,
    \]
    where \(Z_e \subseteq \abs\X\) is the edge cylinder corresponding to \(e\).
\end{lemma}

\begin{proof}
    Let $\varphi \colon \abs\X \to \abs\X$ be a $(\lambda,c)$-quasi-isometry such that $\Phi(\varphi) = \id_\Gamma$.
    Let $Z_e$ be an edge cylinder connecting vertex spaces $X_u$ and $X_v$. 
    Note that, by assumption, $u = \Phi(\varphi)(u)$ and $v = \Phi(\varphi)(v)$.
    We have that 
    \[
        Z_e = B_{\X}(X_u;1) \cap B_{\X}(X_v;1). 
    \]
    And so, since \(\varphi\) is a \((\lambda,c)\)-quasi-isometry,
    \[
        \varphi(Z_e) \subseteq B_{\X}(\varphi(X_u);\lambda + c) \cap B_{\X}(\varphi(X_u); \lambda + c).
    \]
    
    Let $C = C(\lambda,c, \X, \X)$ be as in Lemma~\ref{lem:vertex-spaces-map-to-vertex-spaces}. 
    The lemma tells us that \(\varphi(X_w)\) is a Hausdorff distance of at most \(C\) from \(X_w\) for all \(w \in V(\Gamma)\), as \(w = \Phi(\varphi)(w)\).
    Therefore,
    \[
        \varphi(Z_e)  \subseteq B_{\X}(X_u;\lambda + c + C) \cap B_{\X}(X_v;\lambda + c + C).
    \]
    Now, by Lemma~\ref{lem:intersections-of-nbhds-of-vertex-spaces}, the set on the right-hand side is contained in $B_{\X}(Z_e;K)$, for some $K > 0$ depending only on $\lambda$, $c$, and the uniform hyperbolicity constants of \(\X\). 
    Thus the above inclusion yields 
    \[
        \varphi(Z_e) \subseteq B_{\X}(Z_e;K).
    \]
    Applying the same argument to a quasi-inverse $\psi$ of $\varphi$ shows that $\dHaus_{\mathbf X} (\varphi(Z_e), Z_e)$ is bounded by an constant $T \geq 0$ depending only on $\lambda$, $c$, and $\X$.
\end{proof}

We can say more about the function \(\Phi\) from Lemma~\ref{lem:vertex-spaces-map-to-vertex-spaces}.

\begin{lemma}\label{lem:edge-spaces-map-to-edge-spaces}
    Let $\X$, $\Y$ be a totally link bottlenecked pair of uniformly hyperbolic graphs of spaces. 
    For any quasi-isometry $\varphi \colon \abs\X \to \abs\Y$, the map $\Phi(\varphi) \colon V(\Gamma) \to V(\Delta)$ is a bijection. 
    Moreover, if $\X$ and \(\Y\) have unbounded edge spaces, then $\Phi(\varphi)$ is a graph isomorphism and \(\Phi\) is a group homomorphism \(\QI(\X,\Y) \to \Isom(\Gamma,\Delta)\).
\end{lemma}

\begin{proof}
    We first show that $\Phi(\varphi)$ is surjective. 
    If not, $\varphi$ cannot be coarsely surjective by Lemma~\ref{lem:LB-implies-vertex-space-bigger-than-link}. 
    Next, note that $\Phi(\varphi)$ must also be injective, as otherwise two vertex spaces would lie at finite Hausdorff distance, which again contradicts Lemma~\ref{lem:LB-implies-vertex-space-bigger-than-link}. 

    Now, suppose $\X$ has unbounded edge spaces. 
    Then by Lemma~\ref{lem:intersections-of-nbhds-of-vertex-spaces} we have that the intersection of any tubular neighbourhoods of any two distinct vertex spaces $X_u$ and $X_v$ is infinite if and only if $u$ and $v$ are adjacent in $\Gamma$. 
    This implies that $\Phi(\varphi)(v)$ and $\Phi(\varphi)(u)$ are adjacent if and only if $u$ and $v$ are adjacent. 

    The above shows that \(\Phi \colon \QI(\X,\Y) \to \Isom(\Gamma,\Delta)\) is a well-defined function.
    It only remains to check that it is a homomorphism.
    Let \(\varphi,\varphi' \colon \abs\X \to \abs\X\) be quasi-isometries.
    Applying Lemma~\ref{lem:vertex-spaces-map-to-vertex-spaces}, we see that \(\varphi\varphi'(X_v)\) is Hausdorff close to the vertex space \(X_{\Phi(\varphi\varphi'(v))}\) on one hand.
    However, the same lemma implies that \(\varphi'(X_v)\) is close to \(X_{\Phi(\varphi\varphi')(v)}\).
    On the other hand, then, we have \(\varphi\varphi'(X_v)\) is close to \(X_{\Phi(\varphi)\Phi(\varphi')(v)}\).
    Therefore Lemma~\ref{lem:LB-implies-vertex-space-bigger-than-link} gives that \(\Phi(\varphi\varphi')(v) = \Phi(\varphi)\Phi(\varphi')(v)\) for all \(v\in V(\Gamma)\), as required. 
\end{proof}

\begin{corollary}\label{prop:QI-gos-have-iso-graphs}
     Let $\X = (\Gamma, X_-, Y_-, \alpha_-)$, $\Y = (\Delta, X'_-, Y'_-, \beta_-)$ be uniformly hyperbolic graphs of spaces with unbounded edge spaces, such that the pair $(\X, \Y)$ is totally link bottlenecked. 
     If $\X$ and $\Y$ are quasi-isometric then $\Gamma$ and $\Delta$ are isometric. 
\end{corollary}

Let $\iota \colon \Isom(\X) \to \QI(\X)$ denote the natural map, and let us write
$
\Psi = \Phi \circ \iota,
$
where \(\Phi\) is the map of Lemma~\ref{lem:vertex-spaces-map-to-vertex-spaces}.
We have the following commutative diagram:
\begin{equation}\tag{$\dagger$}\label{eq:diagram}
\begin{tikzcd}[cramped]
	{\Isom(\X)} && {\QI(\X)} \\
	& {\Isom(\Gamma)}
	\arrow["\iota", from=1-1, to=1-3]
	\arrow["\Psi"', from=1-1, to=2-2]
	\arrow["\Phi", from=1-3, to=2-2]
\end{tikzcd}
\end{equation}
Throughout the rest of this section, we will continue to refer to $\Psi$, $\Phi$, and $\iota$ as above.

\subsection{Local congruence}

Our second condition provides a criterion for isometries of the underlying graph to lift to isometries of the whole graph of spaces. 

\begin{definition}[Locally congruent]
    Let $\X = (\Gamma, X_-, Y_-, \alpha_-)$ be a graph of spaces. We say that $\X$ is \emph{locally congruent} if the following holds:
    \begin{equation}\tag{LC}\label{eq:LC}
        \parbox{4.3in}{
        for all $f \in \Isom(\Gamma)$, $v \in V(\Gamma)$, there exists an isometry $\tilde f_v : X_v \to X_{f(v)}$ such that for all $e \in \mathrm{In}(v)$ we have that
        $
        \tilde f_v(\alpha_e(Y_e)) = \alpha_{f(e)}(Y_{f(e)})
        $.
        }
    \end{equation}
\end{definition}

Intuitively, a graph of spaces is locally congruent if, given two vertices of the underlying graph in the same automorphic orbit, the corresponding vertex spaces and their links look identical. 

\begin{proposition}\label{prop:LC-gives-surjective}
    Let $\X = (\Gamma, X_-, Y_-, \alpha_-)$ be a locally congruent graph of spaces. Then there exists an injective function     
    $
    \chi \colon \Isom(\Gamma) \into \Isom(\X). 
    $
    
    Further, if $\X$ is uniformly hyperbolic, link bottlenecked, and has unbounded edge spaces, then 
    $
    \Psi \circ \chi = \id_{\Isom(\Gamma)}$.
    In particular, $\Psi$ and $\Phi$ are surjective. 
\end{proposition}

\begin{proof}
    Let $f \in \Isom(\Gamma)$. 
    We construct an isometry $\chi(f) \in \Isom(\X)$ as follows.
    Given $v \in V(\Gamma)$, we define $\tilde f$ on $X_v$ via
    $
    \chi(f) |_{X_v} = i_v \circ \tilde f_v
    $, 
    where $\tilde f_v$ is the isometry instantiated by (\ref{eq:LC}), and $i_v \colon X_v \into \X$ is the natural inclusion.
    Now (\ref{eq:LC}) implies that there is an isometry $Y_e \to Y_{f(e)}$ for any edge $e \in E(\Gamma)$. 
    We use this to extend our definition of $\chi(f)$ to the edge cylinders \(Z_e = Y_e \times [0,1]\), so that 
    \[
        \chi(f)(y,t) = ({\alpha_{f(e)}}^{-1}\tilde f_v \alpha_e(y),t)
    \]
    for any \(e \in E(\Gamma), y \in Y_e,\) and \(t \in [0,1]\). 
    The condition that $\tilde f_v(\alpha_e(Y_e)) = \alpha_{f(e)}(Y_{f(e)})$ ensures that this extension is well-defined as a map \(Z_e \to Z_{f(e)}\). 
    The constructed map $\chi(f)$ is an isometry of $\X$, as it maps preserves the length of strings in \(\X\). 
    We have thus defined a function \(\chi \colon \Isom(\Gamma) \to \Isom(\X)\).

    If \(f\) and \(g\) are distinct isometries of \(\Gamma\), then they disagree on a vertex or edge of \(\Gamma\).
    Hence \(\chi(f)\) and \(\chi(g)\) disagree on a vertex space or edge space of \(\X\).
    That is, \(\chi\) is injective.
    Now suppose that $\X$ is uniformly hyperbolic, link bottlenecked, and has unbounded edge spaces.
    By construction \(\chi(f)(X_v) = X_{f(v)}\), so Corollary~\ref{prop:QI-gos-have-iso-graphs} tells us that $\Psi(\chi(f))(v) = f(v)$, for all \(v \in V(\Gamma)\).
    Therefore \(\Psi \circ \chi = \id_{\Isom(\Gamma)}\) as required.
\end{proof}

\subsection{Local rigidity}

The third and final property we discuss provides control over the kernel of the homomorphism $\Phi \colon \QI(\X) \to \Isom(\Gamma)$. 

\begin{definition}[Link rigidity]
    Let $\X = (\Gamma, X_-, Y_-, \alpha_-)$. We say that $\X$ is \emph{link rigid} if the following assertion holds:
    \begin{equation}\tag{LR}\label{eq:LR}
        \parbox{4.3in}{
        for all $v \in V(\Gamma)$, if $\psi \colon X_v \to X_v$ is an isometry such that 
        \[
            \hspace{-15mm}\dHaus_{X_v}(\alpha_e(Y_e), \psi(\alpha_e(Y_e))) < \infty,
        \]
        for all $e \in \In(v)$, then $\psi=\id_{X_v}$.}
    \end{equation}
\end{definition}

That is to say, in a link rigid graph of spaces, any isometry that coarsely preserves the link of a vertex space is trivial. 
We will appeal to Proposition~\ref{prop:finite-set-with-trivial-pstab-visibility} to ensure this condition in practice.

\begin{proposition}\label{prop:LR-gives-phi-injective}
    Let $\X = (\Gamma, X_-, Y_-, \alpha_-)$ be a uniformly hyperbolic, link bottlenecked, link rigid graph of spaces with unbounded edge spaces. If the vertex spaces of \(\X\) are uniformly strongly QI-rigid, then the map $\Phi \colon \QI(\X) \to \Isom(\Gamma)$ of Lemma~\ref{lem:vertex-spaces-map-to-vertex-spaces} is injective.
\end{proposition}

\begin{proof}
    We will show that the kernel of $\Phi$ is trivial. Let $\varphi \colon \abs\X \to \abs\X$ be a quasi-isometry such that $\Phi(\varphi) = \id_\Gamma$. 
    By Lemma~\ref{lem:vertex-spaces-map-to-vertex-spaces}, there is \(C \geq 0\) such that
    \[
        \dHaus(\varphi(X_v),X_v) \leq C \; \textnormal{ for all } v \in V(\Gamma).
    \]
    Therefore, there is a \((\lambda,c+C)\)-quasi-isometry \(\varphi' \colon \abs\X \to \abs\X\) with \(\dist_\infty(\varphi,\varphi') \leq C\) with the property that \(\varphi'(X_v) \subseteq X_v\) for all \(v \in V(\Gamma)\).
    
    By Proposition~\ref{prop:vertex_spaces_qi_embed}, there is \(k \geq 0\) such that the inclusion $X_v \hookrightarrow \abs\X$ is a \((1,k)\)-quasi-isometric embedding for all \(v \in V(\Gamma)\), where \(k\) depends only on the uniform hyperbolicity constants of \(\X\).
    Hence \(\varphi'\) restricts to a \((\lambda',c')\)-quasi-isometry \(\varphi'_v \colon X_v \to X_v\), where \(\lambda' \geq 1\) and \(c' \geq 0\) depend only on \(\lambda, c\), and the uniform hyperbolicity constants of \(\X\).
    By uniform strong QI-rigidity of~$X_v$, there is an isometry~$f_v\in\Isom(X_v)$ and a constant $M=M(\lambda^\prime, c^\prime, \delta)$ with $\dist_{\infty, X_v}(f, \varphi_v^\prime)\leq M$.
    
    Furthermore, Lemma~\ref{lem:qi-almost-fixes-edge-spaces-new} tells us that there is \(T = T(\lambda',c',\X) \geq 0\) such that \(\dHaus_\X(\varphi'(Z_e),Z_e) \leq T\) for all \(e \in E(\Gamma)\).
    It follows that
    \[
        \dHaus_{X_v}(\alpha_e(Y_e), f_v(\alpha_e(Y_e))) \leq T+M+k \; \textnormal{ for all } e \in \In(v)
    \]
    as well. 
    Applying property~(\ref{eq:LR}), we have $f_v=\id_{X_v}$ and thus $\dist_{\infty, X_v}(\varphi_v^\prime, \id_{X_v})\leq M$ for all \(v \in V(\Gamma)\).
    Again using the fact that each vertex space is \((1,k)\)-quasi-isometrically embedded, we conclude that \(\dist_\infty(\varphi',\id_\X) \leq M+k\), so that \(\varphi'\) is a representative of the trivial class in \(\QI(\X)\).
    Since \(\dist_\infty(\varphi,\varphi') < \infty\), the same is true for \(\varphi\), and so \(\Phi\) is injective.
\end{proof}

We conclude this section by summarising the results with the following theorem statement, which we note includes the statement of Theorem~\ref{thm:strong-qi-rigidity-theorem}. 

\begin{theorem}\label{thm:conclusion-of-rigidity-section}
    Let $\X = (\Gamma, X_-, Y_-, \alpha_-)$ be a uniformly hyperbolic graph of spaces. Suppose that \(\X\) is locally congruent, link bottlenecked, and link rigid. If the vertex spaces of~$\X$ are uniformly strongly QI-rigid, then:
    \begin{enumerate}
        \item\label{itm:sec5-conc-1} $\abs\X$ is strongly QI-rigid;

        \item\label{itm:sec5-conc-2} there is an isomorphism $\Phi \colon \QI(\X) \to \Isom(\Gamma)$;

        \item\label{itm:sec5-conc-3} if $\Y = (\Delta, X'_-, Y'_-, \beta_-)$ is another uniformly hyperbolic graph of spaces such that $\X$ is quasi-isometric to $\Y$ and $(\X, \Y)$ is a totally link bottlenecked pair, then $\Gamma$ is isometric to $\Delta$. 
    \end{enumerate}
\end{theorem}

\begin{proof}
    Let \(\iota, \Phi,\) and \(\Psi\) be as in diagram (\ref{eq:diagram}). 
    By Proposition~\ref{prop:LC-gives-surjective}, we know that $\Psi$ and therefore also $\Phi$ are surjective.
    By Proposition~\ref{prop:LR-gives-phi-injective}, we have that $\Phi$ is injective, and thus $\Phi$ is an isomorphism. 
    It follows then that $\iota$ must be surjective, by the commutativity of (\ref{eq:diagram}).
    The third item follows directly from Corollary~\ref{prop:QI-gos-have-iso-graphs}.
\end{proof}

\begin{remark}\label{rem:uniform_strong_QI_rigidity}
    Any visual proper hyperbolic metric space that is strongly QI-rigid is uniformly strongly QI-rigid (cf. \cite[Lemme 9.11]{pansu1989metriques}).
    This is not a heavy restriction, since any hyperbolic space admitting a cocompact isometric group action is visual.
\end{remark}


\section{Proof of the main theorem}\label{sec:construction}

We now apply the framework of the previous sections to construct explicit examples of spaces and prove our main theorem.
In particular, we construct graphs of spaces whose vertex spaces are copies of the quaternionic hyperbolic plane $\mathbb{H}\mathbf{H}^2$, and whose edge spaces will be geodesic rays.
The properties of Section~\ref{sec:bottlenecks} will be relatively straightforward to establish for these graphs of spaces. 

The quaternionic hyperbolic plane is a uniformly strongly QI-rigid rank-one symmetric space. It is a contractible eight-dimensional manifold with pinched negative sectional curvature \(-4 \leq \kappa \leq -1\).
It may be thought of as a less symmetric version of the hyperbolic plane. 
By Remark~\ref{rem:hyp_constant_cat-1}, it is $\log(3)$-hyperbolic.

We choose the edge space embeddings of our graph of spaces in such a way to ensure (\ref{eq:LR}). 
To do so, we appeal to Proposition~\ref{prop:finite-set-with-trivial-pstab-visibility}.
Let $\Omega$ be the set obtained by applying Proposition~\ref{prop:finite-set-with-trivial-pstab-visibility} to \(\HH^2\), so $\Omega \subset \partial \mathbb{H}\mathbf{H}^2$ has cardinality $|\Omega| = 9$ and has trivial pointwise stabiliser in $\Isom(\HH^2)$. 

Fix some basepoint $x_0 \in \mathbb{H}\mathbf{H}^2$. 
For each $p \in \Omega$, let $\gamma_p$ be a geodesic ray based at $x_0$ and tending to $p$. 
Let $\Gamma$ be any connected, simplicial, vertex-free, 9-regular graph.
We define the graph of spaces $\X_\Gamma = (\Gamma, X_-, Y_-, \alpha_-)$ as follows:
\begin{enumerate}
    \item every vertex space $X_v$ is a copy of the quaternionic hyperbolic plane $\mathbb{H}\mathbf{H}^2$;

    \item every edge space $Y_e$ is a copy of the half-line $[0,\infty)$;

    \item let $\{v_1 : i \in I \ldots\}$ be some choice of orbit representatives for the action of $\Isom(\Gamma)$ on $V(\Gamma)$. For each $i \in I$, fix a bijection $\sigma_i \colon  \In(v_i)\to \Omega $. 

    Now, given any $v = g \cdot v_i \in V(\Gamma)$, where $g \in \Isom(\Gamma)$, and any $e \in \In(v)$, let $\alpha_e : [0,\infty) \to \mathbb{H}\mathbf{H}^2$ denote the (unique) isometry between $[0,\infty)$ and the ray
    $\gamma_p \subset \mathbb{H}\mathbf{H}^2$, where $p = \sigma_i(g^{-1}\cdot e)$. This is well-defined as the action of $\Isom(\Gamma)$ is free, and so the isometry $g$ is unique.
\end{enumerate}
This concludes the construction of $\X_\Gamma$. See Figure~\ref{fig:gos-cartoon-new} for an illustration of part of this space. 
We now verify that \(\X_\Gamma\) satisfies the hypotheses of Theorem~\ref{thm:conclusion-of-rigidity-section}.

\begin{figure}[ht]
    \centering
    \input{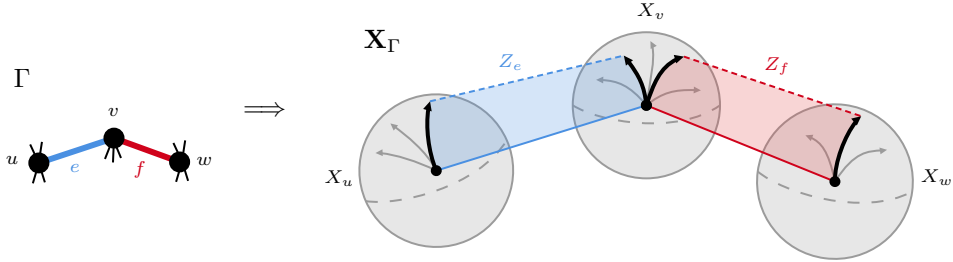}
    \caption{Illustration of a section of our graph of spaces $\X_\Gamma$.}
    \label{fig:gos-cartoon-new}
\end{figure} 

\begin{lemma}
\label{lem:gos_unif_hyp}
    The graph of spaces \(\X_\Gamma\) is uniformly hyperbolic, locally congruent, and has unbounded edge spaces.
\end{lemma}

\begin{proof}
    Following Remark~\ref{rem:hyp_constant_cat-1}, every vertex space of \(\X_\Gamma\) is hyperbolic with hyperbolicity constant \(\delta = \log(3)\).
    Taking \(C = \max\{\langle p,q \rangle_{x_0}  :  p, q \in \Omega\} < \infty\), we see that \(\X_\Gamma\) is \((\delta,C)\)-uniformly hyperbolic, using the point \(0 \in [0,\infty)\) as the basepoint in each edge space.
    By construction, every vertex space and link is an identical isometric copy of the same space, so that \(\X_\Gamma\) is locally congruent.
    Of course, each edge space is \([0,\infty)\) is unbounded.
\end{proof}

\begin{lemma}
\label{lem:gos_LR}
    The graph of spaces \(\X_\Gamma\) is link rigid.
\end{lemma}

\begin{proof}
    Consider an isometry $f \in \Isom(\mathbb{H}\mathbf{H}^2)$. 
    If $f$ coarsely fixes all the incoming edge spaces, then the action of $f$ on the boundary $\partial \mathbb{H}\mathbf{H}^2$ fixes the set $\Omega$ pointwise. 
    By Proposition~\ref{prop:finite-set-with-trivial-pstab-visibility}, we have that $f = \id_{\mathbb{H}\mathbf{H}^2}$. 
    Thus $\X_\Gamma$ is link rigid.  
\end{proof}

\begin{lemma}
\label{lem:gos_LB}
    The graph of spaces \(\X_\Gamma\) is link bottlenecked.
\end{lemma}

\begin{proof}
    Let $S$ denote the metric tree consisting of nine rays with a shared endpoint. It is clear that every $\X_\Gamma$-link is uniformly quasi-isometric to $S$. 
    We write $X = \mathbb{H}\mathbf{H}^2$.
    
    Suppose $Z$ is a set that admits a $(\lambda, c)$-quasi-isometric embedding into $S$.
    The image of \(Z\) separates \(S\).
    We may assume that $Z$ is coarsely connected with a uniform constant. 
    Indeed, if it is not, pass to a subset of \(Z\) whose image is a connected: this will be coarsely connected with constants depending only on \(\lambda\) and \(c\).
    Moreover, since any connected subset of a tree is convex, we may take \(Z\) to be quasiconvex, with constant \(\sigma \geq 0\) depending on \(\lambda, c\), and \(\delta = \log(3)\).
    By Lemma~\ref{lem:effective_coarse_separation}, \(Z\) coarsely separates \(X\).
    Of course, \(\abs{\Lambda Z} \leq 9\).

    As $Z$ is quasiconvex, we know that $\Lambda Z \subset \partial X$ is closed, so that $\Lambda Z$ separates the limit sets of the $\rho$-coarse components $\{V_i\colon i\in I\}$ of $X-Z$. 
    Now $\partial \HH^2$ is homeomorphic to the 7-dimensional sphere~$S^7$, which cannot be separated by the finite set $\Lambda Z$.
    Thus Lemma~\ref{lem:limit_sets_of_coarse_components_are_separated} implies that at most one of the sets $\Lambda V_i- \Lambda Z$ can be non-empty, say $\Lambda V_{i_0}$. 
    Then for all $i\in I-\{i_0\}$, there is \(C_i \geq 0\) with $V_i \subseteq B_X(Z;C_i)$.

    We show that \(C_i \leq 2\sigma\) for all \(i \ne i_0\).
    Suppose otherwise, that $C_i>2\sigma$. Then there is $\varepsilon>0$ and $a\in V_i$ with $\dist_{X}(a, Z)=C_i-\varepsilon > 2\sigma$. 
    Pick a point $x\in Z$.
    Then the geodesic ray~$\gamma$ from~$x$ through~$a$ has endpoint not in $\Lambda V_i \subset \Lambda Z$, since $Z$ is $\sigma$-quasiconvex.
    Thus every point on any geodesic ray with endpoint in~$\partial Z$ is at distance at most~$\sigma$ from~$Z$. 

    We claim that $\gamma \cap Z \subseteq [x,a]$. 
    Suppose otherwise, so there is a point $p\in \gamma \cap Z$ not on \([x,a]\). 
    By $\sigma$-quasiconvexity, there is a point $z\in Z$ with \(\dist_X(a,z) \leq \sigma\).
    This contradicts the fact that $\dist_{X}(a, Z)>2\sigma$.
    It follows that the cofinal segment \(\gamma - [x,a]\) of $\gamma$ is contained in a single coarse component of \(X - Z\), and so $\gamma$ is asymptotic to a point in $\partial X- \Lambda Z$.
    Necessarily then, this coarse component must be~$V_{i_0}$. 
    But \(a\) is in the same \(\rho\)-coarse component as \(\gamma - [x,a]\), a contradiction. 
    It follows that $V_i \subseteq B_X(Z;2\sigma)$, for all \(i \ne i_0\), as required. 
\end{proof}

We are now ready to conclude the proof of our main theorem. 
\main*

\begin{proof}
    Let $G$ be a countable group. By Theorem~\ref{thm:frucht}, there exists an uncountable collection $\mathcal F$ of pairwise non-isomorphic, connected, simplicial, 9-regular, vertex-free graphs $\Gamma$ such that 
    $
    G \cong \Isom(\Gamma)
    $.
    
    For each $\Gamma \in \mathcal F$, form the graph of spaces $\X_\Gamma$ described above. By Remark~\ref{rmk:totally-lbp}, we have that for any two $\Gamma, \Delta \in \mathcal F$, the pair $(\X_\Gamma, \X_\Delta)$ is totally link bottlenecked if \(\X_\Gamma\) and \(\X_\Delta\) are link bottlenecked.
    Thus Lemmas~\ref{lem:gos_unif_hyp}, \ref{lem:gos_LR}, \ref{lem:gos_LB} and Remark~\ref{rem:uniform_strong_QI_rigidity} show that \(\X_\Gamma\) satisfies the hypotheses of Theorem~\ref{thm:conclusion-of-rigidity-section}.
    
    The vertex spaces of \(\X_\Gamma\) are strongly QI-rigid by {\cite[Th\'eor\`eme~1]{pansu1989metriques}}.
    Therefore, by Theorem~\ref{thm:conclusion-of-rigidity-section}, the following hold:
    \begin{enumerate}
        \item if $\X_\Gamma$, $\X_\Delta$ are quasi-isometric then $\Gamma \cong \Delta$;

        \item each $\X_\Gamma$ is strongly QI-rigid;

        \item there is an isomorphism $\QI(\X_\Gamma) \cong \Isom(\Gamma) \cong G$. 
    \end{enumerate}

    Finally, Theorem~\ref{thm:frucht} tells us that if $G$ is a hyperbolic group, then we may take every $\Gamma \in \mathcal F$ as above to be hyperbolic. 
    By Theorem~\ref{thm:hyperbolicity_of_graph_is_equivalent_to_hyperbolicity_of_graph_of_spaces}, $\X_\Gamma$ is hyperbolic for all \(\Gamma \in \mathcal{F}\).  
    This concludes the proof of our theorem.
\end{proof}

\begin{remark}
    The properties we established for \(\X_\Gamma\) are not unique to graphs of spaces whose vertex spaces are copies of \(\HH^2\). 
    Indeed, we could have chosen any strongly QI-rigid rank-one symmetric space \(M\) as the vertex spaces of \(\X_\Gamma\), at the cost of increasing the degree of the the underlying graph~$\Gamma$ based on the dimension of \(M\), as in Proposition~\ref{prop:finite-set-with-trivial-pstab-visibility}. 
\end{remark}

\begin{remark}
    In principle, one could try to apply the above method to realise uncountable groups as quasi-isometry groups.
    Indeed, an analogue of Theorem~\ref{thm:frucht} holds for any group, though the degree of the resulting graphs may have arbitrary cardinality.
    A graph of proper hyperbolic spaces, however, can only be uniformly hyperbolic if the underlying graph is locally finite.
    As uniform hyperbolicity is essential in controlling the geometry of our graphs of spaces, it is not obvious how to extend the method to treat the general case.
\end{remark}


\section{Hyperbolicity of uniformly hyperbolic graphs of spaces}\label{sec:hyperbolicity}

In this section, we will prove that the metric realisation of a uniformly hyperbolic graph of spaces is hyperbolic if and only if the underlying graph is.
We will prove hyperbolicity by verifying the thin triangles condition explicitly. To do this, we need a good understanding of geodesics in uniformly hyperbolic graphs of spaces.

Recall that Proposition~\ref{prop:almost_geodesics_give_graph_quasigeodesics} tells us that geodesics in the graph of spaces project onto quasigeodesics of the graph.  
Building on this, we implement a series of local moves that yield a set of standardised paths which are Hausdorff close to geodesics. 
This reduces thinness of geodesic triangles to thinness of triangles whose sides are given by such standardised paths, a much simpler condition to verify.

\begin{convention}
    Given a graph of spaces \(\X = (\Gamma,X_-,Y_-,\alpha_-)\), we will fix a choice of nearest point projection $\pi_{v,e}\colon X_{v}\to \alpha_e(Y_e)$  of \(X_v\) on $\alpha_e(Y_e)$ for each \(e \in E(\Gamma)\) and vertex \(v \in V(\Gamma)\) with \(v = \tau(e)\). 
\end{convention}

\begin{definition}[Opposite point]
    Let \(Z_e = Y_e \times [0,1]\) be an edge cylinder in \(\X\).
    Given a point \(x = (z_0) \in Z_e\), the \emph{opposite} of \(x\) in \(Z_e\) is the point \((z,1) \in Z_e\), which we denote by \(\overline{x}\).
    Similarly, the opposite of \(x = (z,1)\) is \(\overline{x} = (z,0)\).
\end{definition}

The following lemmas show that geodesics, as they pass through an edge cylinder, fellow-travel paths that are in a sense orthogonal to the edge cylinder.
We first show that every reduced string is at bounded Hausdorff distance from one that travels through edge cylinders parallel to the embedded copy of \(\Gamma\) in \(\X\).
 
\begin{lemma}\label{lemma:any_string_can_be_replaces_by_straight_one_in_edge_cylinder}
    Let \(\X = (\Gamma,X_-,Y_-,\alpha_-)\) be a \((\delta,C)\)-uniformly hyperbolic graph of spaces, and let \(u, v \in V(\Gamma)\) be adjacent vertices, with adjoining edge \(e \in E(\Gamma)\).
    Suppose that \(S = (x_0,x_1,x_2,x_3)\) is a reduced string with \(x_0 \in X_u\) and \(x_3 \in X_v\), and suppose that $\ell(S)=\dist_\X(x_0, x_3)$.
    If \(S' = (x_0,x_1,\overline{x_2},x_3)\), then
    \[\dHaus_{\X}(\operatorname{Path}_\X(S), \operatorname{Path}_\X(S^\prime))\leq 4\delta + 1.\]
\end{lemma}

\begin{proof}
    As the edge cylinder \(Z_e\) is equipped with the \(\ell_2\)-metric and has width 1, we have 
    \(\dHaus_{Z_e}([x_1, x_2], [x_1, \overline{x_2}]\cup[\overline{x_2}, x_2])\leq 1.\)
    Moreover, the first two terms of \(S\) and \(S'\) agree, so the Hausdorff distance bound holds on~$X_u$ and on~$Z_e$.
    It remains to show the bound on \(X_v\).
    By definition, \(\dist_{Z_e}(x_2,\overline{x_2}) = 1\), so that
    \[\ell(S^\prime)=\dist_{X_u}(x_0, x_1)+1+\dist_{X_v}(\overline{x_2}, x_3).\]

    By the triangle inequality, one has $\dist_{Z_e}(\overline{x_2}, x_2)\leq \dist_{Z_e}(x_1, x_2) +1$.
    Thus,
    \begin{align*}
        2\langle \overline{x_2}, x_3 \rangle_{x_2}&\leq
        \dist_{X_v}(x_2, \overline{x_2})+\dist_{X_v}(x_2, x_3)-\dist_{X_v}(\overline{x_2}, x_3) \\
        &\leq \dist_{Z_e}(x_1,x_2)+\dist_{X_v}(x_2, x_3)-\dist_{X_v}(\overline{x_2}, x_3)+1\\
        &= \ell(S)-\ell(S^\prime)+2 \\
        &\leq 2,
    \end{align*}
    because~$S$ was assumed length-minimising.    
    By $\delta$-hyperbolicity of $X_v$ we then have $\dHaus_{X_v}([\overline{x_2}, x_2]\cup [x_2, x_3], [\overline{x_2}, x_3])\leq 4\delta+1$.
    This completes the proof.
\end{proof}

We see that any reduced string that passes transversely through an edge cylinder is a bounded distance from one obtained by taking a nearest point projection.

\begin{lemma}\label{lemma:can_replace_first_with_nearest_point_projection}
    Let \(\X = (\Gamma,X_-,Y_-,\alpha_-)\) be a \((\delta,C)\)-uniformly hyperbolic graph of spaces, \(v \in V(\Gamma)\) and \(e \in E(\Gamma)\) an edge incident to \(v\). 
    Let $x_0 \in X_v$ and $y \in Y_e$. 
    Write $x_1=\alpha_e(y), x_2=\alpha_{\bar{e}}(y), x_1' = \pi_{v,e}(x_0)$. If $S=(x_0, x_1, x_2),$ and $S^\prime = (x_0, x_1', x_2)$, then  
    \[\dHaus_{\X}(\operatorname{Path}_\X(S), \operatorname{Path}_\X(S^\prime))\leq 2\delta+1.\]
\end{lemma}

\begin{proof}
    By Lemma~\ref{lem:triangle_with_projection} we have that $\dHaus_{X_v}([x_0, x_1], [x_0, x_1']\cup [x_1^\prime, x_1])\leq 2\delta$.
    As in the proof of Lemma~\ref{lemma:any_string_can_be_replaces_by_straight_one_in_edge_cylinder}, we have \[\dHaus_{Z_e}([x_1',x_2], [x_1, x_2]\cup[x_1, x_1'])\leq 1.\] 
    It follows that $\dHaus_{\X}(\operatorname{Path}_\X(S), \operatorname{Path}_\X(S^\prime))\leq 2\delta+1$.
\end{proof}

\begin{lemma}\label{lem:geodesics-between-adjacent-vertex-spaces}
    Let \(\X = (\Gamma,X_-,Y_-,\alpha_-)\) be a \((\delta,C)\)-uniformly hyperbolic graph of spaces satisfying the assumptions of Convention~\ref{convention:assumptions_for_hyperbolicity_proof}.
    Suppose that \(p\) is a geodesic and let \(S_p = (x_0, \dots, x_m)\) be its associated reduced string, with pieces \(P_1, \dots, P_m\).

    Suppose that \(P_i = Z_e\) is an edge cylinder with adjacent vertex spaces \(P_{i-1} = X_u\) and \(P_{i+1} = X_v\) for some \(1 < i < m\).
    Let \(x_{i-1}' = \pi_{u,e}(x_{i-2}), x_i' = \pi_{v,e}(x_{i+1})\), and write \(S' = (x_0, \dots, x_{i-2}, x_{i-1}',x_i',x_{i+1}, \dots, x_m)\). 
    Then we have
    \[
        \dHaus_{\X}(p,\operatorname{Path}_\X(S')) \leq 12\delta+4.
    \]
\end{lemma}
\begin{proof}
    For simplicity, we suppose \(m = 3\), so that \(S_p = (x_0,x_1,x_2,x_3)\), \(x_1' = \pi_{u,e}(x_0)\), and \(x_2' = \pi_{v,e}(x_3)\).
    Since the difference between \(S_p\) and \(S'\) is at most two consecutive terms, this assumptions is without loss of generality.
    
    Lemma~\ref{lemma:any_string_can_be_replaces_by_straight_one_in_edge_cylinder} shows that $\dHaus_{\X}(p, \operatorname{Path}_\X (x_0, x_1, \overline{x_2}, x_3))\leq 4\delta+1$.
    By Lemma~\ref{lemma:can_replace_first_with_nearest_point_projection},
    \[\dHaus_{\X}(\operatorname{Path}_\X (x_0, x_1, \overline{x_2}, x_3), \operatorname{Path}_\X (x_0, x_1', \overline{x_2}, x_3))\leq 2\delta+1.\] 
    Again by Lemma~\ref{lemma:any_string_can_be_replaces_by_straight_one_in_edge_cylinder}, we have    
    \[\dHaus_{\X}(\operatorname{Path}_\X (x_0, x_1', \overline{x_2}, x_3),\operatorname{Path}_\X (x_0, x_1', \overline{x_1'}, x_3))\leq 4\delta+1.\]   
    Finally, from Lemma~\ref{lemma:can_replace_first_with_nearest_point_projection} we obtain 
    \begin{align*}
        \dHaus_{\X}(\operatorname{Path}_\X &(x_0, x_1', \overline{x_1'}, x_3),\operatorname{Path}_\X (x_0, x_1', x_2', x_3))\leq 2\delta+1.
    \end{align*}
   Combining the four inequalities above completes the proof.
\end{proof}

We can show that whenever two different geodesic reduced strings start at the same point in a vertex space but travel through different edge cylinders initially, then at least one of them passes close to the base point in that vertex space.

\begin{lemma}\label{lemma:geodesic_strings_starting_in_different_directions_are_close}
    Let \(\X = (\Gamma,X_-,Y_-,\alpha_-)\) be a \((\delta,C)\)-uniformly hyperbolic graph of spaces satisfying the assumptions of Convention~\ref{convention:assumptions_for_hyperbolicity_proof}.
    
    Let \(p\) be a geodesic in \(\abs{\X}\), and write \(S_p = (x_0,  \dots, x_m)\) for the reduced string associated to \(p\), with pieces \(P_i\).
    If \(m \geq 2\) and \(P_1 = X_v\) is a vertex space, then
    \[
        \dHaus_{X_v}([x_0,x_1] \cup [x_1,y_v],[x_0,y_v]) \leq 14\delta+4,
    \]  
    where \(y_v\) is the basepoint of \(X_v\).
\end{lemma}

\begin{proof}
    Let $x_1' = \pi_{v,e}(x_0)$ and write \(S' = (x_0,x_1',x_2, \dots, x_m)\).
    Lemma~\ref{lem:geodesics-between-adjacent-vertex-spaces} tells us that \(\operatorname{Path}_\X(S')\) a Hausdorff distance at most~$12\delta+4$ from~$p$.
    Moreover, by Lemma~\ref{lem:triangle_with_projection}, we have $\dHaus_{X_{v}}([x_0, y_v], [x_0, x_1']\cup[x_1', y_v])\leq 2\delta$.
    Thus,
    \begin{equation*}
        \dHaus_{X_{v}}([x_0, y_v], [x_0, x_1]\cup[x_1, y_v])\leq 14\delta+4, 
    \end{equation*}
    as required.
\end{proof}

We can use Lemma~\ref{lemma:geodesic_strings_starting_in_different_directions_are_close} to show that geodesics that pass through at least two edge cylinders fellow-travel a path through basepoints.

\begin{proposition}
\label{prop:long_paths_close_to_graph}
    Let \(\X = (\Gamma,X_-,Y_-,\alpha_-)\) be a \((\delta,C)\)-uniformly hyperbolic graph of spaces with the assumptions of Convention~\ref{convention:assumptions_for_hyperbolicity_proof}.
    There is a constant \(F = F(\delta,C) \geq 0\) such that the following is true.
    
    Let \(p \colon I \to \abs\X\) be a geodesic with endpoints \(x\) and \(z\). 
    Write \(v_1, \dots, v_n \in V(\Gamma)\) for the sequence of vertices for which the image of \(p\) meets vertex spaces \(X_{v_1}, \dots, X_{v_n}\).
    For each \(i = 1, \dots, n\), let \(y_{v_i}\) be the basepoint of \(X_{v_i}\) given by uniform hyperbolicity.
    If \(p\) intersects at least two edge cylinders, then
    \[
        \dHaus_\X(p, \operatorname{Path}_\X(S)) \leq F,
    \]
    where \(S = (x,y_{v_1},\dots,y_{v_n},z)\).
\end{proposition}

\begin{proof}
    Let \(S_p = (x_0, \dots, x_m)\) be the reduced string associated to \(p\), where \(x_0 = x\) and \(x_m = y\).
    If the first piece of this string is a vertex space, define \(x_0' = x_0\).
    Otherwise \(x_0 = (z,t)\) lies in an edge cylinder \(Z_e\), and \(x_1\) lies in a vertex space \(X_v\) with \(e\) incident on \(v\).
    Let \(u \ne v\) denote the other vertex upon which \(e\) is incident; without loss of generality suppose that \((z,0) \in X_u\) (otherwise,  \((z,1) \in X_u\)).
    We write \(x_0' = (z,0)\) and, similarly define \(x_m'\) if the final piece of \(S_p\) is an edge cylinder.
    Of course, \(\dist_\X(x_0,x_0') \leq 1\) and \(\dist_\X(x_m,x_m') \leq 1\).
    Now the string \(S = (x_0',x_0',x_1, \dots, x_{m-1}, x_m',x_m')\) is reduced and its first and last pieces are vertex spaces.
    We write \(p' = \operatorname{Path}_\X(S)\) so that 
    \begin{equation}
    \label{eq:len_S+4}
        \ell(S) \leq \dist_\X(x_0',x_m') + 4,
    \end{equation}
    and moreover, \(\dHaus_\X(p,p') \leq 1\).

    For convenience, we relabel \(S = (x_0,\dots,x_m)\).
    Since the first and last pieces of \(S\) are vertex spaces, \(m = 2n-1\) is odd.
    For each \(i = 1,\dots,n\), let \(v_i \in V(\Gamma)\) be such that \(x_{2i-2}\) and \(x_{2i-1}\) are contained in \(X_{v_i}\).
    By uniform hyperbolicity and Lemma~\ref{lem:basically_thin_triangles}, there are points \(t\) and \(t'\) on geodesics \([x_2,x_3] \subseteq X_{v_2}\) and \([x_{m-3},x_{m-2}] \subseteq X_{v_{n-1}}\) such that
    \[
        \dist_{X_{v_2}}(t,y_{v_2}) \leq C + \delta \quad \textnormal{ and }\quad \dist_{X_{v_{n-1}}}(t',y_{v_{n-1}}) \leq C + \delta
    \]
    
    Further, applying Lemma~\ref{lem:almost_geodesics_stay_close_to_basepoint} with (\ref{eq:len_S+4}), we have that \(\dist_{X_{v_{i}}}(x_{2i-1},y_{v_i}) \leq D+4\) for \(2 \leq i \leq n-3\) and \(\dist_{X_{v_i}}(x_{2i-2},y_{v_i}) \leq D+4\) for \(4 \leq u \leq n-1\). 
    Thinness of triangles implies that $p'$ remains $(D+\delta)$-close to $\operatorname{Path}_\X (y_{v_2}, \cdots, y_{v_{n-1}})$ on the subpath between $t$ and $t$ (since \(D \geq C+\delta\), in any case).
    By hyperbolicity of $X_{v_2}$ we have 
    \[\dHaus_{X_{v_2}}([\overline{x_2}, t], [\overline{x_2}, y_{v_2}])\leq D+\delta+4\] 
    whereas the geometry of the edge cylinder guarantees \[\dHaus_{ Z_{e_1}}([x_1, y_{v_1}]\cup[y_{v_1}, y_{v_2}] , [x_1, \overline{x_2}]\cup{[\overline{x_2}, y_{v_2}]})=1.\]
    A symmetrical argument shows that analogous inequalities hold involving the final terms of the string.
    Thus applying Lemma~\ref{lemma:geodesic_strings_starting_in_different_directions_are_close} allows us to conclude that
    \begin{equation*}
        \dHaus_{\X}(\operatorname{Path}_\X(S), p')\leq  D + 15\delta + 9.
    \end{equation*}
    As \(\dHaus_\X(p,p') \leq 1\), setting \(F = D + 15\delta + 10\) thus completes the proof.
\end{proof}

The next lemma shows that geodesics whose endpoints lie in edge cylinders are close to geodesics with endpoints in vertex spaces.
Thus, we will only have to check thinness of triangles with vertices in vertex spaces.
We introduce a notation for projections which makes the statement of the following lemma much easier.

\begin{definition}[Projection to vertex space]
    Let \(\X = (\Gamma,X_-,Y_-,\alpha_-)\) be a graph of spaces.
    If \(v \in V(\Gamma)\) is a vertex and \(e \in E(\Gamma)\) is such that \(\tau(e) = v\), we define the \emph{projection from the edge cylinder} \(Z_e = Y_e \times [0,1]\) to \(X_v\) as
    \begin{align*}
        \pi_{e,v}\colon Z_e\to X_{v_i}, (z,t)\mapsto \alpha_{e}(z).
    \end{align*}
    For \(x\) in the interior of an edge cylinder \(Z_e\), we write \(v(z) = \{\pi_{e,u}(x), \pi_{e,v}(x)\}\) and for a point \(x\) in a vertex space, we write \(v(x) = \{x\}\).
\end{definition}

 \begin{lemma}\label{lemma:endpoints-in-edge-cylinder-is-close-to-endpoints-in-vertex-spaces}
    Let~$\X = (\Gamma,X_-,Y_-,\alpha_-)$ be a $(\delta, C)$-uniformly hyperbolic graph of spaces, and let us assume that the assumptions of Convention~\ref{convention:assumptions_for_hyperbolicity_proof} are satisfied.
    Suppose that \(\Gamma\) is \(\delta_\Gamma\)-hyperbolic for some \(\delta_\Gamma \geq 0\) 
    There is a constant \(\Theta = \Theta(\delta,C,\delta_\Gamma) \geq 0\) such that the following is true.
    
    Let \(x, x' \in \abs\X\) and let \(p\) be a geodesic between \(x\) and \(x'\).
    If at least one of \(x\) or \(x'\) lie in an edge cylinder, then for any geodesic \(q\) between a point of \(v(x)\) and \(v(x')\), we have
    \(\dHaus_\X(p,q) \leq \Theta.\)
 \end{lemma}
 
 \begin{proof}
    Suppose that there is \(v \in V(\Gamma)\) such that \(x\) and \(x'\) both lie in \(X_v\) or its adjacent edge cylinders.
    Let \(q\) be a geodesic as in the lemma statement.
    Then \(q\) has endpoints in \(B_{\X}(X_v;1)\), which is \(\sigma\)-quasiconvex following Remark~\ref{rem:vertex_space_plus_link_quasiconvex}, where \(\sigma\) depends only on \(\delta\) and \(C\).
    Therefore \(q\) is contained in \(B = B_\X(X_v;\sigma+1)\).
    
    As \(X_u\) is \(\delta\)-hyperbolic, the set \(B\) is \(\delta'\)-hyperbolic where \(\delta' = \delta + 2\sigma +2\).
    Now applying the Morse Lemma to the \((1,2)\)-quasi-geodesic in \(B\) that follows \(p\) and shares endpoints with \(q\) yields the required statement.

    It remains to consider the cases that \(x\) and \(x'\) are neither in adjacent pieces of \(\X\), nor share any adjacent pieces. 
    It must be that both \(p\) and \(q\) intersect at least two edge cylinders in this case.
    We write \(z \in v(x)\) and \(z' \in v(z')\) for the endpoints of \(q\).
    Let \(X_{u_1}, \dots, X_{u_n}\) be the sequence of vertex spaces intersected by \(p\), and \(X_{v_1}, \dots, X_{v_m}\) those by \(q\).
    We also write \(p' = \operatorname{Path}_\X(x,y_{u_1}, \dots, y_{u_n},x')\) and \(q' = \operatorname{Path}_\X(z,y_{v_1}, \dots, y_{v_m},z')\).
    
    By Proposition~\ref{prop:long_paths_close_to_graph}, there is \(F = F(\delta,C)\) such that
    \begin{equation}
    \label{eq:pp'_qq'}
        \dHaus_\X(p,p') \leq F \quad \textnormal{and} \quad \dHaus_\X(q,q') \leq F.
    \end{equation}
    Moreover, by Proposition~\ref{prop:almost_geodesics_give_graph_quasigeodesics}, \(u_1 \dots u_n\) and \(v_1 \dots v_m\) are \((1,K)\)-quasi-geodesics in \(\Gamma\), where \(K = K(\delta,C) \geq 0\).
    By assumption, \(\Gamma\) is \(\delta_\Gamma\)-hyperbolic.
    Hence we have that \(\dHaus_\Gamma(u_1\dots u_n, v_1\dots v_m) \leq 2M\), where \(M = M(1,K,\delta_\Gamma) \geq 0\) is the constant of the Morse lemma.
    As \(\Gamma\) is isometric to \(\abs\Y\) by Lemma~\ref{lemma:graph_is_isometrically_embedded} and the edge cylinders have unit width, \(\dHaus_\X(p',q') \leq 2M + 1\).
    Combining this with (\ref{eq:pp'_qq'}) and setting \(\Theta = 2F+2M+1\) gives the lemma statement.
\end{proof}

A final useful observation is that whenever two geodesics starting at a point in a vertex space $X_v$ and enter different edge cylinders, then at least one of them must pass close to the basepoint of \(X_v\).

\begin{lemma}\label{lem:geodesics_in_different_directions_one_close_to_base_point}
    Let \(\X = (\Gamma,X_-,Y_-,\alpha_-)\) be a \((\delta,C)\)-uniformly hyperbolic graph of spaces satisfying the assumptions of Convention~\ref{convention:assumptions_for_hyperbolicity_proof}.
    Let \(p, q\) be geodesics in \(\abs{\X}\) with the same initial point $x_0\in X_v$. 
    Let $S_p=(x_0, x_1, \dots, x_m)$ and $S_q=(x_0, x_1^\prime, \dots, x_n')$ be the reduced strings associated to \(p\) and \(q\). 
    Suppose that $x_1$ and $x_1^\prime$ are contained in distinct edge cylinders $Z_e$ and $Z_{e'}$. 
    Then there is a point \(a \in X_v\) lying on \(p\) or \(q\) with
    \begin{equation*}
        \min \{\dist_{X_v}(a, y_v), \dist_{X_v}(b, y_v)\}\leq C+ 5\delta,
    \end{equation*}
    where \(y_v \in X_v\) is the basepoint of \(X_v\).
\end{lemma}
\begin{proof}
    We write \(z = \pi_{v,e}(x_0)\) and \(z' = \pi_{v,e'}(x_0)\).
    Lemma~\ref{lem:triangle_with_projection} give us that 
    \begin{equation}
    \label{eq:bd_inn_prods_of_projs}
        \langle x_0, y_v\rangle_{z}\leq \delta \quad \textnormal{ and } \quad\langle x_0, y_v\rangle_{z'}\leq \delta.
    \end{equation}
    
    By uniform hyperbolicity we know that $\langle z, z' \rangle_{y_v}\leq C$. 
    The four-point condition of Remark~\ref{rem:four_point} thus gives
    \begin{equation*}
        \min \{\langle x_0, z \rangle_{y_v}, \langle x_0,z' \rangle_{y_v}\}\leq C+2\delta.
    \end{equation*}
    Suppose the minimum is attained with the first inner product. 
    Then (\ref{eq:bd_inn_prods_of_projs}) gives that
    \[
        \dist_{X_v}(y_v,z) = \langle x_0,y_v \rangle_z + \langle x_0,z \rangle_{y_v} \leq C + 3\delta.
    \]
    Moreover, Lemma~\ref{lem:triangle_with_projection} tells us that \([x_0,x_1]\) is a Hausdorff distance of at most \(2\delta\) from \([x_0,z] \cup [z,x_1]\).
    Hence there is a point \(a \in [x_0,x_1]\) such that \(\dist_{X_v}(z,a) \leq 2\delta\).
    Combining with the above inequality, we have $\dist_{X_v}(y_v, a)\leq C+ 5\delta$ as required.
    Similarly, if the minimum were attained with the second inner product, we would have \(a \in [x_0,x_1']\) with the same property.
\end{proof}

We are now ready to prove the main result of this section.

\begin{theorem}\label{thm:hyperbolicity_of_graph_is_equivalent_to_hyperbolicity_of_graph_of_spaces}
    Let \(\X = (\Gamma,X_-,Y_-,\alpha_-)\) be a \((\delta,C)\)-uniformly hyperbolic graph of spaces, and suppose that~$\X$ satisfies the assumptions of Convention~\ref{convention:assumptions_for_hyperbolicity_proof}. Then~$|\X|$ is a hyperbolic metric space if and only if its underlying graph~$\Gamma$ is. More precisely:
    \begin{enumerate}
        \item If~\(\Gamma\) is \(\delta_\Gamma\)-hyperbolic, then there exists \(\delta' = \delta'(\delta, C, \delta_\Gamma) \geq 0\) such that \(\abs\X\) is \(\delta'\)-hyperbolic.
        \item If $\abs{\X}$ is $\delta^\prime$-hyperbolic, then so is~$\Gamma$.   
    \end{enumerate}
\end{theorem}

\begin{proof}
    The second statement is a direct consequence of the fact that there exists a convex subspace of~$|\X|$ isometric to~$\Gamma$, as in Lemma~\ref{lemma:graph_is_isometrically_embedded}. 
    The remainder of the proof is dedicated to a proof of the first statement.

    Throughout this proof, let~$D=D(\delta, C)$ be the constant from Lemma~\ref{lem:almost_geodesics_stay_close_to_basepoint}, $\sigma=\sigma(\delta, C)$ the quasiconvexity constant of Proposition~\ref{prop:vertex_spaces_quasiconvex}, and $K=K(\delta, C)$ the constant of Proposition~\ref{prop:almost_geodesics_give_graph_quasigeodesics}.
    Let \(F = F(\delta,C) \geq 0\) be the constant of Proposition~\ref{prop:long_paths_close_to_graph}.
    Recall that \(\abs\Y\) is an isometrically embedded copy of \(\Gamma\) in \(\abs\X\) containing the basepoints of \(\X\), as in Lemma~\ref{lemma:graph_is_isometrically_embedded}.
  
    Assume that~$\Gamma$ is $\delta_\Gamma$-hyperbolic and let~$\Delta$ be a geodesic triangle in~$|\X|$ with vertices $x,y,z$. 
    Let \(\gamma_a\) be the side of \(\Delta\) with endpoints \(x\) and \(z\), \(\gamma_b\) with \(x\) and \(y\), and \(\gamma_c\) with \(y\) and \(z\).

    We distinguish cases depending on the number of edge cylinders traversed by sides of \(\Delta\): write $a, b,$ and~$c$ for the numbers of edge cylinders intersected by the three sides $\gamma_a, \gamma_b, \gamma_c$ of~$\Delta$, respectively.
    By Lemma~\ref{lemma:endpoints-in-edge-cylinder-is-close-to-endpoints-in-vertex-spaces}, there is \(\Theta = \Theta(\delta,C,\delta_\Gamma) \geq 0\) such that the sides of \(\Delta\) are a Hausdorff distance of at most \(\Theta\) from geodesics with endpoints in vertex spaces.
    It thus suffices to consider the case when the vertices of \(\Delta\) are in vertex spaces, say \(x \in X_u, y \in X_v,\) and \(z \in X_w\).
    In each case, we will replace \(\Delta\) with a triangle \(\Delta'\), whose sides are given by the well-behaved paths constructed earlier in the section.
    
    Finally, we will write \(\Delta_\Gamma\) to denote the triangle in \(\abs\Y\) whose sides are the vertices corresponding to the vertex spaces \(\gamma_a, \gamma_b\), and \(\gamma_c\) intersect.
    By Proposition~\ref{prop:almost_geodesics_give_graph_quasigeodesics}, this is a \((1,K)\)-quasi-geodesic triangle in \(\abs\Y\).
    As \(\abs\Y\) is \(\delta_\Gamma\)-hyperbolic, \(\Delta_\Gamma\) is \((\delta_\Gamma + 2M)\)-thin, where \(M = M(1,K,\delta_\Gamma) \geq 0\) is the constant of the Morse lemma.
 
    \medskip
    \underline{The case $a, b, c = 0$:} 
        It must be that \(\Delta \subseteq X_v\) for some \(v \in V(\Gamma)\).
        Since \(X_v\) is \(\delta\)-hyperbolic, \(\Delta\) is \(\delta\)-thin.
    
    \medskip
    \underline{The case $a, b=1, c = 0$:} 
        Assume that \(x\in X_u\) and \(y,z \in X_v\) for some adjacent vertices \(u,v \in V(\Gamma)\), with connecting edge \(e\in E(\Gamma)\).
    
        Observe that $\gamma_a$ and~$\gamma_b$ must both intersect \(Z_e\), since $\Gamma$ is simplicial.
        By Lemma~\ref{lem:geodesics-between-adjacent-vertex-spaces} we know that the Hausdorff distance between \(\gamma_a\) and $\operatorname{Path}_\X(S_{a})$ is at most \(12\delta + 4\), where $S_{a}=(x,\pi_{u,e}(x),\pi_{v,e}(z), z)$.
        Likewise, the Hausdorff distance between \(\gamma_b\) and $\operatorname{Path}_\X(S_{b})$ is at most \(12\delta+4\), where $S_{b}=(x,\pi_{u,e}(x),\pi_{v,e}(y), y)$.

        Now Lemma~\ref{lemma:any_string_can_be_replaces_by_straight_one_in_edge_cylinder} shows that $\operatorname{Path}_\X(S_{a})$ and $\operatorname{Path}_\X(S_{b})$ are at Hausdorff distance at most~$4\delta+1$ from \(\gamma'_{a}\) and \(\gamma'_{b}\) respectively, where $\gamma'_{a}=\operatorname{Path}_\X(x, \pi_{u,e}(x), \overline{\pi_{u,e}(x)}, z)$ and $\gamma'_{b}=\operatorname{Path}_\X(x, \pi_{u,e}(x), \overline{\pi_{u,e}(x)}, z)$. 
        Let \(\Delta'\) be the triangle with sides \(\gamma'_a,\gamma'_b\), and \(\gamma_c\).
        The sides of \(\Delta'\) coincide outside of \(X_v\), so it is \(\delta\)-thin as \(X_v\) is \(\delta\)-hyperbolic.
        Moreover, the sides of \(\Delta'\) are a Hausdorff distance of \(16\delta + 5\) from the sides of \(\Delta\).
        Hence \(\Delta\) is \((33\delta + 10)\)-thin.

    \medskip
    \underline{The case $a\geq 2, b=1, c=0$:} 
        Let \(\Gamma' \subseteq \Gamma\) be the subgraph consisting of \(u\) and \(v\) and the edge that joins them, and let \(\X'\) be the subgraph of spaces with underlying graph \(\Gamma'\).
        Following Remark~\ref{rem:subgraph_of_spaces}, we identify \(\abs{\X'}\) with the corresponding subspace of \(\abs\X\).
        By the \(a,b,c = 0\) and \(a,b = 1, c = 0\) cases treated above, \(\abs{X'}\) is \((33\delta+10)\)-hyperbolic.
        As in Remark~\ref{rem:vertex_space_plus_link_quasiconvex}, \(\abs{\X'}\) is \((\sigma+1)\)-quasiconvex, so that \(\gamma_a\) is contained in its \((\sigma+1)\)-neighbourhood.
        Hence \(\Delta\), being a geodesic triangle in its \((\sigma+1)\)-neighbourhood, is \((33\delta+2\sigma+12)\)-thin.
    
    \medskip
    \underline{The case $a, b\geq 2, c = 0$:}
        By Proposition~\ref{prop:long_paths_close_to_graph} we know that $\gamma_a$ and $\gamma_b$ are at a Hausdorff distance of $F$ from \(\gamma'_a\) and \(\gamma'_b\), where $\gamma'_a= \operatorname{Path}_\X(y, y_1, \ldots, y_m, x)$ and $\gamma'_b=(z, y_1^\prime, \ldots, y_n^\prime, x)$ with \(y_1 = y'_1 = y_v\) and \(y_m = y'_n = y_u\). 
        Let \(\Delta'\) be the triangle with sides \(\gamma'_a, \gamma'_b\), and \(\gamma_c\).
        Outside of \(X_u\), the sides of \(\Delta'\) are concatenations of \((1,K)\)-quasigeodesics in \(\abs\Y\) between \(y_u\) and \(y_v\) with \([y_u,x] \subseteq X_v\).
        Since \(X_u\) is \(\delta\)-hyperbolic and \(\abs\Y\) is \(\delta_\Gamma\)-hyperbolic, then, \(\Delta'\) is \((\delta + M)\)-thin.
        Hence \(\Delta\) is \((\delta + M + 2F)\)-thin.

    \medskip
    \underline{The case $a, b, c = 1$:}   
        Necessarily, the vertices \(u,v,\) and \(w\) are distinct and adjacent.
        Applying Lemma~\ref{lem:geodesics_in_different_directions_one_close_to_base_point} we find points \(p_x \in X_u, p_y \in X_v, p_z \in X_w\) on \(\Delta\) with 
        \begin{equation}
        \label{eq:points_close_to_basepoints}
            \max\{\dist_{X_u}(p_x, y_u), \dist_{X_v}(p_y,y_v), \dist_{X_w}(p_z,y_w)\} \leq C + 5 \delta.
        \end{equation}
        There are two cases, after possibly relabelling: either \(p_x, p_z \in \gamma_a\) and \(p_y \in \gamma_b\),  or otherwise \(p_x \in \gamma_a, p_y \in \gamma_b, p_z \in \gamma_c\).
        In any case, write \(\X_1, \X_2,\) and \(\X_3\) for the subgraphs of spaces with underlying graphs on \(\{u,w\}, \{u,v\},\) and \(\{v,w\}\) respectively.
        As in Remark~\ref{rem:subgraph_of_spaces}, we identify each \(\abs{\X_i}\) with its natural image in \(\X\).
        By the previous cases, each \(\abs{\X_i}\) is \(\delta_0\)-hyperbolic, where \(\delta_0 = 33\delta + 10\).
        
        Let us treat the first case.
        We consider the geodesic polygons \(P_1 \subseteq \abs{\mathbf{X}_1}\) with vertices \(\{p_x,y_u,y_w,p_z\}\), \(P_2 \subseteq \abs{\mathbf{X}_2}\) with vertices \(\{x, p_x, y_u, y_v, p_y\}\), and \(P_3 \subseteq \abs{\mathbf{X}_3}\) with vertices \(\{y, z, p_z, y_w, y_v, p_y\}\).
        For any \(i = 1,2,3\), the space \(\abs{\X_i}\) is \(\delta_0\)-hyperbolic, so each side of \(P_i\) is contained in a \(4\delta_0\) neighbourhood of the others, as the polygons have at most \(6\) sides.
        The sides of \(P_i\) are either segments of distinct sides of \(\Delta\), or are otherwise at most \(2C+10\delta+1\) from two sides of \(\Delta\) by (\ref{eq:points_close_to_basepoints}).
        Moreover, \(\Delta\) is covered by the sides of \(P_1, P_2,\) and \(P_3\).
        Hence \(\Delta\) is \(\delta'\)-thin, where \(\delta'\) is the maximum of \(3C+15\delta+2\) and \(4\delta_0 + 2C+10\delta+1\).

        In the other case, we similarly divide \(\Delta\) using the pentagons \(P_1 \subseteq \abs{\X_1}\) with vertices \(\{p_x,y_u,y_w,p_z,z\}\), \(P_2 \subseteq \abs{\X_2}\) with vertices \(\{x, p_x, y_u, y_v, p_y\}\), and \(P_3 \subseteq \abs{\X_3}\) with vertices \(\{y, p_y, y_v, y_w, p_z\}\).
        Each of these is contained in \(\abs{\X_i}\) for some \(i = 1,2,3\) and is therefore \(3\delta_0\)-thin.
        By the same considerations as above, \(\Delta\) is \(\delta'\)-thin.

    \medskip
    \underline{The case $a\geq 2, b, c = 1$:}
        First, suppose that $\gamma_b, \gamma_c$ meet the same edge cylinder, so that \(\gamma_a\) is a geodesic with both endpoints in a single vertex space, say \(X_v\).
        In this case, \(\gamma_a \subseteq B_\X(X_v;\sigma)\) by Proposition~\ref{prop:vertex_spaces_quasiconvex}.
        This effectively reduces this case to that of \(a = 0, b, c = 1\) as above.
        Hence we may suppose that $\gamma_b$ and  $\gamma_c$ do not intersect the same edge cylinder.
        Necessarily \(v\) is adjacent to both \(u\) and \(w\).

        Let \(S_a\) be the string \((x,y_1, \dots, y_n,z)\), where \(y_1, \dots, y_n\) are the basepoints of the vertex spaces \(\gamma_a\) meets.
        By Proposition~\ref{prop:almost_geodesics_give_graph_quasigeodesics}, the path \(\gamma = \operatorname{Path}_\X(y_1, \dots, y_n)\) is \((1,K)\)-quasi-geodesic and Proposition~\ref{prop:long_paths_close_to_graph} gives \(\dHaus_\X(\gamma_a,\operatorname{Path}_\X(S_a)) \leq F\).
        Since \(\dist_\Gamma(u,w) \leq 2\), we have \(\ell(\gamma) \leq K + 2\).
        Writing \(\gamma'_a = \operatorname{Path}_\X(x,y_u,y_v,y)\), then, we have \(\dHaus_\X(\gamma_a,\gamma'_a) \leq F+K+2\).
        Let \(\Delta'\) be the triangle with sides \(\gamma_a', \gamma_b\), and \(\gamma_c\).

        By Lemma~\ref{lem:geodesics_in_different_directions_one_close_to_base_point}, there is a point \(p \in X_v\) on either \(\gamma_b\) or \(\gamma_c\) with \(\dist_{X_v}(p,y_v) \leq C+5\delta\); suppose it is the former without loss of generality.
        Let \(\X_1\) and \(\X_2\) be the subgraphs of spaces on vertices \(\{u,v\}\) and \(\{v,w\}\) respectively, and identify each \(\abs{\X_i}\) with its natural image in \(\abs\X\) as in Remark~\ref{rem:subgraph_of_spaces}.
        Consider the geodesic rectangle \(P_1 \subseteq \abs{\X_1}\) with vertices \(\{x,y_u,y_v,p\}\) and the geodesic pentagon \(P_2 \subseteq \abs{\X_2}\) with vertices \(\{y, z, y_w, y_v, p\}\).
        As in the case \(a,b,c = 1\), each side is contained in a \(3\delta_0\)-neighbourhood of the others, and the sides either belong to distinct sides of \(\Delta'\) or a distance of at most \(C+5\delta+1\) from two sides of \(\Delta'\).
        It follows that \(\Delta'\) is \(\delta'\)-thin, where \(\delta' = 3\delta_0 + C+ 5\delta + 1\), and hence that \(\Delta\) is \((\delta'+F+K+2)\)-thin.

    \medskip
    \underline{The case $a, b\geq 2, c = 1$:}
        By Proposition~\ref{prop:long_paths_close_to_graph} the sides $\gamma_a$ and $\gamma_b$ are $F$-close to paths of strings $S_a=(y, y_v, y_1, \ldots, y_m, y_u, x)$ and $S_b=(z, y_w, y_1^\prime, \ldots, y_n^\prime, y_u, x)$ respectively.
        By Proposition~\ref{prop:almost_geodesics_give_graph_quasigeodesics}, the segments of these paths in \(\abs\Y\) form \((1,K)\)-quasi-geodesics.
        Write \(S'_b = (z,y_w,y_v,y_1, \dots, y_m,y_u,x)\) and observe that \(\operatorname{Path}_\X(S_b)\) and \(\operatorname{Path}_\X(S'_b)\) are a Hausdorff distance of at most \(M' = M(1,K+1,\delta_\Gamma)\) apart, since they differ only in the middle segments, and \(\abs\Y\) is \(\delta_\Gamma\)-hyperbolic.
        We will write \(\gamma'_a = \operatorname{Path}_\X(S_a)\) and \(\gamma'_b = \operatorname{Path}_\X(S'_b)\).
    
        Let \(e \in E(\Gamma)\) be the edge joining \(v\) and \(w\), and write \(\Delta'\) for the triangle with sides \(\gamma'_a, \gamma'_b\), and \(\gamma_c\).
        Outside of \(H = X_v \cup Z_e \cup X_w\), the sides of the triangle \(\Delta'\) coincide.
        The intersection of \(\Delta'\) with \(H\) is a triangle whose sides consist of two geodesics and one \((1,1)\)-quasigeodesic.
        As we showed \(H\) is \((33\delta+10)\)-hyperbolic in the case \(a,b=1, c= 0\), this triangle is \((33\delta+M''+10)\)-thin, where \(M'' = M(1,1,33\delta+10) \geq 0\) is obtained by the Morse lemma.
        Hence \(\Delta\) is \((33\delta+M''+M+2F+10)\)-thin.
    
    \medskip
    \underline{The case $a, b, c \geq 2$:}
        By Proposition~\ref{prop:long_paths_close_to_graph} the sides of \(\Delta\) are $F$-close to the sides of \(\Delta_\Gamma\), concatenated with \([x,y_u], [y,y_v],\) and \([z,y_w]\) as appropriate.
        This latter triangle is \((\delta_\Gamma+2M)\)-thin, as \(\Delta_\Gamma\) is \((\delta_\Gamma+2M)\)-thin and its sides coincide away from \(\Delta_\Gamma\).
        Hence \(\Delta\) is \((\delta_\Gamma+2M + 2F)\)-thin.
\end{proof}

\begin{remark}
    Some cases in the proof of the above theorem may be bypassed if one applies the `guessing geodesics' criterion for hyperbolicity (see, for example, \cite[Lemma 3.5]{guessing_geods_ursula}).
    However, checking the coherence condition holds among the chosen paths is tantamount to checking many of the same things as is done above.
    With this in mind, we present the potentially longer but more direct proof above.
\end{remark}

\bibliographystyle{abbrv}
\bibliography{references}

\end{document}